\documentclass[reqno,11pt]{amsart}
\usepackage{amsmath,amssymb,latexsym,soul,cite,mathrsfs,accents}
\usepackage{color,enumitem,graphicx}
\usepackage[colorlinks=true,urlcolor=blue,
citecolor=red,linkcolor=blue,linktocpage,pdfpagelabels,
bookmarksnumbered,bookmarksopen]{hyperref}
\usepackage[english]{babel}
\usepackage[left=2.6cm,right=2.6cm,top=2.9cm,bottom=2.9cm]{geometry}
\usepackage{tensor}
\newtheorem{theorem}{Theorem}
\newtheorem{lemma}[theorem]{Lemma}
\newtheorem{corollary}[theorem]{Corollary}
\newtheorem{proposition}[theorem]{Proposition}
\newtheorem{remark}[theorem]{Remark}
\newtheorem{definition}[theorem]{Definition}

\newtheorem{theoremletter}{Theorem}
\newtheorem{propositionletter}{Proposition}

\newtheorem{lemmaletter}{Lemma}

\newenvironment{acknowledgement}{\noindent\textbf{Acknowledgments.}}{}

\newtheorem*{innerthm}{\innerthmname}
\newcommand{\innerthmname}{}
\newcommand{\quotes}[1]{``#1''}
\newenvironment{statement}[1]
{\renewcommand{\innerthmname}{#1}\innerthm}
{\endinnerthm}
\theoremstyle{definition}

\makeatletter
\def\namedlabel#1#2{\begingroup
	#2%
	\def\@currentlabel{#2}%
	\phantomsection\label{#1}\endgroup
}

\def\XXint#1#2#3{{\setbox0=\hbox{$#1{#2#3}{\int}$ }
		\vcenter{\hbox{$#2#3$ }}\kern-.6\wd0}}

\newcommand*\owedge{\mathpalette\@owedge\relax}
\newcommand*\@owedge[1]{%
	\mathbin{%
		\ooalign{%
			$#1\m@th\bigcirc$\cr
			\hidewidth$#1\m@th\wedge$\hidewidth\cr
		}%
	}%
}
\makeatother

\newcommand{\ud}{\mathrm{d}}

\title[Complete metrics with constant fractional higher order $Q$-curvature]{Complete metrics with constant fractional higher order $Q$-curvature on the punctured sphere}
\thanks{This work was partially supported by S\~ao Paulo Research Foundation (FAPESP) \#2020/07566-3 and \#2021/15139-0 and Natural Sciences and Engineering Research Council of Canada (NSERC)}

\author[J.H. Andrade]{Jo\~{a}o Henrique Andrade}
\author[J. Wei]{Juncheng Wei}
\author[Z. Ye]{Zikai Ye}

\address[J.H. Andrade]{
	Department of Mathematics,
	University of British Columbia
	\newline\indent 
	V6T 1Z2, Vancouver-BC, Canada
	\newline\indent
	and
	\newline\indent 
	Institute of Mathematics and Statistics,
	University of S\~ao Paulo
	\newline\indent 
	05508-090, S\~ao Paulo-SP, Brazil
}
\email{\href{mailto:andradejh@math.ubc.ca}{andradejh@math.ubc.ca}}
\email{\href{mailto:andradejh@ime.usp.br}{andradejh@ime.usp.br}}

\address[J. Wei]{
	Department of Mathematics,
	University of British Columbia
	\newline\indent 
	V6T 1Z2, Vancouver-BC, Canada}
\email{\href{mailto:jcwei@math.ubc.ca}{jcwei@math.ubc.ca}}

\address[Z. Ye]{
	Department of Mathematics,
	University of British Columbia
	\newline\indent 
	V6T 1Z2, Vancouver-BC, Canada}
\email{\href{mailto:yezikai@math.ubc.ca}{yezikai@math.ubc.ca}}


\begin{document}
	
	\begin{abstract}
		This manuscript is devoted to constructing complete metrics with constant higher fractional curvature on punctured spheres with finitely many isolated singularities. Analytically, this problem is reduced to constructing singular solutions for a conformally invariant integro-differential equation that generalizes the critical GJMS problem. Our proof follows the earlier construction in Ao {\it et al.} \cite{MR3694645}, based on a gluing method, which we briefly describe. Our main contribution is to provide a unified approach for fractional and higher order cases. This method relies on proving Fredholm properties for the linearized operator around a suitably chosen approximate solution. The main challenge in our approach is that the solutions to the related blow-up limit problem near isolated singularities need to be fully classified; hence we are not allowed to use a simplified ODE method. To overcome this issue, we approximate solutions near each isolated singularity by a family of half-bubble tower solutions. Then, we reduce our problem to solving an (infinite-dimensional) Toda-type system arising from the interaction between the bubble towers at each isolated singularity. Finally, we prove that this system's solvability is equivalent to the existence of a balanced configuration.    
	\end{abstract}
	
	\maketitle
	
	\begin{center}
		\footnotesize
		\tableofcontents
	\end{center}
	
	\section{Introduction}
	The problem of constructing complete metrics on punctured spheres with prescribed fractional higher order curvature is longstanding in differential geometry.
	In \cite{MR1190438}, Graham, Jenne, Mason, and Sparling constructed conformally 
	covariant differential operators $P_{2m}(g)$ on a given compact $n$-dimensional Riemannian manifold 
	$(M^n,g)$ for any  
	$m \in \mathbb{N}$ such that the leading order term of $P_{2m}(g)$ is 
	$(-\Delta_{g})^m$ with $n>2m$. 
	One can then construct the associated $Q$-curvature 
	of order $2m$ by $Q_{2m}(g) = P_{2m}(g) (1)$. 
	When $m=1$,
	one recovers the conformal Laplacian 
	\begin{equation*}
		P_2(g) = -\Delta_g + \frac{n-2}{4(n-1)} R_g \quad {\rm with} \quad Q_2(g) =\frac{n-2}{4(n-1)} R_g,
	\end{equation*}
	where $\Delta_g$ is the Laplace-Beltrami operator of $g$ and $R_g$ is 
	its scalar curvature. 
	We also refer to \cite[Appendix A]{arXiv:2302.05770} for the explicit
	formulae for $P_2(g)$, $P_4(g)$ and $P_6(g)$.
	Subsequently, Grahan and Zworski \cite{MR1965361} and Chang and 
	Gonz\'alez \cite{MR2737789} 
	extended these definitions in the case the background metric is the 
	round metric on the sphere to obtain (nonlocal) operators $P_{2\sigma}(g)$ of any 
	order $\sigma\in (0,\frac{n}{2})$ as well as its corresponding $Q$-curvature. Once again, the leading order part of $P_{2\sigma}(g)$ is $(-\Delta_{g})^\sigma$, 
	understood as the principal value of a singular integral operator.  
	
	Nevertheless, 
	the expressions for $P_{2\sigma}(g)$ and $Q_{2\sigma}(g)$ for a general $\sigma\in\mathbb R_+$ are 
	far more complicated.
	Namely, the fractional curvature $Q_{2\sigma}(g)$ is defined from the conformal fractional Laplacian $P_{2\sigma}(g)$ as $Q_{2\sigma}(g)=P_{2\sigma}(g)(1)$. It is a nonlocal version of the scalar curvature (corresponding to the local case $\sigma=1$). 
	The conformal higher order fractional Laplacian $P_{2\sigma}(g)$ is a (nonlocal) pseudo-differential operator of order $2 \sigma$, which can be    
	constructed from scattering theory on the conformal infinity $M^n$ of a conformally compact Einstein manifold $(X^{n+1}, g^{+})$ as a generalized Dirichlet-to-Neumann operator for the eigenvalue problem
	\begin{equation*}
		-\Delta_{g^{+}} U-\frac{(n+2\sigma)^2}{4} U=0 \quad \text { in } \quad X,
	\end{equation*}
	where $U\in \mathcal{C}^\infty(X)$ is the respective extension of $u\in \mathcal{C}^\infty(M)$
	This construction is a natural one from the point of view of the AdS/CFT correspondence in Physics, also known as Maldacena's duality \cite{MR1633016}.
	We refer the reader to \cite{MR1743597,MR1633012} for more details.
	
	In this manuscript, we are restricted to
	the $n$-dimensional sphere  $\mathbb{S}^n\subset \mathbb R^{n+1}$, where $n > 2\sigma$ and $\sigma\in(1,+\infty)$ equipped with the standard round metric $g_0$, which is given by the pullback of the usual Euclidean metric $\delta$ under the stereographic projection $\Pi:\mathbb{S}^n\setminus\{\mathrm{e}_1\}\rightarrow \mathbb R^n\setminus\{0\}$ with $\mathrm{e}_1=(1,0,\dots,0)\in\mathbb{S}^n$ denoting its north pole. 
	For any $k \in \mathbb{R}$ with $0\leqslant k\leqslant n$, we seek complete metrics on $\mathbb{S}^n \setminus\Lambda^k$ of the form $g = \mathrm{u}^{{4}/{(n-2\sigma)}} g_0$, where $\Lambda\subset\mathbb S^n$ is such that $\#\Sigma=N$.  
	In order to $g$ to be complete on $\mathbb{S}^n \setminus 
	\Lambda$, one has to impose $\liminf_{\ud(p,\Lambda)} \mathrm{u}(p) = +\infty$. 
	Also,
	we prescribe the resulting metric to have constant $Q_{2\sigma}$-curvature, which we normalize to 
	be 
	\begin{equation*}
		Q_{n,\sigma} = Q_{2\sigma}({g_0}) = \Gamma\left(\frac{n+2s}{2}\right)\Gamma\left(\frac{n-2s}{2}\right)^{-1},
	\end{equation*}
	where $\Gamma(z)=\int_{0}^{\infty}\tau^{z-1}e^{-\tau}\ud\tau$ is the standard Gamma function.  
	
	Let us now introduce some standard terminology.
	For any $\sigma\in(1,+\infty]$ with $n>2\sigma$ and $N\geqslant 2$, we respectively denote by
	\begin{align*}
		\mathcal{M}_{2\sigma,\Lambda}(g_0)=\left\{g\in [g_0] :  \mbox{$g$ is complete on $\mathbb S^n\setminus\Lambda$} \; \mbox{and} \;  Q_{2\sigma}(g)\equiv Q_{n,\sigma}   \right\}
	\end{align*}
	and
	\begin{align}\label{markedmodulispace}
		\mathcal{M}_{2\sigma,N}(g_0)=\left\{g\in [g_0] :  \mbox{$g$ is complete on $\mathbb S^n\setminus\Lambda$ 
			with $\#\Lambda=N$} \; \mbox{and} \;  Q_{2\sigma}(g)\equiv Q_{n,\sigma}   \right\}
	\end{align}
	the marked and unmarked moduli spaces of complete constant higher order fractional $Q$-curvature metrics with isolated singularities.
	We also denote by ${\rm sing}({g})=\Lambda$ its respective singular set. 
	
	In this fashion, our main theorem in this paper is the following 
	\begin{theorem}
		Let $\sigma\in(1,+\infty)$ with $n>2\sigma$. 
		For any configuration $\Lambda\subset \mathbb{S}^n$ such that $\#\Lambda=N$ with $N\geqslant 2$, there exists a metric $g \in \mathcal{M}_{2\sigma,N}(g_0)$ satisfying ${\rm sing}(g)=\Lambda$ and is unmarked nondegenerate. 
		For a generic set of $\Lambda=\{p_1, \ldots, p_N\}$, this solution is marked nondegenerate, and for such a metric $(p_1, \ldots, p_N,\varepsilon_1, \ldots, \varepsilon_N)\in\mathbb R^{N(n+1)}$
		constitute a full set of coordinates in $\mathcal{M}_{2\sigma,N}(g_0)$ near $g_0$.
		In particular, one has $\mathcal{M}_{2\sigma,N}(g_0)\neq\varnothing$.
	\end{theorem}
	
	Let us derive an analytical formulation for our main result.
	The family of higher order fractional curvatures transform nicely under a conformal change.
	Indeed, for any $\bar{g}\in[g]$, one has 
	\begin{equation*}
		Q_{2\sigma}({\bar{g}})=\frac{2}{n-2\sigma} u^{-\frac{n+2\sigma}{n-2\sigma}} P_{2\sigma}({g_0})u,
	\end{equation*}
	where $P_{2\sigma}({g_0}):\mathcal{C}^{\infty}(M)\rightarrow\mathcal{C}^{\infty}(M)$ is the fractional higher order GJMS operator on the sphere
	\begin{equation*}
		P_{2\sigma}({g_0}):=\Gamma\left(\sqrt{-\Delta_{g_0}+\frac{(n-1)^2}{4}}+2\sigma+\frac{1}{2}\right)\Gamma\left(\sqrt{-\Delta_{g_0}+\frac{(n-1)^2}{4}}-2\sigma+\frac{1}{2}\right)^{-1},
	\end{equation*}
	where $\Delta_{g_0}$ is the Laplace--Beltrami operator and $[g]=\{\bar{g}=u^{{4}/{(n-2\sigma)}} g: u \in \mathcal{C}_+^{\infty}(M)\}$ is the conformal class of $g$,  where $u\in\mathcal{C}_+^{\infty}(M)$ if and only if $u\in\mathcal{C}^{\infty}(M)$ and $u>0$.
	Furthermore, one has the transformation law
	\begin{equation*}
		P_{2\sigma}({g})\phi=\mathrm{u}^{-\frac{n+2\sigma}{n-2\sigma}}P_{2\sigma}({g_0})(\mathrm{u}\phi) \quad \mbox{for all} \quad \phi\in \mathcal{C}^{\infty}(\mathbb{S}^n \setminus 
		\Lambda),
	\end{equation*}
	which means that GJMS operators are conformally covariant.	
	Hence, finding conformal complete metrics $g = \mathrm{u}^{{4}/{(n-2\sigma)}} g_0$ with prescribed curvature $Q_{2\sigma}(g) = Q_{n,\sigma}$ on 
	$\mathbb{S}^n \setminus 
	\Lambda$ is equivalent to finding smooth positive solutions $\mathrm{u}\in\mathcal{C}^\infty(\mathbb{S}^n \setminus \Lambda)$ to the nonlocal higher order geometric PDE 
	\begin{equation}\tag{$\mathcal{Q}_{2\sigma,\Lambda,g_0}$}\label{ourequation}
		\begin{cases}
			P_{2\sigma}({g_0})\mathrm{u}=c_{n,\sigma}\mathrm{u}^{\frac{n+2\sigma}{n-2\sigma}} \quad \mbox{on} \quad \mathbb{S}^n \setminus \Lambda,\\
			\liminf_{\ud(p,\Lambda)\rightarrow0}\mathrm{u}(p)=+\infty,
		\end{cases}          
	\end{equation}
	where $c_{n,\sigma}>0$ is a normalizing constant and ${\rm sing}(\mathrm{u}):=\Lambda$ denotes the singular set.
	
	Next, it will be convenient to transfer the PDE \eqref{ourequation} to Euclidean 
	space, which we can do using the standard stereographic projection.
	In these coordinates, our conformal metric takes the 
	form $g = \mathrm{u}^{{4}/{(n-2\sigma)}}g_0 = (\mathrm{u}\cdot u_{\rm sph})^{{4}/{(n-2\sigma)}}
	\delta$. 
	Thus, $u\in \mathcal{C}^{\infty}(\mathbb R^n\setminus\Sigma)$ given by $u = \mathrm{u} \cdot u_{\rm sph}$ is a positive singular solution to  \eqref{ourintegralequation}.
	As a notational shorthand, we adopt the 
	convention that $\mathrm{u}$ refers to a conformal factor relating the metric $g$ to the 
	round metric, {\it i.e.} $g = \mathrm{u}^{{4}/{(n-2\sigma)}} g_0$, while $u$ refers to a conformal 
	factor relating the metric $g$ to the Euclidean metric, {\it i.e.} $g = u^{{4}/{(n-2\sigma)}}
	\delta$, with the two related by $u = \mathrm{u} \cdot u_{\rm sph}$.   
	Hence, we aim to construct positive singular solutions $u\in \mathcal{C}^{\infty}(\mathbb{R}^n\setminus\Sigma)$ to the following higher order fractional Yamabe equation with prescribed isolated singularities     
	\begin{equation}\label{ourintegralequation}\tag{$\mathcal{Q}_{2\sigma,\Sigma}$}
		\begin{cases}
			(-\Delta)^{\sigma}u=f_{\sigma}(u) \quad {\rm in} \quad \mathbb{R}^n\setminus\Sigma,\\
			u(x)=\mathcal{O}(|x|^{2\sigma-n}) \quad {\rm as} \quad |x|\rightarrow+\infty,
		\end{cases}
	\end{equation}
	where $\sigma\in(1,+\infty)$ with $n> 2\sigma$.
	The subset ${\rm sing}(u):=\Sigma\subset\mathbb R^n$ is called the singular set, which is assumed to be $\Sigma=\{x_1,\cdots,x_N\}$ for some $N\in \mathbb N$ and such that 
	\begin{equation*}
		\liminf_{\ud(x,\Sigma)\rightarrow 0}u(x)=+\infty.
	\end{equation*}
	We are interested in fast-decaying solutions; we assume the following condition $\lim_{|x|\rightarrow +\infty}u(x)=0$.    
	The integral operator on the right-hand side of \eqref{ourintegralequation} is the so-called higher order fractional Laplacian which is defined as 
	\begin{equation*}
		(-\Delta)^{\sigma}:=(-\Delta)^{s}\circ (-\Delta)^{m},
	\end{equation*}
	where $m:=[\sigma]$ and $s:=\sigma-[\sigma]$.
	
	Here $(-\Delta)^{m}=(-\Delta)\circ\cdots\circ (-\Delta)$ denotes the poly-Laplacian and $(-\Delta)^{s}$ denotes the fractional Laplacian defined as
	\begin{equation*}
		(-\Delta)^{s}u(x):={\rm p.v.}\int_{\mathbb R^n}\mathcal{K}_{s}(x-y)[u(x)-u(y)]\ud y,
	\end{equation*}
	where $\mathcal{K}_{s}:\mathbb{R}^n\times\mathbb{R}^n\rightarrow\mathbb R$ is a singular potential given by
	\begin{equation}\label{singularpotential}
		\mathcal{K}_{s}(x-y):=\kappa_{n,s}|x-y|^{-(n+2s)}
	\end{equation}
	with
	\begin{equation*}
		\kappa_{n, s}=\pi^{-\frac{n}{2}} 2^{2 s} s{\Gamma\left(\frac{n}{2}+s\right)}{\Gamma(1-s)^{-1}}.
	\end{equation*}       
	The nonlinearity $f_{\sigma}:\mathbb R\rightarrow\mathbb R$ in the left-hand side of \eqref{ourintegralequation} is given by
	\begin{equation*}
		f_{\sigma}(\xi)=c_{n,\sigma}|\xi|^{\frac{n+2\sigma}{n-2\sigma}},
	\end{equation*}
	where 
	\begin{equation*}
		c_{n,\sigma}:=2^{2\sigma}{\Gamma\left(\frac{n+2\sigma}{4}\right)^2}{\Gamma\left(\frac{n-2\sigma}{4}\right)^{-2}}.
	\end{equation*}
	is a normalizing constant.
	We remark that this nonlinearity has critical growth in the sense of the Sobolev embedding $H^{\sigma}(\mathbb R^n)\hookrightarrow L^{2^*_{\sigma}}(\mathbb R^n)$, where $2^*_{\sigma}:=\frac{2n}{n-2\sigma}$.
	
	Our main result in this manuscript extends this result for the remaining cases.
	We are based on the unified approach given by Ao {\it et al.}\cite{MR3694645} and Jin and Xiong \cite{arxiv:1901.01678} to prove the existence of solutions to our integral equation, which can be stated as follows
	\begin{theorem}\label{maintheorem}
		Let $\sigma\in(1,+\infty]$ with $n>2\sigma$.
		For any configuration $\Sigma=\{x_1,\dots,x_N\}$ with $N\geqslant 2$, one can find a smooth positive singular solution to \eqref{ourintegralequation} such that ${\rm sing}(u)=\Sigma$.
	\end{theorem}
	
	In \cite{arxiv:1901.01678}, the authors use a dual representation and maximization methods to study the existence of Emden--Fowler solution on the range $\sigma\in(0,\frac{n}{2})$. 
	Although this representation is enough to prove the existence of blow-up limit solutions by direct maximization methods, it is unsuitable for our gluing technique.
	This paper follows the approach in \cite{MR3694645} with the dual equation \eqref{ourintegralequationdual}.
	Nevertheless, we need to give an alternative proof to describe the local behavior near each isolated singularity in terms of the bubble tower solution (see Lemma~\ref{lm:estimatesdelaunay}).
	This alternative proof is the main feature of this paper since it enables us to extend the techniques in \cite{MR3694645} for integral equations that cannot be realized as the dual formulation of a differential equation, which is undoubted of independent interest.
	
	Instead, we notice that \eqref{ourintegralequation} has a dual counterpart, which is given by   
	
	\begin{equation}\label{ourintegralequationdual}\tag{${\mathcal{Q}}^{\prime}_{2\sigma,\Sigma}$}
		\begin{cases}
			u=(-\Delta)^{-\sigma}(f_\sigma\circ u) \quad {\rm in} \quad \mathbb{R}^n\setminus\Sigma,\\
			u(x)=\mathcal{O}(|x|^{2\sigma-n}) \quad {\rm as} \quad |x|\rightarrow+\infty,
		\end{cases}  
	\end{equation}   
	where $(-\Delta)^{-\sigma}$ denotes the inverse operator of the standard higher 
	order fractional Laplacian, namely
	\begin{equation*}
		(-\Delta)^{-\sigma}f_{\sigma}(u(x)):=(\mathcal{R}_{\sigma}\ast f_{\sigma}(u))(x)={\rm p.v.}\int_{\mathbb R^n}\mathcal{R}_{\sigma}(x-y) f_{\sigma}(u(y))\ud y,
	\end{equation*}
	where $\mathcal{R}_{\sigma}:\mathbb{R}^n\times\mathbb{R}^n\rightarrow\mathbb R$ is the Riesz potential given by
	\begin{equation}\label{rieszpotential}
		\mathcal{R}_{\sigma}(x-y):=C_{n,\sigma}|x-y|^{-(n-2\sigma)}
	\end{equation}
	with $C_{n,\sigma}>0$ a normalizing constant.
	Our starting point in this paper will be to prove that 
	\eqref{ourintegralequation} and \eqref{ourintegralequationdual} are equivalent (see Lemma~\ref{lm:dualrepresentation}).
	
	When $\sigma\in \mathbb N$ is integer, that is $\sigma=m$, Eq. \eqref{ourintegralequation} becomes the poly-harmonic equation
	\begin{equation}\label{ourlocalequationinteger}\tag{$\mathcal{P}_{2m,\Sigma}$}
		\begin{cases}
			(-\Delta)^{m}u=f_{m}(u) \quad {\rm in} \quad \mathbb{R}^n\setminus\Sigma,\\
			u(x)=\mathcal{O}(|x|^{2\sigma-n}) \quad {\rm as} \quad |x|\rightarrow+\infty.
		\end{cases}        
	\end{equation}
	The most natural case of \eqref{ourlocalequationinteger} is when $m=1$. 
	In this situation, this equation becomes the classical Lane--Emden equation.
	On this subject, Mazzeo and Pacard \cite[Theorem~2]{MR1425579} based on a gluing technique via ODE theory to prove an existence theorem.     
	Furthermore, when $\sigma\in(0,1)$, we arrive at
	\begin{equation}\label{ourequationfractional}\tag{$\mathcal{F}_{2s,\Sigma}$}
		\begin{cases}
			(-\Delta)^{s}u=f_{s}(u) \quad {\rm in} \quad \mathbb{R}^n\setminus\Sigma,\\
			u(x)=\mathcal{O}(|x|^{2\sigma-n}) \quad {\rm as} \quad |x|\rightarrow+\infty.
		\end{cases}         
	\end{equation}
	Recently, Ao {\it et al.} \cite[Theorem~1.1]{MR4104278} extended the earlier existence results for this case. 
	Their construction is substantially different from the previous one and relies on the concept of a bubble tower (or Half-Dancer solutions). 
	We can summarize these results in the following statement
	\begin{theoremletter}
		Let $\sigma\in(0,1]$ with $n>2\sigma$.
		For any configuration $\Sigma=\{x_1,\dots,x_N\}$ with $N\geqslant 2$, one can find a smooth positive singular solution to \eqref{ourintegralequation} such that ${\rm sing}(u)=\Sigma$.
	\end{theoremletter}
	
	Let us briefly explain our strategy for the proof.
	We are based
	Schoen's \cite{MR929283} tactic, which consists in finding an explicit infinite set of functions that span an approximate nullspace, such that the linearized nonlinear nonlocal operator around this infinite-dimensional family of solutions is invertible on its orthogonal complement. 
	He first solves the equation on the complement. Then he provides a set of balancing conditions to ensure that the solution to this restricted problem is a solution to the original problem.
	This method was recently extended for fractional operators \cite{MR4030366}. 
	
	This technique set differs substantially from the one in \cite{MR1425579}.
	In their construction, the authors obtain an one-parameter family of solutions which blows up quickly enough near the singular set. 
	These solutions are different in spirit from the ones in the nonlocal case. 
	Since blow-limit Delaunay solutions for the scalar curvature problem are classified to depend only on two parameters in \cite{MR982351}.
	Then, by linearizing the problem around these solutions families, the resulting linear operator is proved to be surjective on some reasonable space of functions, at least when the neck size parameter is sufficiently small. A standard iteration argument may be used to obtain an exact solution to \eqref{ourintegralequation} with a suitable blow-up rate. 
	This strategy derives from its connections with the earlier constructions of the CMC with Delaunay-type ends \cite{MR1807955}.
	Compared with the fractional case $\sigma\in(0,1)$, the main difference in our strategy is the proof of the refined asymptotics near half-bubble towers solutions, which holds in a more general setting.
	
	Using this approach, we can perturb each bubble within the tower separately and construct a bubble tower at each singularity, and as an appropriate approximate solution to \eqref{ourintegralequation}. 
	However, it is essential to note that the linearization of this approximate solution is not injective, as there is an infinite-dimensional kernel. 
	As a result, an infinite-dimensional Lyapunov-Schmidt reduction procedure is employed. 
	This approach is similar to Kapouleas' CMC construction \cite{MR1100207}, which Malchiodi adapted in \cite{MR2522830} to produce new entire solutions for a semilinear equation with a subcritical exponent that differ from the well-known spike solutions since they decay to zero when moving away from three half lines and do not tend to zero at infinity. 
	For this, he constructed a half-Dancer solution along each half-line.
	Whence, to solve the original problem from the perturbed one, an infinite-dimensional system of Toda-type needs to be solved, which arises from studying the interactions between the different bubbles in the tower. 
	The most robust interactions occur in the zero-mode level and turn into some compatibility conditions (see Definition~\ref{def:balancedparameters}).
	In this fashion, a configuration satisfying such conditions is called balanced, related to well-known balancing properties enjoyed by the sum of the Pohozaev invariants.
	Nonetheless, the remaining interactions can be made small and are dealt with through a fixed-point argument.
	
	These compatibility conditions do not restrict the location of the singularity points but only affect the Delaunay parameter (neck size) at each end. 
	We also note that due to the heavy influence of the underlying geometry, the first compatibility condition is similar to the ones found in \cite{MR1712628} for the local case $\sigma=1$. 
	However, the rest of the configuration depends on the Toda-type system. 
	In the local setting, a similar procedure to remove the resonances of the linearized problem was considered in \cite{MR2812575}, but the Toda-type system was finite-dimensional in their case.
	
	On the technical level, our strategy is to employ the gluing method and Lyapunov--Schmidt reduction method. 
	First, we find a suitable approximate solution: a perturbation of the summation of half-Delaunay solutions with a singularity at each puncture. 
	Then, use the reduction method to find a perturbed solution that satisfies the associated linearized problem with the right-hand side given by some Lagrangian multiplier containing the approximate kernels of the linearized operator.
	This family of kernels spans an infinite-dimensional set called the approximate null space.
	The last step is to determine the infinitely-dimensional free parameter set such that all the coefficients of the projection on approximate null space vanish. 
	This problem is reduced to the solvability of some infinite-dimensional Toda system around each singular point.
	A fundamental property in the proof is to have a sufficiently good approximate solution (a half-Dancer) so that all the estimates are exponentially decreasing in terms of the bubble tower parameter. 
	A fixed point argument in suitable weighted sequence spaces then solves the problem of adjusting the parameters to have all equal to zero.
	
	We remark that instead of relying on the well-known extension problem for the fractional Laplacian \cite{MR2354493}, we are inspired by the approach in the approach given by Delatorre {\it et al.} \cite{MR3694655} to rewrite the fractional Laplacian in radial coordinates in terms of a new integro-differential operator in logarithmic cylindrical coordinates.
	In our case, such an extension does not exist in general.
	We emphasize that our proof is written solely in the dual formulation, and it can be extended to general integral equations not arising as the dual representation of a differential equation.
	
	In light of the seminal result of Mazzeo, Pollack, and Uhlenbeck \cite{MR1356375} (see also \cite{MR1371233}), it is natural to wonder if the marked moduli space in \eqref{markedmodulispace} can be furnished with more structure.
	It is believed that the result below should hold
	
	\begin{statement}{Conjecture~1}
		Let $\sigma\in(0,+\infty]$ with $n>2\sigma$.
		For any singular set $\Lambda\subset \mathbb S^n$ such that $\#\Lambda=N$ with $N\geqslant 2$, the marked moduli space of complete constant higher order fractional singular $Q$-curvature metrics on the punctured round sphere $\mathcal{M}_{2\sigma,N}(g_0)$ is an analytic manifold with formal dimension equal the number of isolated singularities, that is, ${\rm dim}(\mathcal{M}_{2\sigma,N}(g_0))=N$.
	\end{statement}
	
	Another possible development is to study the case in which a singular set is a disjoint union of smooth submanifolds with possible distinct positive Hausdorff dimensions.
	In this situation, it would be interesting to prove that the moduli space defined in \eqref{markedmodulispace} is still non-empty and, in strong contrast with the case of isolated singularities, is infinite-dimensional; this will be the topic of a forthcoming paper.
	
	Let us explain this case in more detail.
	It is well-known that the character of the analysis required to prove the existence of solutions when $R(g)<0$, which dates back to the work of Loewner and Nirenberg \cite{MR0358078} (see also \cite{MR932852,MR1266103}), is fundamentally different than in the positive scalar curvature case.
	Therefore, most of the literature is concentrated on the positive scalar curvature case  $R(g)>0$.
	In this setting, it is natural to have a solution one needs to impose some necessary conditions on the dimension.
	More challenging it would be to construct solutions to \eqref{ourintegralequation} with uncountable isolated singularities, for instance, in the lattice $\Sigma=\mathbb{Z}^n$.	
	The existence of weak solutions with larger dimension singular set for the singular Yamabe equation has been studied by Mazzeo and Smale \cite{MR1139641} and by  
	Mazzeo and Pacard \cite{MR1425579} for the scalar curvature case. 
	As well as by Hyder and Sire \cite{arXiv:1911.11891} for the (fourth order) $Q$-curvature metrics, and by Ao {\it et al.} \cite{MR4030366} for the (fractional order) $Q$-curvature metrics, based in the construction of entire solutions from \cite{MR3858832}.  
	
	More generally, such solutions may be constructed on an arbitrary compact manifold $(M^n,g)$ of nonnegative scalar curvature $R(g)\geqslant0$ whenever the singular set is a finite disjoint union of submanifolds with positive bounded Hausdorff dimension, which we describe as follows.
	Given $\sigma\in(0,+\infty]$ with $n>2\sigma$ and $N\geqslant 2$, we let 
	$\Lambda\subset \mathbb S^n$ be a finite disjoint union of submanifolds $\Lambda=\Lambda^0\cup \Lambda^+$, where $\Lambda_+=\cup_{\ell=1}^d\Lambda_\ell^{k_\ell}$ with $k_\ell:={\rm dim}_{\mathcal H}(\Lambda_\ell)$ denoting its Hausdorff dimension.
	Furthermore, we denote by
	\begin{align*}
		\mathcal{M}^k_{2\sigma,\Lambda}(g_0)=\left\{g\in [g_0] :  \mbox{$g$ is complete on $\mathbb S^n\setminus\Lambda^k$} \; \mbox{and} \;  Q_{2\sigma}(g)\equiv Q_{n,\sigma}   \right\}
	\end{align*}
	the moduli space of complete constant higher order fractional $Q$-curvature metrics with higher dimensional singularities.
	Notice that we simply denote $\mathcal{M}^0_{2\sigma,\Lambda}(g_0)=\mathcal{M}_{2\sigma,\Lambda}(g_0)$.
	
	To summarize this discussion, we have the following statement
	\begin{theoremletter}
		Let $\sigma\in(0,+\infty]$ with $n>2\sigma$.
		Assume that $\Lambda=\Lambda^0\cup \Lambda^+$ is a finite disjoint union of submanifolds satisfying $\Lambda^0=\varnothing$ and $\Lambda_+=\cup_{\ell=1}^d\Lambda_\ell^{k_\ell}$ with $0<k_\ell<\frac{n-2\sigma}{2}$.
		Then, there exists a metric $g \in \mathcal{M}^k_{2\sigma,\Lambda}(g_0)$  that ${\rm sing}(g)=\Lambda$. 
		In particular, one has $\mathcal{M}^k_{2\sigma,\Lambda}(g_0)\neq\varnothing$ and it is an infinite-dimensional analytic manifold.
	\end{theoremletter}
	
	With a mind on this statement, it would be natural to prove a similar result as below
	\begin{statement}{Conjecture~2}
		Let $\sigma\in(0,+\infty]$ with $n>2\sigma$.
		Assume that $\Lambda=\Lambda^0\cup \Lambda^+$ such that $\#\Lambda^0=N$ with $N\geqslant2$ and $\Lambda_+=\cup_{\ell=1}^d\Lambda_\ell^{k_\ell}$ with $0<k_\ell<\frac{n-2\sigma}{2}$.
		Then, there exists a metric $g \in \mathcal{M}^k_{2\sigma,\Lambda}(g_0)$ satisfying that ${\rm sing}(g)=\Lambda$. 
		In particular, one has $\mathcal{M}^k_{2\sigma,\Lambda}(g_0)\neq\varnothing$ and it is an infinite-dimensional analytic manifold.
	\end{statement}
	
	As usual, we need to prove the existence 
	of positive solutions to the GJMS equation on the conformally flat case $\mathbb{S}^n\setminus\mathbb S^k\simeq\mathbb{R}^n\setminus\mathbb R^k$ with fast-decay away from the singular set.
	Moreover, the dimension estimate above is sharp in the same sense of 
	González, Mazzeo, and Sire \cite{MR2927681}.
	Namely, if a complete metric blows up at a smooth $k$-dimensional submanifold and is polyhomogeneous, then $k\in\mathbb R_+$ must satisfy the restriction above.
	All the analysis for this type of equation comes from its conformal properties, which produce a geometric interpretation of scattering theory and conformally covariant operators. 
	Exploiting the conformal equivalence $\mathbb{R}^n\setminus \mathbb{R}^k \simeq \mathbb{S}^{n-k-1} \times \mathbb{H}^{k+1}$, where $\mathbb{R}_{+}^{n+1}$ is replaced by anti-de Sitter (AdS) space, but the arguments run in parallel.  
	In the same direction but with another flavor, we quote the multiplicity results in \cite{MR3504948,MR4251294,arXiv:2302.11073,arXiv:2306.00679,case-malchiodi}, which also exploit this conformal invariance and are based on a topological bifurcation technique; this is believed to be true in the much broader case of conformally variational invariants (c.f. \cite{MR4392224,MR3955546}). 
	
	One could extend this construction in a more geometric direction for not-round metrics.
	On this subject, we cite \cite{MR2639545,MR2010322} for non-flat gluing constructions for the constant curvature equation.
	Recently, in \cite{arXiv:2110.05234}, a similar gluing construction is used to prove existence results for fourth order constant $Q$-curvature nondegenerate metrics with suitable growth conditions on the Weyl tensor.

	We now describe the plan for the rest of the paper.
	In Section~\ref{sec:notations}, we establish some terminology that will be used throughout the paper.
	In Section~\ref{sec:dualreprensetation}, we prove the dual representation formula relating \eqref{ourintegralequation} with \eqref{ourintegralequationdual}.
	In Section~\ref{sec:delaunaytypesolutions}, we classify Delaunay-type solutions as bubble towers.     
	In Section~\ref{sec:approximatesolutions}, we provide balancing equations. 
	Next, we define balanced configuration parameters and admissible perturbation sequences. 
	We use this to define approximate solutions and prove some estimates for the linearized operator around this approximating family.
	In Section~\ref{sec:estimatesparameters}, we summarize some estimates involving the coefficients of the projection on the approximate null space.
	In Section~\ref{sec:gluingtechnique}, we reduce the proof of Theorem~\ref{maintheorem} to solving an infinite dimensional Toda system. 
	We prove that under admissibility conditions, this system can be solved.       
	In Appendix~\ref{sec:bubbetowerinteractionestimates}, we recall some estimates concerning the interaction between two spherical solutions with different centers and radii.
	
	\numberwithin{equation}{section} 
	\numberwithin{theorem}{section}
	
	\section{Notations}\label{sec:notations}
	We establish some notations and definitions that we will use frequently throughout the text for easy reference.   
	\begin{itemize}
		\item $m:=\lfloor  \sigma \rfloor$ is the integer part of $\sigma$, that is, be the greatest integer that does not exceed $\sigma$;
		\item $s:=\{\sigma\}$ is the fractional part of $\sigma$, that is, $s:=\sigma-\lfloor  \sigma \rfloor$;
		\item $0<\xi,\nu,\zeta_1\ll1$ are small constants;
		\item  $C>0$ is a universal constant that may vary from line to line and even in the same line. 
		\item $a_1 \lesssim a_2$ if $a_1 \leqslant C a_2$, $a_1 \gtrsim a_2$ if $a_1 \geqslant C a_2$, and $a_1 \simeq a_2$ if $a_1 \lesssim a_2$ and $a_1 \gtrsim a_2$.
		\item $u=\mathcal{O}(f)$ as $x\rightarrow x_0$ for $x_0\in\mathbb{R}\cup\{\pm\infty\}$, if $\limsup_{x\rightarrow x_0}(u/f)(x)<\infty$ is the Big-O notation;
		\item $u=\mathrm{o}(f)$ as $x\rightarrow x_0$ for $x_0\in\mathbb{R}\cup\{\pm\infty\}$, if $\lim_{x\rightarrow x_0}(u/f)(x)=0$ is the little-o notation;
		\item $u\simeq\widetilde{u}$, if $u=\mathcal{O}(\widetilde{u})$ and $\widetilde{u}=\mathcal{O}(u)$ as $x\rightarrow x_0$ for $x_0\in\mathbb{R}\cup\{\pm\infty\}$;
		\item $\mathcal{C}^{j,\alpha}(\mathbb R^n)$, where $j\in\mathbb N$ and $\alpha\in (0,1)$, is the H\"{o}lder space;  we simply denote $\mathcal{C}^{j}(\mathbb R^n)$ if $\alpha=0$.
		\item $W^{m,q}(\mathbb R^n)$ is the classical Sobolev space, where $m\in\mathbb N$ and $q\in[1,+\infty]$; when $m=0$ we simply denote $L^{q}(\mathbb R^n)$ when $q=2$, we simply denote $H^{\ell}(\mathbb R^n)$; 
		\item $\mathcal{C}^{2\sigma}(\mathbb R^n)={\mathcal{C}^{2m,2s}}(\mathbb R^n)$ is the classical H\"older space;
		\item $\gamma_\sigma=\frac{n-2\sigma}{2}$ is the Fowler rescaling exponent with $\gamma_\sigma^{\prime}=\frac{n+2\sigma}{2}$ its dual;
		\item $2^*_\sigma=\frac{2n}{n-2\sigma}$ is the critical Sobolev exponent with $2^{*\prime}_\sigma=\frac{2n}{n+2\sigma}$ its dual;
		\item $A_1, A_2>0, A_3<0$ are constant defined by \eqref{A1}, \eqref{A2}, and \eqref{A3}, respectively;
		\item $\mathcal{I}_\infty:=\{1,\dots,N\}\times \mathbb N\times \{0,\dots,n\}\simeq \ell^\infty(\mathbb R^{(n+1)N})$ is total index set;
		\item $\boldsymbol{p}:=(p_1,\dots, p_N)\in \mathbb S^{nN}$ with $\Lambda:=\{p_1,\dots, p_N\}\subset \mathbb S^n$;
		\item $\boldsymbol{x}:=(x_1,\dots, x_N)\in \mathbb R^{nN}$ with $\Sigma:=\{x_1,\dots, x_N\}\subset \mathbb R^n$;
		\item $\boldsymbol{L}=(L_1,\dots,L_N)\in\mathbb R^N_+$ is a vector of periods such that $|\boldsymbol{L}|\gg1$ is large enough arising from Proposition~\ref{prop:localbehavior}. Equivalently, $\boldsymbol{\varepsilon}=(\varepsilon_1,\dots,\varepsilon_N)\in\mathbb R^N_+$ is a vector of necksizes such that $0<|\boldsymbol{\varepsilon}|\ll1$ is small enough;
		\item $(\boldsymbol{x},\boldsymbol{L})\in\mathbb R^{(n+1)N}$ are the moduli space parameters and $\Upsilon_{\rm conf}:\mathbb R^{(n+1)N}\rightarrow \mathbb R^{(n+2)N}$ is the configuration map;        
		\item $\boldsymbol{q}=(q_1,\dots,q_N)\in\mathbb R^N_+$ is a vector of comparable periods such that $|\boldsymbol{q}|\gg1$ is also large enough and satisfy \eqref{balancing0}, $\boldsymbol{R}=(R^1,\dots,R^N)\in\mathbb R^N$ and $\boldsymbol{a}_0=(a_0^1,\dots,a_0^N)\in\mathbb R^{nN}$ are the deformation parameters; \item $(\boldsymbol{q},\boldsymbol{a}_0,\boldsymbol{R})\in\mathbb R^{(n+2)N}$ are the configurations parameters and $\Upsilon_{\rm per}:\mathbb R^{(n+2)N}\rightarrow \ell^\infty_{\tau}(\mathbb R^{(n+1)N})$ is the configuration map;
		\item $(\boldsymbol{q}^b,\boldsymbol{a}_0^b,\boldsymbol{R}^b)\in{\rm Bal}_\sigma(\Sigma)$ denotes a balanced configuration, that is, it satisfies \eqref{balancing1} and \eqref{balancing2}. 
		\item $(\boldsymbol{a}_j,\boldsymbol{\lambda}_j)\in\ell^\infty_\tau(\mathbb R^{(n+1)N})$ (or $(\boldsymbol{a}_j,\boldsymbol{r}_j)\in\ell^\infty_\tau(\mathbb R^{(n+1)N})$) are the perturbation sequences and $\Upsilon_{\rm per}:\mathbb R^{(n+2)N}\rightarrow \ell^\infty_{\tau}(\mathbb R^{(n+1)N})$ is the perturbation map;
		\item $(\boldsymbol{a}_j,\boldsymbol{\lambda}_j)\in {\rm Adm}_\sigma(\Sigma)$ denotes the admissible perturbation sequences, that is, it satisfies \ref{itm:A0} and \ref{itm:A1}; equivalently $\Upsilon_{\rm per}^{-1}(\boldsymbol{a}_j,\boldsymbol{\lambda}_j)\in {\rm Bal}_\sigma(\Sigma)$;
		\item $\bar{u}_{(\boldsymbol{x},\boldsymbol{L},\boldsymbol{a}_j,\boldsymbol{\lambda}_j)}\in \mathcal{C}_{*,\tau}(\mathbb R^n\setminus\Sigma)$ denotes a Delaunay solution with associated error denoted by  $\phi_{(\boldsymbol{x},\boldsymbol{L},\boldsymbol{a}_j,\boldsymbol{\lambda}_j)}\in \mathcal{C}_{*,\tau}(\mathbb R^n\setminus\Sigma)$ and $\Upsilon_{\rm sol}:\ell^\infty_{\tau}(\mathbb R^{(n+1)N}) \rightarrow \mathcal{C}_{*,\tau}(\mathbb R^n\setminus\Sigma)$ is the solution map;
		\item $\bar{u}_{(\boldsymbol{x},\boldsymbol{L},\boldsymbol{a}_j,\boldsymbol{\lambda}_j)}\in {\rm Apx}_\sigma(\Sigma)$ is an approximate solution, that is, $\Upsilon_{\rm sol}^{-1}(\bar{u}_{(\boldsymbol{x},\boldsymbol{L},\boldsymbol{a}_j,\boldsymbol{\lambda}_j)})\in {\rm Adm}_\sigma(\Sigma)$;
		\item $\{Z_{j \ell}^i(\boldsymbol{a}_j,\boldsymbol{\lambda}_j)\}_{(i,j,\ell)\in \mathcal{I}_\infty}\subset \mathcal{C}^{0}(\mathbb R^n\setminus\Sigma)$ is the associated family of approximating normalized cokernels;
		\item $\{c_{j \ell}^i(\boldsymbol{a}_j,\boldsymbol{\lambda}_j)\}_{(i,j,\ell)\in \mathcal{I}_{\infty}}\subset \mathcal{C}^{0}(\mathbb R^n\setminus\Sigma)$ is the associated family of coefficients of the projections on approximating normalized cokernels;
		\item $\{\beta_{j \ell}^i(\boldsymbol{a}_j,\boldsymbol{\lambda}_j)\}_{(i,j,\ell)\in \mathcal{I}_\infty}\subset \mathcal{C}^{0}(\mathbb R^n\setminus\Sigma)$ is the associated family of projections on approximating normalized cokernels.
	\end{itemize}
	
	\section{Dual representation formula}\label{sec:dualreprensetation}
	This section shows that our equation and its dual are correspondents. We are based on the removable singularity result from \cite[Theorem~1.1]{MR4420104}.
	We also refer to \cite[Proposition~4.1]{MR4438901} for the integer cases $\sigma\in\mathbb N$.
	In what follows, we consider the space
	\begin{equation*}
		L_s(\mathbb R^n):=\left\{u \in L^1_{\rm loc}(\mathbb R^n) : \int_{\mathbb R^n}\frac{|u(x)|}{1+|x|^{n+2s}}\ud x<+\infty\right\}
	\end{equation*}
	with $s\in(0,1)$. 
	
	We first introduce the notation of distributional solutions to \eqref{ourintegralequation}.
	
	\begin{definition}
		Let $\sigma\in\mathbb R_+$ and $n>2\sigma$.
		We say that a smooth solution $u\in \mathcal{C}^{2\sigma}(\mathbb R^n\setminus\Sigma)\cap  L^{1}_{\rm loc}(\mathbb{R}^n)$ to \eqref{ourintegralequation} is a solution in the distributional sense to \eqref{ourintegralequation} if the equality below holds
		\begin{equation}\label{distributionalsense}
			\int_{\mathbb{R}^n}u(-\Delta)^\sigma \varphi\ud x=\int_{\mathbb{R}^n} f_\sigma(u)\varphi  \ud x \quad {\rm in} \quad \mathbb{R}^n\setminus \Sigma
		\end{equation}
		for all $\varphi\in \mathcal{C}_{c}^{\infty}(\mathbb{R}^n\setminus \Sigma)$.
	\end{definition}
	
	\begin{remark}
		One can check that smooth solution to \eqref{ourintegralequation} are indeed distributional solutions.
	\end{remark}
	
	We need the following auxiliary result to prove the equivalence: a combination of \cite[Theorem 1.1 and Lemma 5.4]{MR4420104}.
	
	\begin{lemmaletter}\label{removablesingularity}
		Let $\sigma\in\mathbb R_+$ and $n>2\sigma$.
		If $u\in \mathcal{C}^{2\sigma}(\mathbb R^n\setminus\Sigma)\cap  L^{1}_{\rm loc}(\mathbb{R}^n)$ is a distributional solution to \eqref{ourintegralequation}, then $f_\sigma\circ u\in L^{1}_{\rm loc}(\mathbb{R}^n)$ and $u \in L^{1}_{\rm loc}(\mathbb{R}^n)$ is a distributional solution in $\mathbb{R}^n$, {\it that is}, the distributional equation \eqref{distributionalsense} holds.
		Moreover, one has 
		\begin{equation}\label{growthatinfty}
			\int_{\mathbb{R}^n}\frac{f_\sigma(u(x))}{1+|x|^{n-2\sigma}}\ud x<+\infty.
		\end{equation}
		Consequently, we obtain that
		$w\in \mathcal{C}^{\infty}(\mathbb R^n\setminus\Sigma)$ is defined as 
		\begin{equation}
			w(x):=\int_{\mathbb R^n}\mathcal{R}_{\sigma}(x-y) f_{\sigma}(u(y))\ud y 
		\end{equation}    
		is well-defined and belongs to $L_s(\mathbb{R}^n)$ for every $s>0$.
	\end{lemmaletter}
	
	Finally, we also recall the  Liouville theorem from \cite[Lemma~2.4]{MR3945763}.
	\begin{lemmaletter}\label{lm:liouville}
		Let $\sigma\in\mathbb R_+$ and $n>2\sigma$.
		Assume that  $w\in L_s(\mathbb{R}^n)$ for some $s\geqslant0$ and
		\begin{equation*}
			(-\Delta)^{\sigma}w=0 \quad {\rm in} \quad \mathbb{R}^n,
		\end{equation*}
		for some $\sigma\geqslant s$.
		Then, one has that $w$ is a polynomial of degree at most $\lfloor 2s\rfloor$.
	\end{lemmaletter}
	
	With the lemmas above, we have our main result in this section.
	
	\begin{proposition}\label{lm:dualrepresentation}
		Let $\sigma\in\mathbb R_+$ and $n>2\sigma$.
		It holds that \eqref{ourintegralequation} and \eqref{ourintegralequationdual} are equivalents.  
	\end{proposition}
	
	\begin{proof}[Proof of Proposition \ref{lm:dualrepresentation}]
		Let $u\in \mathcal{C}^{\infty}(\mathbb R^n\setminus\Sigma)$ be a positive singular fast-decaying solution to \eqref{ourintegralequation}.
		From \eqref{growthatinfty}, we have that $w\in L_s(\mathbb{R}^n)$ for every $s>0$, $s\neq 2\sigma$. Hence, if we define $\widehat{w}=u-w$, then $\widehat{w}\in L_s(\mathbb{R}^n)$ for all $s>0$ with $s\neq 2\sigma$. 
		In addition, since $(-\Delta)^\sigma \widehat{w}=0$ in $\mathbb{R}^n$, we conclude that $\widehat{w}$ is a polynomial of degree at most $2m$, thanks to the Liouville theorem in Lemma~\ref{lm:liouville}. 
		Recall that we are considering solutions satisfying $\lim_{|x|\rightarrow +\infty}u(x)=0$. 
		Consequently, $\widehat{w}\equiv 0$, and the dual representation holds.
	\end{proof}
	
	\section{Delaunay-type solutions}\label{sec:delaunaytypesolutions}
	This section is devoted to the construction of solutions for the case of a single isolated singularity, that is, $\Sigma=\{0\}$.
	We are inspired in \cite{arxiv:1901.01678}, which is an adaption of the earlier constructions in \cite{MR4104278,MR3694655,MR3896244} for the cases $\sigma\in(0,1)$ and $\sigma=1$.  
	
	\subsection{Integral Emden--Fowler coordinates}
	As a matter of fact, when $\Sigma=\{0\}$, Eq. \eqref{ourintegralequation} can be rewritten as 
	\begin{equation}\label{ourintegralequation1pt}\tag{$\mathcal{Q}_{2\sigma,\infty}$}
		\begin{cases}
			(-\Delta)^{\sigma}u=f_{\sigma}(u) \quad {\rm in} \quad \mathbb{R}^n\setminus\{0\},\\
			\displaystyle\lim_{|x|\rightarrow+\infty}u(x)=0,
		\end{cases}
	\end{equation}
	or into its dual form
	\begin{equation}\label{ourintegralequationdual1pt}\tag{$\mathcal{Q}^{\prime}_{2\sigma,\infty}$}
		\begin{cases}
			u=(-\Delta)^{-\sigma}(f_{\sigma}\circ u) \quad {\rm in} \quad \mathbb{R}^n\setminus\{0\}\\
			\displaystyle\lim_{|x|\rightarrow+\infty}u(x)=0.
		\end{cases}           
	\end{equation}
	It is straightforward to see from Proposition~\ref{lm:dualrepresentation} that \eqref{ourintegralequation1pt} are \eqref{ourintegralequationdual1pt} equivalents. 
	\begin{remark}
		For any $\sigma\in(1,+\infty]$ and $n>2\sigma$, there are two distinguished solutions to \eqref{ourintegralequationdual1pt}, which we describe as follows:
		\begin{itemize}
			\item[{\rm (a)}] The cylindrical solution
			\begin{equation}\label{cylindricalsolutions}
				u_{\rm cyl}(|x|)=a_{n,\sigma}|x|^{-\gamma_{\sigma}},
			\end{equation}
			which is singular at the origin.
			\item[{\rm (b)}] The standard spherical solution $($also known as \quotes{bubble} solution$)$ 
			\begin{equation}\label{sphericalsolutions}
				u_{\rm sph}(|x|)=\left(\frac{2}{1+\left|x\right|^2}\right)^{\gamma_{\sigma}},
			\end{equation}               
			which is non-singular at the origin.
		\end{itemize}
	\end{remark}
	
	We remark that all non-singular solutions to the blow-up limit problem were classified in \cite{MR2200258}, which are given by deformations of the standard bubble solution.
	This reflects the invariance of equation \eqref{ourintegralequationdual1pt} with respect to the entire Euclidean group with translations and dilations.
	\begin{propositionletter}
		Let $\sigma\in(1,+\infty]$ and $n>2\sigma$. If $u\in \mathcal{C}^{2\sigma}(\mathbb R^n)$ is a positive smooth non-singular solution to \eqref{ourintegralequationdual1pt}, then there exists $\lambda\in\mathbb R$ and $x_0\in \mathbb R^n$ such that 
		\begin{equation}
			u\equiv U_{\lambda,x_0},
		\end{equation}
		where 
		\begin{equation}\label{bubbles}
			U_{\lambda,x_0}(x)=\left(\frac{2\lambda}{\lambda^2+\left|x-x_0\right|^2}\right)^{\gamma_{\sigma}}
		\end{equation}
		for some $\lambda>0$ and $x_0\in \mathbb R^n$. 
		This family of solutions will be called spherical or bubble solutions.
	\end{propositionletter}
	
	The problem of classifying the complete set of positive smooth singular solutions to \eqref{ourintegralequationdual1pt} is much more challenging and only accomplished for a few cases. 
	On this subject, Chen, Li, and Ou proved that all solutions are radially symmetric with respect to the origin.
	In addition, Jin and Xiong \cite{arxiv:1901.01678} only proved the existence of such a solution by a direct maximization method.
	Furthermore, they also study the local asymptotic behavior of positive singular solutions to  
	\begin{equation}\label{ourintegralequation1ptlocal}\tag{$\mathcal{Q}_{2\sigma,R}$}
		(-\Delta)^{\sigma}u=f_{\sigma}(u) \quad {\rm in} \quad B_R^*,
	\end{equation}
	or into its dual form
	\begin{equation}\label{ourintegralequationdual1ptlocal}\tag{$\mathcal{Q}^{\prime}_{2\sigma,R}$}
		u=(-\Delta)^{-\sigma}(f_{\sigma}\circ u) \quad {\rm in} \quad B_R^*,
	\end{equation}
	where $B_R^*\subset \mathbb{R}^n\setminus\{0\}$  given by $B_R^*=B_R(0)\setminus\{0\}$ is the punctured ball of radius $R>0$.
	
	To study this class of equations, we define an important change of variables that turns \eqref{ourintegralequation1pt} into an integral one-dimensional problem.  
	
	\begin{definition}\label{def:fowlercoordinates}
		Let $\sigma\in(1,+\infty]$ and $n>2\sigma$
		We define the integral Emden--Fowler change of variables $($or cylindrical logarithm coordinates$)$ given by
		\begin{equation}\label{fowlercoordinates}
			\mathfrak{F}_{\sigma}:\mathcal{C}_c^{\infty}(B_R^*)\rightarrow \mathcal{C}_c^{\infty}(\mathcal{C}_{L}) \quad \mbox{given by} \quad
			\mathfrak{F}_{\sigma}(u)=e^{-\gamma_\sigma t}u(e^{-t},\theta),
		\end{equation}
		where $t=-\ln R$, $\theta=x/|x|$, $\mathcal{C}_{L}:=(L,+\infty)$ with $L=-\ln |x|$ and $\gamma_{\sigma}:=\frac{n-2\sigma}{2}$.
		The inverse of this isomorphism is
		\begin{equation}\label{fowlercoordinatesinverse}
			\left(\mathfrak{F}_{\sigma}\right)^{-1}:\mathcal{C}_c^{\infty}(\mathcal{C}_{L})\rightarrow \mathcal{C}_c^{\infty}(B_R^*) \quad \mbox{given by} \quad
			\left(\mathfrak{F}_{\sigma}\right)^{-1}(v)=|x|^{\gamma_\sigma}v(-\ln |x|,\theta).
		\end{equation}
		The quantity $\gamma_{\sigma}>0$ will be referred to as the Fowler rescaling exponent.
		From now on, we denote by $v(t,\theta):=\mathfrak{F}_{\sigma}(u(x))$ and $u(x):=(\mathfrak{F}_{\sigma})^{-1}(v(t,\theta))$, conversely. 
	\end{definition}
	
	Using this change of variable, Eq. \eqref{ourintegralequation1ptlocal} can be reformulated as
	the following one-dimensional problem
	\begin{align}\label{ourcylindricalequation1local}\tag{$\mathcal{O}_{2\sigma,L}$}
		\begin{cases}
			(-\Delta)_{\rm cyl}^{\sigma} v=f_\sigma(v) \quad {\rm in} \quad \mathcal{C}_{L},\\
			\displaystyle\lim_{t\rightarrow+\infty}v(t)=0.
		\end{cases}
	\end{align}      
	Here $(-\Delta)_{\rm cyl}^{\sigma}:\mathcal{C}^{2\sigma}(\mathcal{C}_{L})\rightarrow \mathcal{C}^0(\mathcal{C}_{L})$ is the operator higher-order operator given by 
	\begin{equation}\label{higherorderfractionallaplaciancyl}
		(-\Delta)^{\sigma}_{\rm cyl}:=(-\Delta)^{s}_{\rm cyl}\circ (-\Delta)^{m}_{\rm cyl},
	\end{equation}
	where $(-\Delta)^{m}_{\rm cyl}$ and   $(-\Delta)^{s}_{\rm cyl}$ denote the cylindrical poly-Laplacian and the fractional Laplacian, respectively, defined as
	\begin{equation*}
		(-\Delta)^m_{\rm cyl}:=\sum_{\ell=0}^{2m}\sum_{j=0}^{2m}K^{(\ell)}_{2m,j}\partial_t^{(j)}(-\Delta_\theta)^{\ell},
	\end{equation*}
	where $K^{(\ell)}_{2m,j}=K^{(\ell)}_{2m,j}(n)>0$ for $j,\ell\in\{0,\dots,2m\}$ are dimensional constants, and 
	\begin{equation*}
		(-\Delta)_{\rm cyl}^{s} v(t,\theta):=\int_{-L}^{+L} \widehat{\mathcal{K}}_{\sigma}(t-\tau,\theta-\varsigma)[v(t,\theta)-v(\tau,\varsigma)] \mathrm{d}\tau\ud\varsigma,
	\end{equation*}
	where ${\mathcal{K}}_{\sigma,{\rm cyl}}:\mathcal{C}_{L}\times\mathcal{C}_{L}\rightarrow\mathbb R$ is the kernel \eqref{singularpotential} written in Emden--Fowler coordinates. 
	As usual, the dual form of this equation is given by
	\begin{align}\label{ourcylindricalequation1dual}\tag{$\mathcal{O}_{2\sigma,L}^{\prime}$}
		\begin{cases}
			v=(-\Delta)_{\rm cyl}^{-\sigma}(f_\sigma \circ v) \quad {\rm in} \quad \mathcal{C}_{L},\\
			\displaystyle\lim_{t\rightarrow+\infty}v(t)=0.
		\end{cases}         
	\end{align}
	Here $(-\Delta)_{\rm cyl}^{-\sigma}$ is the integral linear operator defined by
	\begin{align*}
		(-\Delta)_{\rm cyl}^{-\sigma}(f_\sigma \circ v)(t,\theta):=(\widehat{\mathcal{R}}_{\sigma}\ast (f_\sigma\circ v))(t,\theta)=
		\int_{-\infty}^{+\infty} \widehat{\mathcal{R}}_{\sigma}(t-\tau,\theta-\varsigma) f_\sigma(v(\tau,\varsigma))\mathrm{d}\tau,
	\end{align*}       
	where $\widehat{\mathcal{R}}_{\sigma}:\mathcal{C}_{L}\times\mathcal{C}_{L}\rightarrow\mathbb R$ is the Riesz kernel \eqref{rieszpotential} written in Emden--Fowler coordinates. 
	Henceforth, we keep the notation ${\mathcal{K}}_{\sigma,{\rm cyl}}=  \widehat{\mathcal{K}}_{\sigma}$ and ${\mathcal{R}}_{\sigma,{\rm cyl}}=  \widehat{\mathcal{R}}_{\sigma}$ for the sake of simplicity.
	
	\begin{remark}
		Notice that $(-\Delta)_{\rm cyl}^{-\sigma}$ is an abuse of notation, which we keep for simplicity. 
		In the geometric language, this change of variables corresponds to a restriction of the conformal diffeomorphism between the entire cylinder and the punctured space.
		In other words, one has
		\begin{equation*}
			(-\Delta)_{\rm cyl}^{-\sigma}=P_{2\sigma}(g_{\rm cyl}), 
		\end{equation*}
		where $g_{{\rm cyl}}=\ud t^2+\ud\theta^2$ stands for the cylindrical metric and ${\ud\theta}=e^{-2t}\delta$, where $\delta$ is the standard flat metric.  
	\end{remark}
	
	Notice that since in the blow-up limit situation ($R=+\infty$), solutions to \eqref{ourintegralequation1pt} are rotationally invariant, that is, $u(x)=u(r)$ with $r=|x|$.      
	Using this change of variable, Eq. \eqref{ourintegralequation1pt} can be reformulated as
	the following one-dimensional problem
	\begin{align}\label{ourcylindricalequation1}\tag{$\mathcal{O}_{2\sigma,\infty}$}
		\begin{cases}
			(-\Delta)_{\rm cyl}^{\sigma} v=f_\sigma(v) \quad {\rm in} \quad \mathbb{R},\\
			\displaystyle\lim_{t\rightarrow+\infty}v(t)=0.
		\end{cases}
	\end{align}      
	Here $(-\Delta)_{\rm cyl}^{\sigma}$ represents the operator higher-order operator (written in Emden--Fowler coordinates \eqref{fowlercoordinates}), namely 
	\begin{equation}
		(-\Delta)_{\rm cyl}^{\sigma} v(t):=\int_{-\infty}^{+\infty} \widehat{\mathcal{K}}_{\sigma}(t-\tau)[v(t)-v(\tau)] \mathrm{d}\tau,
	\end{equation}
	where $\widehat{\mathcal{K}}_{\sigma}:\mathbb R\times\mathbb R\rightarrow\mathbb R$ is a kernel given by
	
	\begin{align}\label{singularkernel1}
		\widehat{\mathcal{K}}_{\sigma}(t)=2^{-\gamma^{\prime}_{\sigma}}\int_{\mathbb{S}^{n-1}} |\cosh (t)-\langle\theta, \tau\rangle|^{-\gamma_{\sigma}^{\prime}} \ud \tau=\int_{\mathbb{S}^{n-1}} e^{-\gamma_{\sigma}^{\prime}t}\left(1+e^{-2t}-2 e^{-t}\langle\theta, \tau\rangle\right)^{-\gamma_{\sigma}^{\prime}} \ud \tau.
	\end{align}
	
	As before, the dual form of this equation is given by
	\begin{align}\label{ourcylindricalequation1dualcyl}\tag{$\mathcal{O}_{2\sigma,\infty}^{\prime}$}
		\begin{cases}
			v=(-\Delta)_{\rm cyl}^{-\sigma}(f_\sigma \circ v) \quad {\rm in} \quad \mathbb{R},\\
			\displaystyle\lim_{t\rightarrow+\infty}v(t)=0.
		\end{cases}         
	\end{align}
	Here $(-\Delta)_{\rm cyl}^{-\sigma}$ is the integral linear operator defined by
	\begin{align*}
		(-\Delta)_{\rm cyl}^{-\sigma}(f_\sigma \circ v)(t):=(\widehat{\mathcal{R}}_{\sigma}\ast (f_\sigma\circ v))(t)=
		\int_{-\infty}^{+\infty} \widehat{\mathcal{R}}_{\sigma}(t-\tau) f_\sigma(v(\tau))\mathrm{d}\tau,
	\end{align*}       
	where $\widehat{\mathcal{R}}_{\sigma}:\mathbb R\times\mathbb R\rightarrow\mathbb R$ is a kernel given by
	\begin{align}\label{singularkerneldual}
		\widehat{\mathcal{R}}_{\sigma}(t)={2^{-\gamma_{\sigma}}} \omega_{n-2} \int_{-1}^{1} {\left(1-\zeta_1^2\right)^{\frac{n-3}{2}}}{\left|\cosh (t)-\zeta_1\right|^{-\gamma_{\sigma}}} \ud \zeta_1.
	\end{align}
	
	\begin{remark}
		It is possible to express this kernel in terms of hypergeometric functions.        
		We also observe $\widehat{\mathcal{R}}_{\sigma}(\xi)\sim 1$ is bounded and H\"older continuous, whereas $\widehat{\mathcal{K}}_{\sigma}(\xi)\sim  |\xi|^{1-2s}$ when $\sigma\in(1,+\infty)$.
		Furthermore, they behave qualitatively as  
		\begin{equation}\label{asmpkernel1}
			\widehat{\mathcal{K}}_{\sigma}(\xi) \sim  e^{-\gamma_\sigma^{\prime}|\xi|} \quad {\rm as} \quad |\xi| \rightarrow +\infty
		\end{equation}
		and
		\begin{equation}\label{asmpkerneldual}
			\widehat{\mathcal{R}}_{\sigma}(\xi) \sim  e^{-\gamma_\sigma|\xi|} \quad {\rm as} \quad |\xi| \rightarrow +\infty,
		\end{equation}
		where $\xi:=|t-\tau|$. 
		We refer to \cite{MR4104278,arxiv:1901.01678} for proof of these facts.
	\end{remark}
	
	Using this new formulation, one has the following
	
	\begin{remark}
		As before, there are two distinguished solutions to \eqref{ourcylindricalequation1dual}, which we describe as follows:
		\begin{itemize}
			\item[{\rm (a)}] The cylindrical solution, which is 
			\begin{equation*}
				v_{\rm cyl}(t) \equiv a_{n,\sigma},
			\end{equation*}
			where $v_{\rm cyl}=\mathfrak{F}_{\sigma}(u_{\rm cyl})\in\mathcal{C}^{2\sigma}(\mathbb R)$ with $u_{\rm cyl}\in\mathcal{C}^{2\sigma}(\mathbb R^n\setminus\{0\})$ given by \eqref{cylindricalsolutions}.
			\item[{\rm (b)}] The standard spherical solution $($also known as \quotes{bubble}$)$ which is 
			\begin{equation}\label{sphericalsolutionsEF}
				v_{\rm sph}(t)=\cosh(t)^{\gamma_{\sigma}},
			\end{equation}
			where $v_{\rm sph}=\mathfrak{F}_{\sigma}(u_{\rm sph})\in\mathcal{C}^{2\sigma}(\mathbb R)$ with $u_{\rm sph}\in\mathcal{C}^{2\sigma}(\mathbb R^n\setminus\{0\})$ given by \eqref{sphericalsolutions}.
		\end{itemize}
	\end{remark}
	
	\subsection{Asymptotic classification of Delaunay-type solutions}
	Now we prove the existence of even solutions to \eqref{ourcylindricalequation1dual} with large periods which are close 
	to the standard bubble tower solution given by \eqref{standardbubbletowerEF} in a suitable weighted H\"older norm.
	
	First, for the standard bubble solution, we have the following nondegeneracy result, which is based on \cite[Lemma~5.1]{MR3899029} and \cite[Lemma~5.1]{MR3091775}.
	In our situation, this is proved in \cite[Lemma~A.1]{kim-wei}.
	Nevertheless, we include a sketch of the proof in Appendix~\ref{sec:complentaryproofs} for completeness.
	
	\begin{lemma}\label{lm:nondegeneracy}
		Let $\sigma\in(1,\infty)$ and $n>2\sigma$.
		The standard bubble solution $u_{\rm sph}\in \mathcal{C}^{2\sigma}({\mathbb R^n})$ 
		given by \eqref{sphericalsolutions} satisfying \eqref{ourintegralequation1pt} is nondegenerate in a sense, the set of bounded solutions to the linearized equation
		\begin{equation}\label{linearizedequation}
			\phi-(-\Delta)^{-\sigma}(f^{\prime}_\sigma(u_{\rm sph})\phi)=0 \quad {\rm in} \quad \mathbb{R}^n
		\end{equation}
		are spanned by the functions
		\begin{equation*}
			\gamma_{\sigma}u_{\rm sph}+x \cdot \nabla u_{\rm sph} \quad \text { and } \quad
			\partial_{x_i} u_{\rm sph} \quad {\rm for} \quad i\in \{1,\dots,n\}. 
		\end{equation*}
	\end{lemma}
	
	\begin{proof}
		See Appendix~\ref{sec:complentaryproofs}.
	\end{proof}
	
	One can also reformulate the last result as follows
	\begin{lemma}\label{lm:nondegeneracyEF}
		Let $\sigma\in(1,\infty)$ and $n>2\sigma$.
		The standard bubble solution $v_{\rm sph}\in \mathcal{C}^{2\sigma}({\mathbb R})$ 
		given by \eqref{sphericalsolutionsEF} satisfying \eqref{ourcylindricalequation1dual} is nondegenerate in the sense that all bounded solutions of the linearized equation
		\begin{equation*}
			\psi-(-\Delta)_{\rm cyl}^{-\sigma}(f^{\prime}_\sigma(v_{\rm sph})\psi)=0 \quad {\rm in} \quad \mathbb{R}
		\end{equation*}
		are spanned by the translations
		$v_{\rm sph}(\cdot-T)$ with $T>0$.
	\end{lemma}
	
	\begin{proof}
		It follows by undoing the Emden--Fowler change of variables in \eqref{fowlercoordinates}. 
	\end{proof}
	
	Second, we restrict ourselves to the open interval $(-L,L)$ equipped with Dirichlet boundary conditions.
	In what follows, we fix $L\in\mathbb N$.
	Let $j\in\mathbb N$ and $\alpha\in (0,1)$, we denote by $\mathcal{C}^{j,\alpha}_L(\mathbb R)$ the classical H\"{o}lder space $\mathcal{C}^{j,\alpha}(\mathbb R)$ restricted to $2L$-periodic functions on the open interval $(-L,L)$.
	For $\alpha=0$, we simply denote $\mathcal{C}^{j}_L(\mathbb R)$.
	Let $\ell\in\mathbb N$ and $q\in[1,+\infty]$, we will keep the notation $W^{\ell,q}_L(\mathbb R)$ for the classical Sobolev space $W^{\ell,q}(\mathbb R)$ restricted to $2L$-periodic functions on the open interval $(-L,L)$.
	For $q=2$, we simply denote $H^{\ell}_L(\mathbb R)$.

	To seek $2L$-periodic solutions, we consider the following periodic problem
	\begin{equation}\label{ourODEperiodic}\tag{$\mathcal{O}^{\prime}_{2\sigma,L}$}
		\begin{cases}
			v=(-\Delta)_{\rm cyl}^{-\sigma,L}({f}_\sigma\circ v) \quad {\rm in} \quad \mathbb R,\\
			\lim_{t\rightarrow+\infty}v(t)=0,
		\end{cases}
	\end{equation}      
	where  $(-\Delta)_{\rm cyl}^{-\sigma,L}:\mathcal{C}_L^0(\mathbb R)\rightarrow\mathcal{C}_L^{2\sigma}(\mathbb R)$ is the integral periodic linear operator defined by
	\begin{equation*}
		(-\Delta)_{\rm cyl}^{-\sigma,L}(f_\sigma\circ v)(t):={\rm p . v.} \int_{-L}^{L}f_\sigma(v(\tau))\widehat{\mathcal{R}}_{\sigma,L}(t-\tau) \ud \tau.
	\end{equation*}
	For this, we shall work with the norm given by
	\begin{equation*}
		\|v\|_{H_{L,0}^{\sigma}(\mathbb R)}:= \left([v^{(m)}]_{L_{L}^{s}(\mathbb R)}+\sum_{\ell=0}^{m}\|v^{(\ell)}\|^2_{L_L^2(\mathbb R)}\right)^{1/2},
	\end{equation*}
	where 
	\begin{equation*}
		[v^{(m)}]_{L_{L}^{s}(\mathbb R)}:=\int_{-L}^{L} \int_{-L}^{L}[v^{(m)}(t)-v^{(m)}(\tau)]^2 \widehat{\mathcal{K}}_{s,L}(t-\tau) \ud \tau \ud t.
	\end{equation*}
	
	We also define the following higher-order functional space
	\begin{align*}
		H_{L}^{\sigma}(\mathbb R)=\{v\in \mathcal{C}^{2\sigma}_L(\mathbb R):  \ \|v\|_{H_{L}^{\sigma}(\mathbb R)}<\infty
		\}.
	\end{align*}
	Furthermore, the evenness and periodicity
	\begin{align*}
		H_{L,*}^{\sigma}(\mathbb R)=\{v\in H_{L}^{\sigma}(\mathbb R): 
		v(t)=v(-t) \ {\rm and} \ v(t+2{L})=v(t) \ {\rm for \ all} \ t\in\mathbb R \}.
	\end{align*}
	As well as,          
	taking into consideration the boundary condition 
	\begin{align*}
		H_{L,0}^{\sigma}(\mathbb R)=\{v\in H_{L}^{\sigma}(\mathbb R): v^{(\ell)}(-L)=v^{(\ell)}(L)=0 \ {\rm for} \ \ell\in\{1,\dots,m\} \}.
	\end{align*}
	Finally, a suitable space to work is 
	\begin{align*}
		H_{L,0,*}^{\sigma}(\mathbb R)=H_{L,0}^{\sigma}(\mathbb R)\cap H_{L,*}^{\sigma}(\mathbb R).
	\end{align*}
	Here $\widehat{\mathcal{K}}_{s,L}:(-L,L)\times(-L,L)\rightarrow\mathbb R$ given by
	\begin{equation}\label{periodickernel1}
		\widehat{\mathcal{K}}_{s,L}(t-\tau)=\sum_{j \in \mathbb Z} \widehat{\mathcal{K}}_s(t-\tau-j L)
	\end{equation}      
	is a periodic Kernel, where $\widehat{\mathcal{K}}_s:\mathbb R\times\mathbb R\rightarrow\mathbb R$ is defined as \eqref{singularkernel1}.
	
	Now we will introduce some standard H\"older fractional from \cite[Theorem 8.2]{MR2944369}. 
	
	\begin{lemmaletter}\label{lm:fractionalregularity}
		Let $s\in(0,1)$ and $n>2s$.
		Assume that $p\in [1,+\infty)$. 
		Then, there exists a constant $C>0$, depending on $\sigma$ and $p$, such that
		\begin{equation}
			\|v\|_{\mathcal{C}^{0,\alpha}_L(\mathbb R)}\leqslant C\left(\|v\|^p_{L^{p}_L(\mathbb R)} +\int_{-L}^{L}\int_{-L}^{L}\frac{|v(t)-v(\tau)|^p}{|t-\tau|^{1+s p}}\ud t \ud\tau\right)^{\frac{1}{p}}
		\end{equation}
		for any $v\in L^p_L(\mathbb R)$, where $\alpha=s-\frac{1}{p}$.
	\end{lemmaletter}
	
	\begin{lemma}\label{lm:fractionalregularityhigherorder}
		Let $\sigma\in(1,+\infty]$ and $n>2\sigma$.
		Assume that $p\in [1,+\infty)$ and $\sigma\in(1,\frac{n}{2})$ is such that  $\sigma-\frac{1}{p}\notin\mathbb Z$. 
		Then, there exists a constant $C>0$, depending on $\sigma$ and $p$, such that
		\begin{equation}
			\|v\|_{\mathcal{C}^{\ell,\alpha}_L(\mathbb R)}\leqslant C\left(\|v\|^p_{W^{m,p}_L(\mathbb R)} +\int_{-L}^{L}\int_{-L}^{L}\frac{|v^{(m)}(t)-v^{(m)}(\tau)|^p}{|t-\tau|^{1+s p}}\ud t \ud\tau\right)^{\frac{1}{p}}
		\end{equation}
		for any $v\in W^{\sigma,p}_L(\mathbb R)$, where $\ell=\lfloor \sigma-\frac{1}{p} \rfloor$ and $\alpha=\sigma-\frac{1}{p}-\lfloor \sigma-\frac{1}{p} \rfloor$.
	\end{lemma}
	
	\begin{proof}
		It is a direct consequence of Lemma~\ref{lm:fractionalregularity} by using a standard induction argument.
	\end{proof}
	
	We also need the following strong maximum principle.
	
	\begin{lemma}\label{lm:maximumprinciple}
		Let $\sigma\in(1,+\infty]$ and $n>2\sigma$.
		If $v \in H_L^{\sigma}(\mathbb{R}) \cap \mathcal{C}^0(\mathbb{R})$ is a nonnegative solution to \eqref{ourcylindricalequation1dual}.
		Then, either $v>0$ or $v \equiv 0$.
	\end{lemma}
	
	\begin{proof}
		Indeed, since $v \geqslant 0$, it follows that
		\begin{equation}\label{mp1}
			v=(-\Delta)_{\rm cyl}^{-\sigma}({f}_\sigma\circ v)\geqslant0.
		\end{equation}
		Assume that there exists a point $t_0 \in \mathbb{R}$ with $v(t_0)=0$, then
		\begin{align*}
			&v(t_0)-v(t)=(-\Delta)_{\rm cyl}^{-\sigma} [{f}_\sigma(v(t_0))- {f}_\sigma(v(t))]\\  =&\mathrm{p.v.} \int_{-\infty}^{+\infty} {f}_\sigma(v(t_0)) \widehat{\mathcal{R}}_{\sigma}(t_0-\tau) \ud \tau-\mathrm{p.v.} \int_{-\infty}^{+\infty} {f}_\sigma(v(t)) \widehat{\mathcal{R}}_{\sigma}(t-\tau) \ud \tau\\
			=&-\mathrm{p.v.} \int_{-\infty}^{+\infty} {f}_\sigma(v(\tau))\widehat{\mathcal{R}}_{\sigma}(t-\tau) \ud \tau \leqslant 0
		\end{align*}
		satisfies \eqref{mp1} only in the case $v \equiv 0$.
	\end{proof}  
	
	Now we have the most important lemma in this section
	\begin{lemma}\label{lm:estimatesdelaunay}
		Let $\sigma \in(1,+\infty)$ and $n>2\sigma$.
		For any $L\gg1$ sufficiently large,
		there exist there exist a sequence of periods $(L_j)\in \ell^\infty(\mathbb R_+)$, an error function $\psi_{(0,L_j)}\in H^\sigma_L(\mathbb R)$ and a unique positive even solution $\bar{v}_{(0,L_j)}\in H_L^\sigma(\mathbb R)$ to the following periodic boundary value problem \begin{equation}\label{ourODEperidociboundary}\tag{$\overline{{\mathcal{O}}}^{\prime}_{2\sigma,L}$}
			\left\{\begin{aligned}
				v=(-&\Delta)_{\rm cyl}^{-\sigma,L}({f}_\sigma\circ v) \quad \text { in } \quad (-L, L),\\
				v^{(\ell)}(-L) & =v^{(\ell)}(L)=0 \quad \text { for } \quad \ell=1, 3, \dots,2m-1,
			\end{aligned}\right.
		\end{equation}
		which satisfy 
		\begin{equation*}
			\bar{v}_{(0,L_j)}(t)= \widehat{V}^+_{(0,L_j)}(t)+\psi_{(0,L_j)}(t)
		\end{equation*}
		and 
		\begin{equation*}
			\|\psi_{(0,L_j)}\|_{H^\sigma(\mathbb R)}\rightarrow 0 \quad {\rm as} \quad L\rightarrow+\infty,
		\end{equation*}
		where $\widehat{V}^+_{(0,L_j)}=\sum_{j\in\mathbb Z}V_{(0,L_j)}(t)$ with $V_{(0,L_j)}(t)=\cosh(t-L_j)$ and $L_j=2jL$ for $j\in\mathbb Z$ is the standard bubble tower solution $($see Definition~\ref{standardbubbletower}$)$  .                    
		Moreover, we have the following H\"older estimate
		\begin{equation}\label{refinedasymptoticsmodel}
			\|\psi_{(0,L_j)}\|_{\mathcal{C}^{2\sigma}_L(\mathbb R)} \lesssim e^{-\gamma_{\sigma}L(1+\xi)}
		\end{equation}
		for some $\alpha \in(0,1)$ and $\xi>0$ independent of the period $L\gg1$ large.
	\end{lemma}
	
	\begin{proof}
		First, by symmetry $\bar{v}_{L}\in\mathcal{C}^{2\sigma}_L(\mathbb R)$ given by \eqref{standardbubbletower} satisfies the boundary condition at $t=\pm L$, that is, $\widehat{V}^+_{(0,L_j)}\in H_{(0,L_j)}^\sigma(\mathbb R)$.
		Now writing $v=\widehat{V}^+_{(0,L_j)}+\psi$, we can 
		reformulate \eqref{ourODEperiodic} as
		\begin{equation*}
			\mathscr{N}_{\sigma,L}(\widehat{V}^+_{(0,L_j)}+\psi)=0 \quad {\rm in} \quad \mathbb R,
		\end{equation*}
		where 
		\begin{align}\label{nonlinearfunctionalcylindrical}
			\mathscr{N}_{\sigma,L}(v):=v-(-\Delta)_{\rm cyl}^{-\sigma,L}(f_{\sigma}\circ v).
		\end{align}
		From now on, let us fix the notation
		\begin{align*}
			\mathscr{N}_{\sigma}(0,L_j)(\psi):=\mathscr{N}_{\sigma,L}(\widehat{V}^+_{(0,L_j)}+\psi)
		\end{align*}
		Next, by linearizing this functional 
		around the standard bubble tower solution, we find
		\begin{equation}\label{ourlinearizedequation}
			\mathscr{L}_{\sigma}(0,L_j)(\psi)=\mathscr{E}_{\sigma}(0,L_j)(\widehat{V}^+_{(0,L_j)})+\mathscr{S}_{\sigma}(0,L_j)(\psi),
		\end{equation}
		where $ \mathscr{L}_{\sigma}(0,L_j):H_{L}^{\sigma}(\mathbb{R})\to  H_{L}^{\sigma}(\mathbb{R})$ defined as $ \mathscr{L}_{\sigma}(0,L_j):=\ud\mathscr{N}_{\sigma}[\widehat{V}^+_{(0,L_j)}]$ satisfies
		\begin{equation}\label{ourmodellinearizedequation}
			\mathscr{L}_{\sigma}(0,L_j)(\psi):=\psi-\mathscr{K}_{\sigma}(0,L_j)(\psi),
		\end{equation}  
		where
		\begin{align}\label{compactoperator}
			\mathscr{K}_{\sigma}(0,L_j)(\psi):=(-\Delta)^{-\sigma}(f^{\prime}_{\sigma}\circ \widehat{V}^+_{(0,L_j)})\psi=\int_{-L}^{L}f^{\prime}_\sigma(\widehat{V}^+_{(0,L_j)})\psi\widehat{\mathcal{R}}_{\sigma,L}(t-\tau)\ud \tau
		\end{align} 
		represents the derivative of the nonlinear functional \eqref{nonlinearfunctionalcylindrical} at the standard bubble tower solution \eqref{standardbubbletower}.
		Also, the superlinear term $ \mathscr{L}_{\sigma}(0,L_j):H_{L}^{\sigma}(\mathbb{R})\to  H_{L}^{\sigma}(\mathbb{R})$ is given by
		\begin{equation*}
			\mathscr{S}_{\sigma}(0,L_j)(\psi)=\int_{-L}^{L}\left[f_\sigma(\widehat{V}^+_{(0,L_j)}+\psi)-f_\sigma(\widehat{V}^+_{(0,L_j)})-f^{\prime}_\sigma(\widehat{V}^+_{(0,L_j)})\psi\right]\widehat{\mathcal{R}}_{\sigma,L}(t-\tau)\ud \tau.
		\end{equation*}
		is a superlinear term, and 
		the remainder error term is given by
		\begin{equation}\label{erroroperator}
			\mathscr{E}_{\sigma}(0,L_j)(\widehat{V}^+_{(0,L_j)})=\int_{-L}^{L}\left[f_\sigma\left(\sum_{j\in\mathbb{Z}}V_{(0,L_j)}(\tau)\right)-\sum_{j\in\mathbb{Z}}f_\sigma\left(V_{(0,L_j)}(\tau)\right)\right]\widehat{\mathcal{R}}_{\sigma,L}(t-\tau)\ud \tau,
		\end{equation}
		which represents the error in approximating a solution to \eqref{standardbubbletower} by a standard bubble-tower solution.
		
		To apply classical Fredholm theory
		we need to prove the following claim:           
		
		\noindent{\bf Claim 1:} The operator $\mathscr{K}_{\sigma}(0,L_j): H_{L}^{\sigma}(\mathbb{R})\to  H_{L}^{\sigma}(\mathbb{R})$ defined in \eqref{compactoperator} is bounded and compact and satisfies
		\begin{equation}\label{boundedestimate}
			\left\|\mathscr{K}_{\sigma}(0,L_j)(\psi)\right\|_{H^{\sigma}_{L}(\mathbb{R})}\lesssim\left\|\psi\right\|_{L^{2}_{L}(\mathbb{R})}
		\end{equation}
		uniformly on $L\gg1$ large.
		
		\noindent Initially, we prove that the operator is bounded and compact.
		Indeed, by its definition, $\widehat{\mathcal{K}}_{\sigma,L}\in\mathcal{C}^{j,\alpha}(\mathbb{R})$ for any $j\in\{0,\dots, m\}$ and some $\alpha\in(0,1)$. Moreover, for each $j$, we have
		\begin{equation}\label{kerneldecay}
			\left|\frac{\ud^j }{\ud t^j}\widehat{\mathcal{K}}_{\sigma,L}(t)\right|\lesssim e^{-c_j |t|}
		\end{equation}
		uniformly on $L\gg 1$ large.
		Similarly, we have 
		\begin{equation*}
			|\widehat{V}^+_{(0,L_j)}|\lesssim e^{-c|t|}
		\end{equation*}
		uniformly on $L\gg 1$ large. By H\"older's inequality, it follows directly 
		\begin{equation*}
			\left\|\mathscr{K}_{\sigma}(0,L_j)(\psi)^{(j)}\right\|_{L_{L}^{2}(\mathbb{R})}\lesssim \left\|\psi\right\|_{L_{L}^{2}(\mathbb{R})} \quad {\rm for \ all} \quad j\in\{0,\dots,m\}
		\end{equation*}             uniformly on $L\gg 1$ large. 
		
		\noindent Also, using the H\"older continuity of $(\widehat{\mathcal{K}}_{\sigma,L})^{(m)}\in\mathcal{C}^{0,\alpha+s}(\mathbb{R})$, we have
		\begin{align*}
			&\left|\mathscr{K}_{\sigma}(0,L_j)(\psi)^{(m)}(\tau)-\mathscr{K}_{\sigma}(0,L_j)(\psi)^{(m)}(t)\right|\\
			\leqslant&\int_{-L}^{L}f^\prime_\sigma(\widehat{V}^+_{(0,L_j)}(\xi)) \left|\frac{\ud^m}{\ud\xi^m}\left(\widehat{\mathcal{K}}_{\sigma,L}(\tau-\xi)-\widehat{\mathcal{K}}_{\sigma,L}(t-\xi)\right)\right||\psi(\xi)|  \ud \xi\\
			\lesssim&\int_{-L}^{L}|\psi(\xi)||t-\tau|^{\alpha} \ud \xi.
		\end{align*}
		Thus, by using the asymptotic behavior of the kernel near the origin given by \eqref{asmpkernel1} and \eqref{kerneldecay} combined with the last inequality, we obtain
		\begin{align*}
			\left[\mathscr{K}_{\sigma}(0,L_j)(\psi)\right]_{L^{s}_{L}(\mathbb{R})}&=\int_{-L}^{L}\left|\mathscr{K}_{\sigma}(0,L_j)(\psi)^{(m)}(\tau)-\mathscr{K}_{\sigma}(0,L_j)(\psi)^{(m)}(t)\right|^2 \widehat{\mathcal{K}}_{\sigma,L}(t-\tau) \ud \tau\ud t\\
			&\lesssim \left\|\psi\right\|_{L^{2}_{L}(\mathbb{R})},
		\end{align*}
		uniformly on $L\gg 1$ large, which proves \eqref{boundedestimate}.            
		\noindent In conclusion, by compact embedding, the desired conclusion holds for the map $\mathscr{K}_{\sigma}(0,L_j): H_{L}^{\sigma}(\mathbb{R})\to  H_{L}^{\sigma}(\mathbb{R})$.

		The proof of the first claim is now finished.
		
		Second, now in order to apply Fredholm alternative to conclude that for any $h\in L^{2}_{L}(\mathbb{R})$, there exists a unique solution $\psi_{(0,L_j)}\in H_{L}^{\sigma}(\mathbb{R})$ to              
		the linear inhomogeneous problem
		\begin{equation*}
			\psi_{(0,L_j)}-\mathscr{K}_{\sigma}(0,L_j)(\psi_{(0,L_j)})=h \quad {\rm in} \quad (-L,L).
		\end{equation*}
		One needs to prove the uniqueness result below:
		
		\noindent{\bf Claim 2:} 
		The linear homogeneous equation               
		\begin{equation*}
			\psi-\mathscr{K}_{\sigma}(0,L_j)(\psi)=0 \quad {\rm in} \quad (-L,L)
		\end{equation*}
		admits only zero solutions in $L^{2}_{L}(\mathbb{R})$. 
		
		\noindent As a matter of fact, note that the equation above with H\"older's inequality yields directly that 
		\begin{equation}\label{Linfty estimate}
			\left\|\psi\right\|_{L^{\infty}_L(\mathbb R)}\lesssim\left\|\psi\right\|_{L^{2}_L(\mathbb R)}.
		\end{equation}
		
		\noindent Next, we use the nondegeneracy of the standard bubble solution in Lemma~\ref{lm:nondegeneracy} to conclude that $\psi\equiv 0$, and thus we prove Claim 2.
		
		Lastly, by the standard fixed-point argument, a unique solution $\psi_{(0,L_j)}\in H_{L}^{\sigma}(\mathbb{R})$ to \eqref{ourlinearizedequation} satisfying the estimate
		\begin{equation}\label{refinedasymptotics0}
			\|\psi_{(0,L_j)}\|_{H_{L,0}^\sigma(\mathbb R)} \lesssim\|\mathscr{E}_{\sigma}(0,L_j)(\widehat{V}^+_{(0,L_j)})\|_{L_{L}^2(\mathbb R)},
		\end{equation}   
		to conclude the proof of the proposition, we are left to obtain estimates for the right-hand side of the last inequality.
		
		This is the content of our third claim.
		
		\noindent{\bf Claim 3:} It holds that $
		\|\psi\|_{H_{L,0}^\sigma(\mathbb R)} \lesssim e^{-\gamma_{\sigma}L(1+\xi)}$ for some $\xi>0$ uniformly on $L\gg1$ large.
		
		\noindent In fact, using \eqref{standardbubbletower} it follows  
		\begin{equation*}
			\mathscr{E}_{\sigma}(0,L_j)(\widehat{V}^+_{(0,L_j)})=c_{n, \sigma}\int_{-L}^{L}\left[\left(\sum_{j\in\mathbb Z} V_{(0,L_j)}(\tau)\right)^{\frac{n+2\sigma}{n-2\sigma}}-\left(\sum_{j\in\mathbb Z} V_{(0,L_j)}(\tau)^{\frac{n+2\sigma}{n-2\sigma}}\right)\right]\widehat{\mathcal{R}}_{\sigma,L}(t-\tau)\ud \tau.
		\end{equation*}
		Since by symmetry, we have $V_{(0,-L_j)}(t) \leqslant V_{(0,L_j)}(t)$ for $t \geqslant 0$, it holds 
		\begin{align}\label{errorestimates1}
			|\mathscr{E}_{\sigma}(0,L_j)(\widehat{V}^+_{(0,L_j)})| &\lesssim \int_{-L}^{L}\left(V_{(0,\infty)}^{\frac{4\sigma}{n-2\sigma}} \sum_{j\in\mathbb Z^*}V_{(0,L_j)}(\tau)+\sum_{j\in\mathbb Z^*} V_{(0,L_j)}(\tau)^{\frac{n+2\sigma}{n-2\sigma}}\right)\widehat{\mathcal{R}}_{\sigma,L}(t-\tau)\ud \tau\\\nonumber
			&\lesssim\sum_{j\in\mathbb Z^*}\int_{-L}^{L}V_{(0,\infty)}(\tau)^{\frac{4\sigma}{n-2\sigma}}  V_{(0,L_j)}(\tau) \widehat{\mathcal{R}}_{\sigma,L}(t-\tau)\ud \tau+
			\sum_{j\in\mathbb Z^*}\int_{-L}^{L} V_{(0,L_j)}(\tau)^{\frac{n+2\sigma}{n-2\sigma}}\widehat{\mathcal{R}}_{\sigma,L}(t-\tau)\ud \tau.
		\end{align}
		
		\noindent From \eqref{errorestimates1}, we find
		\begin{align}\label{refinedasymptotics7}
			\nonumber
			&\int_{-L}^{L} |\mathscr{E}_{\sigma}(0,L_j)(\widehat{V}^+_{(0,L_j)})|^2\ud t\\\nonumber 
			&\lesssim\int_{-L}^{L}\left( \sum_{j\in\mathbb Z^*}\int_{-L}^{L}V_{(0,\infty)}(\tau)^{\frac{8\sigma}{n-2\sigma}}  V_{(0,L_j)}(\tau)^2 \widehat{\mathcal{R}}_{\sigma,L}(t-\tau)^2\ud \tau+
			\sum_{j\in\mathbb Z^*}\int_{-L}^{L} V_{(0,L_j)}(\tau)^{\frac{2n+4\sigma}{n-2\sigma}}\widehat{\mathcal{R}}_{\sigma,L}(t-\tau)^2\ud \tau\right)\ud t\\\nonumber
			&\lesssim \sum_{j\in\mathbb Z^*} \int_{-L}^{L}\int_{-L}^{L}
			V_{(0,\infty)}(\tau)^{\frac{8\sigma}{n-2\sigma}}V_(0,L_j)(\tau)^2 \widehat{\mathcal{R}}_{\sigma,L}(t-\tau)^2\ud \tau\ud t+\sum_{j\in\mathbb Z^*}\int_{-L}^{L} \int_{-L}^{L} V_{(0,L_j)}(\tau)^{\frac{2n+4\sigma}{n-2\sigma}}\widehat{\mathcal{R}}_{\sigma,L}(t-\tau)^2\ud \tau\\
			&=:I_1+I_2.
		\end{align}
		
		\noindent To estimate these two terms, we fix $\alpha \in(0,1)$ and subdivide $\mathbb R=\{|t| \leqslant \alpha L\}\mathbin{\dot{\cup}}\{|t| \geqslant \alpha L\}$.
		Then, we use the exponential decay of the standard bubble solution from Proposition~\ref{prop:asymptotics} to obtain 
		\begin{equation}\label{refinedasymptotics3}
			\sum_{j\in\mathbb Z^*} V_{(0,L_j)}\lesssim e^{-\gamma_{\sigma}L(2-\alpha)} \quad {\rm and } \quad \sum_{j\in\mathbb Z^*} V_{(0,L_j)} \lesssim  e^{-\gamma_{\sigma}L}.
		\end{equation}
		Hence, by substituting in \eqref{refinedasymptotics3} into the first term in \eqref{refinedasymptotics7}, we obtain
		\begin{align}\label{errorestimates2}
			\nonumber
			I_1&=\sum_{j\in\mathbb Z^*} \int_{-L}^{L}\int_{-L}^{L}
			V_{(0,\infty)}(\tau)^{\frac{8\sigma}{n-2\sigma}}V_{(0,L_j)}(\tau)^2 \widehat{\mathcal{R}}_{\sigma,L}(t-\tau)^2\ud \tau\ud t\\
			&\lesssim e^{-2\gamma_{\sigma}L(2-\alpha)}+e^{-2\gamma_\sigma L\left(\frac{4\sigma\alpha}{n-2\sigma}\right)}e^{-2\gamma_{\sigma} L}e^{-2\gamma_\sigma L\left(\frac{n+2\sigma}{n-2\sigma}\right)}\\\nonumber
			&\lesssim e^{-2\gamma_{\sigma}L(2-\alpha)}+e^{-2\gamma_{\sigma} L(1+\xi)}
		\end{align}
		for some $\xi>0$ (depending only on $n$, $\sigma$, and $\alpha$), where we used the asymptotic behavior of the Kernel \eqref{asmpkerneldual} for $\tau\rightarrow+\infty$ large and the fact that it is bounded for $\tau\rightarrow0$ small.
		
		\noindent Furthermore, by substituting \eqref{refinedasymptotics3} into the second term in \eqref{refinedasymptotics7}, we have
		\begin{align}\label{errorestimates3}
			I_2=\sum_{j\in\mathbb Z^*}\int_{-L}^{L} \int_{-L}^{L} V_{(0,L_j)}(\tau)^{\frac{2n+4\sigma}{n-2\sigma}}\widehat{\mathcal{R}}_{\sigma,L}(t-\tau)^2 \ud \tau\ud t\lesssim e^{-2\gamma_{\sigma}L(2-\alpha)}+e^{-2\gamma_{\sigma}L(\frac{n+2\sigma}{n-2\sigma})}.
		\end{align}
		In conclusion, by substituting \eqref{errorestimates2} and \eqref{errorestimates3} into \eqref{errorestimates1}, we have
		\begin{equation*}
			\|\mathscr{E}_{\sigma}(0,L_j)(\widehat{V}^+_{(0,L_j)})\|_{L^2(\mathbb R)} \lesssim e^{-\gamma_{\sigma}L(1+\xi)}
		\end{equation*}
		for some $\xi>0$ uniformly on $L\gg1$ large, which combined with \eqref{refinedasymptotics0} proves the third claim.

		Finally, by standard estimates in Lemma~\ref{lm:fractionalregularityhigherorder} combined with the regularity lifting theorem from \cite[Theorem 3.3.1]{MR2759774} applied to \eqref{ourlinearizedequation}, it follows that $\psi_{(0,L_j)}\in H_{L,0,*}^\sigma(\mathbb R)$ is smooth and satisfies
		\begin{equation*}
			\|\psi_{(0,L_j)}\|_{\mathcal{C}^{2\sigma+\alpha}(\mathbb R)} \lesssim e^{-\gamma_{\sigma}L(1+\xi)} 
		\end{equation*}
		for some $\xi>0$ independent of $L\gg1$ large.

		Therefore, the maximum principle in Lemma~\ref{lm:maximumprinciple} concludes the proof of the proposition.
	\end{proof}
	
	\begin{remark}
		It is worth noticing that the last proof differs in spirit from the local inversion technique in \cite[Proposition~2.3]{MR3694645}.
		Instead of using this method, we give an alternative proof based on the dual formulation from Lemma~\ref{lm:dualrepresentation}. 
		This technique is of independent interest to a larger class of integral equations not necessarily arising as the dual of a differential equation.
	\end{remark}
	
	\begin{remark}
		We notice that it is straightforward to extend the local inversion method in \cite[Proposition~2.3]{MR3694645} at least for the higher order local cases $\sigma=m\in\mathbb N$.
		For see this fact, we write the poly-harmonic operator in Emden--Fowler coordinates, which gives us
		\begin{equation*}
			(-\Delta)_{\rm cyl}^m:=(-\Delta)_{\rm rad}^m+(-\Delta)_{\rm ang}^m
		\end{equation*}
		with 
		\begin{equation*}
			(-\Delta)^m_{\rm rad}:=\partial_t^{(m)}-K^{(0)}_{2m-2}\partial_t^{(2m-2)}+\cdots+(-1)^{m}K_1^{(0)}\partial_t^{(1)}+(-1)^{m+1}K_0^{(0)},
		\end{equation*}
		and 
		\begin{equation*}
			(-\Delta)^m_{\rm ang}:=\sum_{\ell=1}^{2m}\sum_{j=0}^{2m}(-1)^{\frac{j+2}{2}}K^{(\ell)}_{2m,j}\partial_t^{(j)}(-\Delta_\theta)^{\ell},
		\end{equation*}
		where $K^{(\ell)}_{2m,j}=K^{(\ell)}_{2m,j}(n)>0$ for $j\in\{0,\dots,2m\}$ and $\ell\in\{1,\dots,2m\}$ are dimensional constants.
		For this computation, we refer the interested reader to \cite{arXiv:2210.04376}.
		After that, we need to build on the classification result from standard bubbles for the critical Sobolev embedding $H^m(\mathbb R^n)\rightarrow L^{2_m^*}(\mathbb R^n)$, where $2_m^*:=\frac{2n}{n-2m}$ from \cite{MR1679783} and a standard nondegeneracy technique as in Lemma \ref{lm:nondegeneracy}. 	
	\end{remark}

	As an immediate consequence of the last proposition, one has
	\begin{corollary}\label{cor:estimatesdelaunay}
		Let $\sigma \in(1,+\infty)$ and $n>2\sigma$.
		For any $L\gg1$ sufficiently large,
		there exist there exist a sequence of periods $(L_j)\in \ell^\infty(\mathbb R_+)$, an error function $\psi_{(0,L_j)}\in H^\sigma_L(\mathbb R)$ and a unique positive even periodic solution $\bar{v}_{(0,L_j)}\in H_L^\sigma(\mathbb R)$ to \eqref{ourcylindricalequation1dual} satisfying
		\begin{equation*}
			\bar{v}_{(0,L_j)}(t)= \widehat{V}^+_{(0,L_j)}(t)+\psi_{(0,L_j)}(t)
		\end{equation*}
		and 
		\begin{equation*}
			\|\psi_{(0,L_j)}\|_{H_L^\sigma(\mathbb R)}\rightarrow 0 \quad {\rm as} \quad L\rightarrow+\infty,
		\end{equation*}
		where $\widehat{V}^+_{(0,L_j)}\in \mathcal{C}^{2\sigma}(\mathbb R)$ is the standard bubble tower solution given by \eqref{standardbubbletowerEF}.                    
		Moreover, we have the following H\"older estimate
		\begin{equation}\label{holderestimate}
			\|\psi_{(0,L_j)}\|_{\mathcal{C}^{2\sigma}_L(\mathbb R)} \lesssim e^{-\gamma_{\sigma}L(1+\xi)}
		\end{equation}
		for some $\alpha \in(0,1)$ and $\xi>0$ independent of the period $L\gg1$ large.
	\end{corollary}
	
	Since \eqref{ourcylindricalequation1} is translational invariant, we now will use the periodic solution $\bar{v}_{(0,L_j)}$ which attains its minimum at the points $t=2j L$ with $ j \in \mathbb{Z}$.        
	Indeed, using Lemma~\ref{lm:estimatesdelaunay}, this periodic solution can be expressed as a perturbation of a bubble tower with estimated error.
	\begin{definition}
		Let $\sigma \in(1,+\infty)$ and $n>2\sigma$.
		For any $L\gg1$ sufficiently large,
		let us define the generalized bubble tower solution
		\begin{itemize}  
			\item[{\rm (i)}] $($Emden--Fowler coordinates$)$
			\begin{equation}
				\bar{v}_{(0,L_j)}(t):=\widehat{V}^+_{(0,L_j)}(t)+\psi_{(0,L_j)}(t),
			\end{equation}
			where $\widehat{V}^+_{(0,L_j)}\in\mathcal{C}^{2\sigma}(\mathbb R)$ is the standard half bubble tower solution given by \eqref{standardbubbletowerEF+} and $\psi_{(0,L_j)}\in\mathcal{C}^{2\sigma}(\mathbb R)$ the perturbation function constructed in Corollary~\ref{cor:estimatesdelaunay}. 
			More precisely, one has 
			\begin{equation*}
				\widehat{V}^+_{(0,L_j)}(t)=\sum_{j\in\mathbb N}\cosh (t-L_j-L)^{\gamma_\sigma}, \quad {\rm where} \quad L_j=(1+2j)L \quad {\rm for} \quad j \in \mathbb{N}.
			\end{equation*}
			\item[{\rm (ii)}] $($Spherical coordinates$)$
			\begin{equation}
				\bar{u}_{(0,L_j)}(x):=\widehat{U}^+_{(0,L_j)}(x)+\phi_{(0,L_j)}(x),
			\end{equation}
			where $\widehat{U}^+_{(0,L_j)}\in\mathcal{C}^{2\sigma}(\mathbb R^n\setminus \Sigma)$ is the standard half bubble tower solution given by \eqref{standardbubbletower+} and $\phi_{(0,L_j)}\in\mathcal{C}^{2\sigma}(\mathbb R^n\setminus\Sigma)$ is the perturbation function constructed in Corollary~\ref{cor:estimatesdelaunaySph}.
			More precisely, we have
			\begin{equation*}
				\widehat{U}^+_{(0,L_j)}(x)=\sum_{j\in\mathbb N} \left(\frac{\lambda_j}{\lambda_j^2+|x|^2}\right)^{\gamma_\sigma},
				\quad {\rm where} \quad \lambda_j=e^{-(1+2j)L} \quad {\rm for} \quad j \in \mathbb{N}.
			\end{equation*}     
		\end{itemize}        
	\end{definition}
	
	We find better asymptotics near the isolated singularities for the deformed solution obtained in Lemma~\ref{lm:estimatesdelaunay}.
	These refined estimates in terms of the bubble tower solution will be a crucial part of estimating the errors in our approximate solution in the gluing procedure in Section~\ref{sec:gluingtechnique}.


	\begin{lemma}\label{lm:refinedestimatesEF}
		The asymptotics holds
		\begin{equation}\label{refinedasymptoticsEF}
			\bar{v}_{(0,L_j)}(t)=v_{\rm sph}(t)(1+\mathrm{o}(1)) \quad {\rm as} \quad L\rightarrow+\infty,
		\end{equation}
		or undoing the Emden--Fowler change of variables, it holds
		\begin{equation*}
			\bar{u}_{(0,L_j)}(t)=u_{\rm sph}(|x|)(1+\mathrm{o}(1)) \quad {\rm as} \quad L\rightarrow+\infty.
		\end{equation*}
		Moreover, one has 
		\begin{equation}\label{refinedasymptotics}
			\bar{u}_{(0,L_j)}(x)=|x|^{n-2\sigma} e^{-\gamma_{\sigma}L}(1+\mathrm{o}(1))  \quad {\rm as} \quad L\rightarrow+\infty.
		\end{equation}
		and
		\begin{equation*}
			\varepsilon_L:=\bar{v}_{(0,L_j)}(0)=e^{-\gamma_\sigma L}(1+\mathrm{o}(1)) \quad {\rm as} \quad L\rightarrow+\infty.
		\end{equation*}
		This parameter is called the neck size or Delaunay parameter.   
	\end{lemma}
	
	\begin{proof}
		Notice that for $t \leqslant 0$ (or $|x| \geqslant 1$), the proof of the lemma follows as a combination of Corollary~\ref{cor:estimatesdelaunay} together with exponential decay of the standard spherical solution in Emden--Fowler coordinates to prove \eqref{refinedasymptoticsEF}.
	\end{proof}
	
	\section{Approximate solution}\label{sec:approximatesolutions}
	This section will construct a suitable approximate solution to \eqref{ourintegralequationdual}. 
	We also prove some estimates concerning the behavior of such a solution near the singular set.
	As we have mentioned, one of the main ideas is that, although we would like the approximate solution to have Delaunay-type singularities around each point isolated singularity, it should have a fast decay once we are away from the singular set to glue to the flat background manifold. 
	To this end, we will only take half a Delaunay solution (this is, only values $j\in\mathbb N$).  
	
	\subsection{Local asymptotic behavior}
	In this subsection, we study the local behavior of solutions to \eqref{ourintegralequationdual1pt} near the isolated singularity at the origin. 
	Namely, we show that near the origin, it can be approximated by a bubble tower solution.
	In contrast with the cases $\sigma\in\{1,2,3\}$ on which a complete classification of this local behavior is given in terms of the two-parameter family of Delaunay solutions studied in \cite{MR982351,MR3869387,arXiv:2210.04376}, which are inspired by the classical result of Korevaar {\it et al.} \cite{MR1666838} for $\sigma=1$ and Caffarelli {\it et al.} $\sigma\in(0,1)$ \cite{MR3198648}.
	These will be the building blocks in constructing suitable approximate solutions to \eqref{ourintegralequationdual1pt}.
	
	First, recall the local asymptotic classification result from \cite{arxiv:1901.01678}.
	\begin{propositionletter}\label{prop:asymptotics}
		Let $\sigma \in(1,+\infty)$ and $n>2\sigma$.
		
		\begin{itemize}
			\item[{\rm (i)}] Assume that $R=+\infty$. For any $L\gg1$ sufficiently large, there exists a blow-up limit solution to  \eqref{ourintegralequationdual1pt} denoted by $u_L\in \mathcal{C}^{2\sigma}(\mathbb R^n\setminus\{0\})$ and given by 
			\begin{equation*}
				u_{(0,L)}(x)=\left(\mathfrak{F}_\sigma\right)^{-1}(v_{(0,L)})=|x|^{-\gamma_\sigma}v_{(0,L)}(-\ln|x|),
			\end{equation*}
			where $v_L\in \mathcal{C}^{2\sigma}(\mathbb R)$ is a bounded periodic even solution to \eqref{ourcylindricalequation1dual}.
			In addition, one has 
			\begin{equation*}
				v_{(0,L)}(x)=\mathcal{O}(e^{-\gamma_\sigma L}) \quad {\rm as} \quad t\rightarrow+\infty.
			\end{equation*}
			These will be called Delaunay solutions.
			\item[{\rm (ii)}] Assume that $0<R<+\infty$. If $u\in \mathcal{C}^{2\sigma}(B_R^*)$ is a positive singular solution to \eqref{ourintegralequationdual1ptlocal}, then there exists a Delaunay solution with a large period, denoted by $u_L$, such that 
			\begin{equation*}
				u(x)=u_{(0,L)}(x)(1+\mathrm{o}(1)) \quad {\rm as} \quad |x|\rightarrow0,
			\end{equation*}
			or 
			\begin{equation*}
				v(t)=v_{(0,L)}(t)(1+\mathrm{o}(1)) \quad {\rm as} \quad t\rightarrow+\infty,
			\end{equation*}
			where $L\gg1$ is sufficiently large.
		\end{itemize}        
	\end{propositionletter}
	
	In addition, writing Lemma~\ref{lm:refinedestimatesEF} into using Emden--Fowler change of variables, we can reformulate it as an improvement for the result above.         
	\begin{corollary}\label{cor:estimatesdelaunaySph}
		Let $\sigma \in(1,+\infty)$ and $n>2\sigma$.
		For any $L\gg1$ sufficiently large,
		there exist a sequence of periods $(L_j)\in \ell^\infty(\mathbb R_+)$ such that $\phi_{(0,L_j)}\in H^\sigma_L(\mathbb R^n\setminus\{0\})$ and a unique positive even solution $\bar{u}_{(0,L_j)}\in H_L^\sigma(\mathbb R^n\setminus\{0\})$ to \eqref{ourintegralequationdual1pt} satisfying
		\begin{equation*}
			\bar{u}_{(0,L_j)}(x)= \widehat{U}^+_{(0,L_j)}(x)+\phi_{(0,L_j)}(x),
		\end{equation*}
		and 
		\begin{equation*}
			\|\phi_{(0,L_j)}\|_{H_L^\sigma(\mathbb R^n\setminus\{0\})}\rightarrow 0 \quad {\rm as} \quad L\rightarrow+\infty,
		\end{equation*}
		where $\widehat{U}^+_{(0,L_j)}\in \mathcal{C}^{2\sigma}(\mathbb R^n\setminus\{0\})$ is the standard bubble tower solution given by \eqref{standardbubbletower}.                    
		Moreover, we have the following H\"older estimate
		\begin{equation}\label{holderestimatespherical}
			\|\phi_{(0,L_j)}\|_{\mathcal{C}^{2\sigma}(\mathbb R^n\setminus\{0\})} \lesssim e^{-\gamma_{\sigma}L(1+\xi)}
		\end{equation}
		for some $\alpha \in(0,1)$ and $\xi>0$ independent of $L\gg1$ large.
	\end{corollary}
	
	Based on the definition of a spherical solution in \eqref{sphericalsolutions} and \eqref{sphericalsolutionsEF}, we introduce the concept of a standard bubble tower solution.
	In addition, in order to have fast decay far from the singularity $(t \rightarrow-\infty)$, we will need only half a bubble tower. 
	This fact motivates the following definition
	
	\begin{definition}
		Let $\sigma \in(1,+\infty)$ and $n>2\sigma$.
		For any $L\gg1$ sufficiently large,
		let us define the following standard bubble tower solution
		\begin{itemize}
			\item [{\rm (i)}] $($Spherical coordinates$)$
			\begin{equation}\label{standardbubbletower}
				\widehat{U}_{(0,L_j)}(x):=\sum_{j\in\mathbb Z} U_{(0,L_j)}(x),
			\end{equation}
			and
			\begin{equation}\label{standardbubbletower+}
				\widehat{U}^+_{(0,L_j)}(x):=\sum_{j\in\mathbb N} U_{(0,L_j)}(x),
			\end{equation}
			where 
			\begin{equation}\label{componentsbubbletowers}
				U_{(0,L_j)}(x)=\left(\frac{\lambda_j}{\lambda_j^2+\left|x\right|^2}\right)^{\gamma_{\sigma}} \quad {\rm with} \quad \lambda_j=e^{-2jL} \quad {\rm for} \quad j\in \mathbb Z.
			\end{equation}    
			\item [{\rm (ii)}] $($Emden--Fowler coordinates$)$ 
			\begin{equation}\label{standardbubbletowerEF}
				\widehat{V}_{(0,L_j)}(t):=\sum_{j\in\mathbb Z}V_{(0,L_j)}(t),
			\end{equation}
			and
			\begin{equation}\label{standardbubbletowerEF+}
				\widehat{V}^+_{(0,L_j)}(t):=\sum_{j\in\mathbb N}V_{(0,L_j)}(t),
			\end{equation}
			where 
			\begin{equation}\label{componentsbubbletowersEF}
				V_{(0,L_j)}(t)=\cosh(t-L_j)^{\gamma_{\sigma}} \quad {\rm with} \quad L_j=2jL \quad {\rm for} \quad j\in \mathbb Z.
			\end{equation} 
		\end{itemize}
		These will be called the standard bubble tower solution.      
	\end{definition}

	As a consequence of Corollary~\ref{cor:estimatesdelaunay}, we will improve the last two results.
	Indeed, we show that near an isolated singularity, solutions are close to some bubble tower solution up to some controlled error.
	
	\begin{proposition}\label{prop:localbehavior}
		Let $\sigma\in(1,+\infty]$ with $n>2\sigma$.
		If $u\in \mathcal{C}^{2\sigma}(B_R^*)$ is a positive smooth singular solution to \eqref{ourintegralequationdual1ptlocal} with $R>0$, then there exists a sequence of periods $(L_j)\in \ell^{\infty}(\mathbb R_+)$ and a blow-up limit solution $\bar{u}_{(0,L_j)}\in \mathcal{C}^{2\sigma}(\mathbb R^n\setminus\{0\})$ to \eqref{ourintegralequationdual1pt} such that
		\begin{equation*}
			u(x)=\bar{u}_{(0,L_j)}(x)(1+\mathrm{o}(1)) \quad {\rm as} \quad |x|\rightarrow0.
		\end{equation*}
		More precisely, one has
		\begin{equation}\label{exactdelaunay}
			\bar{u}_{(0,L_j)}(x)=\widehat{U}_{(0,L_j)}(x)+\phi_{(0,L_j)}(x),
		\end{equation}
		where $\widehat{U}_{(0,L_j)}\in\mathcal{C}^{2\sigma}(\mathbb R^n\setminus\{0\})$ is the standard bubble tower solution in \eqref{standardbubbletower+} and $\phi_{(0,L_j)}\in H^{\sigma}(\mathbb R^n)$ satisfies
		\begin{equation}\label{holderEF}
			\|\phi_{(0,L_j)}\|_{\mathcal{C}^{2\sigma}(\mathbb R^n\setminus\{0\})} \lesssim e^{-\gamma_{\sigma}L(1+\xi)}
		\end{equation}
		for some $\alpha \in(0,1)$ and $\xi>0$ independent of $L\gg1$ large.
	\end{proposition}
	
	\subsection{Balanced configurations}
	Here we introduce a necessary set of compatibility conditions for the configuration parameters. 
	
	\begin{definition}
		Let $\sigma \in(1,+\infty)$, $n>2\sigma$, and $N\geqslant 2$.
		Given $L\gg1$ large enough, we will fix the vector $\boldsymbol{L}=(L_1, \ldots, L_N)\in\mathbb R^N_+$ to be the Delaunay parameters, which also are related to the neck sizes of each Delaunay solution. 
		They will be chosen (large enough) in the proof. They will satisfy the following conditions $|L_i-L| \lesssim 1$ for all $i\in \{1,\dots,N\}$.
		More precisely, they will be related by the vector $\boldsymbol{q}=(q_1,\dots,q_N)\in\mathbb R^N_+$,  which satisfy the following relations
		\begin{equation}\label{balancing0}
			q_i e^{-\gamma_{\sigma}L}=e^{-\gamma_{\sigma}L_i} \quad {\rm for} \quad i\in\{1,\dots,N\}.
		\end{equation}
	\end{definition}
	
	Next, we will give some explanation about the choice of parameters. Given the $N(n+2)$ balancing parameters $(\boldsymbol{q}^b,\boldsymbol{R}^b,\boldsymbol{\hat{a}}_0^{b})$ satisfying the balancing conditions \eqref{balancing1} and \eqref{balancing2}, we first choose $N(n+2)$ initial perturbation parameters $(\boldsymbol{q},\boldsymbol{R},\boldsymbol{\hat{a}}_0)$ which are close to the balancing parameters, {\it i.e.}, \eqref{balancing3} and \eqref{balancing4}.
	
	\begin{definition}\label{def:balancedparameters}
		Let $\sigma \in(1,+\infty)$, $n>2\sigma$, and $N\geqslant 2$.
		For any fixed nonnegative vector $\boldsymbol{q}^b=(q_1^b, \ldots, q_N^b)\in\mathbb R^N_+$, let us define the vector $(\mathbf{a}_0^b,\boldsymbol{R}^b)=(a_0^{i, b}, \dots, a_0^{N, b}, R^{1, b}, \dots R^{N, b})\in\mathbb R^{(n+1)N}$ to be determined by the following balancing conditions:
		\begin{equation}\label{balancing1}\tag{$\mathscr{B}_1$}
			q_i^b=A_2 \sum_{i^{\prime} \neq i} q_{i^{\prime}}^b(R^{i, b} R^{i^{\prime}, b})^{\gamma_{\sigma}}|x_i-x_{i^{\prime}}|^{-2\gamma_{\sigma}} \quad {\rm for} \quad i\in\{1,\dots,N\}.
		\end{equation}
		and
		\begin{equation}\label{balancing2}\tag{$\mathscr{B}_2$}
			\frac{a_0^{i, b}}{(\lambda_0^{i, b})^2}=-\frac{A_3}{A_1} \sum_{i^{\prime} \neq i} \frac{x_{i^{\prime}}-x_i}{|x_{i^{\prime}}-x_i|^{2 \gamma_{\sigma}+2}} \frac{q_{i^{\prime}}^b}{q_i^b}(R^{i, b} R^{i^{\prime}, b})^{\gamma_\sigma} \quad {\rm for} \quad i\in\{1,\dots,N\}
		\end{equation} 
		where $\lambda_0^{i, b}=R^{i, b} e^{-{L_i^b}}$, and the $L_i^b\in\mathbb R_+$ are defined from the $q_i^b\in\mathbb R_+$ by \eqref{balancing0} for each $i\in \{1,\dots,N\}$ and the constants $A_1, A_2>0, A_3<0$ are defined in \eqref{A1}, \eqref{A2}, and \eqref{A3}, respectively.
		We denote by $(\boldsymbol{q}^b,\mathbf{a}_0^b,\boldsymbol{R}^b)\in{\rm Bal}_{\sigma}(\Sigma)$ the set of balanced configurations.
	\end{definition}
	
	\begin{remark}
		We remark that it has been shown in \cite[Remark 3]{MR1712628} that for $\boldsymbol{q}:=(q_1^b, \ldots, q_N^b)\in\mathbb R^N_+$ in the positive octant, there exists a solution $\boldsymbol{R}^b=(R^{1, b},\dots, R^{N, b})$ to equation \eqref{balancing1}. 
		Once this is chosen, then we can use equation \eqref{balancing2} to determine $\mathbf{a}_0^b=(a_0^{1, b},\dots,a_0^{N, b})\in(\mathbb R^n_+)^N$.
		In other words, the set of balanced configurations is non-empty ${\rm Bal}_{\sigma}(\Sigma)\neq\varnothing$ for all $\sigma\in\mathbb R_+$.
	\end{remark}
	
	Although the meaning of these compatibility conditions will become apparent in the following sections, we have just seen that they are analogous to those of \cite{MR1425579} for the local case. The idea is that perturbations at the base level should be very close to those for a single bubble. This fact also shows, in particular, that even though our problem is nonlocal, very near the singularity, it presents a local behavior due to the strong influence of the underlying geometry.
	However, for the rest of the perturbation parameters, we must solve an infinite-dimensional system of equations.
	
	The last discussion motivates the definition below
	\begin{definition}\label{def:compactibleparameters}
		Let $\sigma \in(1,+\infty)$, $n>2\sigma$ and $N\geqslant 2$.
		We define the so-called configuration map $\Upsilon_{\rm conf}:\mathbb R^{(n+1)N}\rightarrow \mathbb R^{(n+2)N}$ which associates compatible moduli space parameters $(\boldsymbol{x},\boldsymbol{L})$ with configuration parameters $(\boldsymbol{q},\boldsymbol{a_0},\boldsymbol{R})$.
		We say that a set moduli space parameters $(\boldsymbol{x},\boldsymbol{L})\in\mathbb R^{(n+1)N}$ is compatible if its associated set of configuration parameters $(\boldsymbol{q},\boldsymbol{a_0},\boldsymbol{R})\in{\rm Bal}_\sigma(\Sigma)$ is balanced.
	\end{definition}
	
	\subsection{Admissible perturbation parameters}
	We also would like to introduce some perturbation parameters $R \in \mathbb{R}, a \in \mathbb{R}^n$, since each standard bubble has $n+1$ free parameters corresponding to scaling and translations, which is done for each bubble in the bubble tower independently. Thus, we will have an infinite-dimensional set of perturbations.
	
	\begin{definition}
		Let $\sigma \in(1,+\infty)$, $n>2\sigma$ and $N\geqslant 2$.
		For any $L\gg1$ sufficiently large and $\boldsymbol{L}=(L^1, \ldots, L^N)\in\ell^{\infty}(\mathbb R^N_+)$ and $\boldsymbol{R}=(R^1, \ldots, R^N)\in\ell^{\infty}(\mathbb R^N_+)$ such that \eqref{balancing0} holds,
		let us define  the full set of perturbation parameters $(\boldsymbol{a}_j,\boldsymbol{\lambda}_j)=(a_j^1,\dots,a_j^N,\lambda_j^1,\dots,\lambda_j^N)\in\ell^{\infty}(\mathbb R^{(n+1)N})$,
		where 
		\begin{equation*}
			\lambda_j^i=R_j^i e^{-{(1+2 j) L_i}} \quad  {\rm for} \quad i\in\{1,\dots,N\} \quad {\rm and} \quad j\in\mathbb Z.
		\end{equation*}
	\end{definition}
	
	We introduce the perturbation parameters we will use in the gluing technique
	\begin{definition}
		Let $\sigma \in(1,+\infty)$, $n>2\sigma$ and $N\geqslant 2$.
		Let $\boldsymbol{R}=(R^1, \dots, R^N)\in\mathbb R^N_+$ and $\boldsymbol{q}\in\mathbb R^N_+=(q_1,\dots, q_N)$ be such that 
		\begin{equation}\label{balancing3}
			|R^i-R^{i, b}| \lesssim 1 \quad {\rm and} \quad |q_i-q_i^b| \lesssim 1 \quad {\rm for} \quad i\in\{1,\dots,N\}.
		\end{equation}
		Also, we let $\boldsymbol{\lambda}^0_0=(\lambda_0^{1, 0}, \dots, \lambda_0^{N, 0})\in\mathbb R^N_+$ and $\boldsymbol{\hat{a}}_0^i=({\hat{a}}_0^1,\dots,{\hat{a}}_0^N)\in (\mathbb R^n)^N$
		be respectively given by
		\begin{equation*}
			\lambda_0^{i, 0}=R^i e^{-\frac{(1+2 j)}{2}L_i} \quad {\rm for} \quad i\in\{1,\dots,N\}
		\end{equation*}
		and
		\begin{equation*}
			\frac{a_0^{i, 0}}{(\lambda_0^{i, 0})^2}=\hat{a}_0^i \quad {\rm for} \quad i\in\{1,\dots,N\}
		\end{equation*}
		such that
		\begin{equation}\label{balancing4}
			|\hat{a}_0^i-\hat{a}_0^{i, b}| \lesssim 1,
		\end{equation}
		where
		\begin{equation*}
			\hat{a}_0^{i, b}=\frac{a_0^{i, b}}{(\lambda_0^{i, b})^2}.
		\end{equation*}
		For all $i\in\{1,\dots,N\}$ and $j\in\mathbb N$, let us  define the sequence of perturbation parameters $(\mathbf{a}_j,\boldsymbol{\lambda}_j)=(a^{1}, \dots, a^{N}, R^{1}, \dots, R^{N})\in\ell^\infty(\mathbb R^{(n+1)N})$ by
		\begin{equation}\label{balancing5}
			R_j^i=R^i(1+r_j^i) \quad {\rm and} \quad \frac{a_j^i}{(\lambda_j^i)^2}=\bar{a}_j^i=\hat{a}_0^i+\tilde{a}_j^i,
		\end{equation}
		where $(\tilde{\boldsymbol{a}}_j,\boldsymbol{r}_j)=(\tilde{a}^{i}, \dots, \tilde{a}^{N}, r^{1}, \dots r^{N})\in\ell^\infty(\mathbb R^{(n+1)N})$ satisfy
		\begin{equation}\label{balancing6}
			|r_j^i| \lesssim e^{-\tau t_j^i} \quad {\rm and} \quad |\tilde{a}_j^i| \lesssim e^{-\tau t_j^i} \quad {\rm for} \quad i\in\{1,\dots,N\}
		\end{equation}
		for some $\tau>0$, where $t_j^i=(1+2j) L_i$.
	\end{definition}
	
	\begin{definition}\label{def:admissibleparameters}
		Let $\sigma \in(1,+\infty)$, $n>2\sigma$ and $N\geqslant 2$.
		We define the so-called perturbation map $\Upsilon_{\rm per}:\mathbb R^{(n+2)N}\rightarrow \ell^\infty_{\tau}(\mathbb R^{(n+1)N})$ such that it associates balanced configurations with a sequence of admissible perturbations.
		A sequence of perturbation parameters $(\boldsymbol{a}_j,\boldsymbol{\lambda}_j)\in\ell^{\infty}(\mathbb R^{(n+1)N})$ or $(\boldsymbol{a}_j,\boldsymbol{r}_j)\in\ell^{\infty}(\mathbb R^{(n+1)N})$ is said to be admissible if the parameters satisfy 
		\begin{enumerate}
			\item[\namedlabel{itm:A0}{($\mathscr{A}_0$)}] For $j=0$, the configuration parameters $\Upsilon_{\rm per}^{-1}(\boldsymbol{a}_j,\boldsymbol{\lambda}_j)=(\mathbf q,\boldsymbol{a}_0,\boldsymbol{R})\in \mathbb R^{(n+2)N}$ is a balanced, that is, $(\mathbf q,\boldsymbol{a}_0,\boldsymbol{R})\in{\rm Bal}_{\sigma}(\Sigma)$; 
			\item[\namedlabel{itm:A1}{($\mathscr{A}_1$)}]
			For $j\geqslant 1$, the parameters $(\boldsymbol{a}_j,\boldsymbol{\lambda}_j)\in \ell^\infty(\mathbb R^{(N+1)n})$ satisfy the set of relations \eqref{balancing0},\eqref{balancing3},  \eqref{balancing4}, \eqref{balancing5}, and \eqref{balancing6}.
		\end{enumerate}  
		We denote by $(\boldsymbol{a}_j,\boldsymbol{\lambda}_j)\in{\rm Adm}_{\sigma}(\Sigma)$ the set of admissible configurations.
		Notice under \eqref{balancing5}, one can work indiscriminately with either parameter.
		In this fashion, we call $(\boldsymbol{0},\boldsymbol{1})\in\ell^{\infty}(\mathbb R^{(n+1)N})$ or $(\boldsymbol{0},\boldsymbol{0})\in\ell^{\infty}(\mathbb R^{(n+1)N})$ the trivial configurations.
	\end{definition}
	
	\subsection{Generalized Delaunay solutions}
	We now define a family of approximate solutions to the problem using the Delaunay solutions from the previous section. 
	From now on, we denote by $\chi:\mathbb R\rightarrow\mathbb R$ the cut-off function such that
	\begin{align*}
		\chi(x)=
		\begin{cases}
			1, & {\rm if} \ 0<|x| \leqslant \frac{1}{2}\\
			0, & {\rm if} \ \frac{1}{2}\leqslant|x| \leqslant 1\\
			\chi(x), & {\rm if} \ |x| \geqslant 1.
		\end{cases}
	\end{align*}
	
	First, one can always assume that all the balls $B_2(x_i)$ are disjoint since we may dilate the problem by some factor $\kappa>0$ that will change the set $\Sigma$ into $\kappa \Sigma$ and a function $u$ defined in $\mathbb{R}^n \setminus \Sigma$ into $\kappa^{-\gamma_{\sigma}} u({x}{\kappa}^{-1})$ defined in $\mathbb{R}^n \setminus \kappa \Sigma$.
	
	\begin{definition}
		Let $\sigma \in(1,+\infty)$ and $n>2\sigma$.
		For any $L\gg1$ sufficiently large and $\boldsymbol{L}=(L^1, \ldots, L^N)\in\ell^{\infty}(\mathbb R^N_+)$ and $\boldsymbol{R}=(R^1, \ldots, R^N)\in\ell^{\infty}(\mathbb R^N_+)$ such that \eqref{balancing0} holds let          
		$(\boldsymbol{a}_j,\boldsymbol{\lambda}_j)\in\ell^{\infty}(\mathbb R^{(n+1)N})$ be its associated perturbation parameters. 
		Fix                
		$x_i\in\Sigma$ for $i\in\{1,\dots,N\}$,
		let us define the following generalized bubble tower solution
		\begin{itemize}
			\item [{\rm (i)}] $($Spherical coordinates$)$
			\begin{equation}\label{standardbubbletowerdeformed}
				\widehat{U}_{(x_i,\boldsymbol{L},\boldsymbol{a}_j,\boldsymbol{\lambda}_j)}(x):=\sum_{j\in\mathbb Z} U_{(x_i,L_j^i,a_j^i,\lambda_j^i)}(x),
			\end{equation}
			and
			\begin{equation}            \label{standardbubbletower+deformed}
				\widehat{U}^+_{(x_i,\boldsymbol{L},\boldsymbol{a}_j,\boldsymbol{\lambda}_j)}(x):=\sum_{j\in\mathbb N} U_{(x_i,L_j^i,a_j^i,\lambda_j^i)}(x),
			\end{equation}
			where 
			\begin{equation}\label{componentsbubbletowersdeformed}
				U_{(x_i,L_j^i,a_j^i,\lambda_j^i)}(x)=\left(\frac{\lambda^i_j}{\lambda_j^2+|x-a^i_j-x_i|^2}\right)^{\gamma_{\sigma}}
			\end{equation}   
			with
			\begin{equation*}
				\lambda^i_j=R_j^ie^{-2jL^i_j} \quad {\rm for} \quad j\in \mathbb Z.
			\end{equation*}
			and
			\begin{equation*}
				L^i_j=L_i-jL_i+\ln R_j^i \quad {\rm for} \quad j\in \mathbb Z.
			\end{equation*}
			\item [{\rm (ii)}] $($Emden--Fowler coordinates$)$ 
			\begin{equation}\label{standardbubbletowerEFdeformed}
				\widehat{V}_{(x_i,\boldsymbol{L},\boldsymbol{a}_j,\boldsymbol{\lambda}_j)}(t):=\sum_{j\in\mathbb Z}V_{(x_i,L_j^i,a_j^i,\lambda_j^i)}(t),
			\end{equation}
			and
			\begin{equation}\label{standardbubbletowerEF+deformed}
				\widehat{V}_{(x_i,\boldsymbol{L},\boldsymbol{a}_j,\boldsymbol{\lambda}_j)}(t):=\sum_{j\in\mathbb N}V_{(x_i,L_j^i,a_j^i,\lambda_j^i)}(t),
			\end{equation}
			where 
			\begin{equation}\label{componentsbubbletowersEFdeformed}
				V_{(x_i,L_j^i,a_j^i,\lambda_j^i)}(t)=\cosh(-\ln|x-x_i-a_j^i|- L^i_j)^{\gamma_{\sigma}}.
			\end{equation} 		
		\end{itemize}
		These will be called the general $($half$)$ bubble tower solutions.      
	\end{definition}
	
	We also have the most basic definition of this section.
	We observe that although in the definition the solution is indexed by $(\boldsymbol{x},\boldsymbol{L},\boldsymbol{a}_j,\boldsymbol{\lambda}_j)$, one should recall that the configuration map from Definition~\ref{def:compactibleparameters} relates them  and by the perturbation map from Definition~\ref{def:compactibleparameters}, namely $(\boldsymbol{x},\boldsymbol{L})=(\boldsymbol{a}_j(\boldsymbol{q},\boldsymbol{a}_0,\boldsymbol{R}),\boldsymbol{\lambda}_j(\boldsymbol{q},\boldsymbol{a}_0,\boldsymbol{R}))$ and $(\boldsymbol{q},\boldsymbol{a}_0,\boldsymbol{R})=(\boldsymbol{q}(\boldsymbol{x},\boldsymbol{L}),\boldsymbol{a}_0(\boldsymbol{x},\boldsymbol{L}),\boldsymbol{R}(\boldsymbol{x},\boldsymbol{L}))$.
	
	\begin{definition}\label{def:approximatesolution}
		Let $\sigma \in(1,+\infty)$, $n>2\sigma$, and $N\geqslant 2$.
		For any $L\gg1$ sufficiently large and $\boldsymbol{L}=(L^1, \ldots, L^N)\in\ell^{\infty}(\mathbb R^N_+)$ and $\boldsymbol{R}=(R^1, \ldots, R^N)\in\ell^{\infty}(\mathbb R^N_+)$ such that \eqref{balancing0} holds let          
		$(\boldsymbol{a}_j,\boldsymbol{\lambda}_j)\in\ell^{\infty}(\mathbb R^{(n+1)N})$ be its associated perturbation parameters. 
		We define its associated solution $\bar{u}_{(\boldsymbol{x},\boldsymbol{L},\boldsymbol{a}_j,\boldsymbol{\lambda})}\in \mathcal{C}^{\infty}(\mathbb R^n\setminus \Sigma)$ as
		\begin{equation}\label{approximatesolution}
			\bar{u}_{(\boldsymbol{x},\boldsymbol{L},\boldsymbol{a}_j,\boldsymbol{\lambda}_j)}(x)=  \sum_{i=1}^N\bar{u}_{(x_i,\boldsymbol{L},\boldsymbol{a}_j,\boldsymbol{\lambda}_j)}(x).
		\end{equation}
		Here
		\begin{equation}\label{approximatesolutionsummands}
			\bar{u}_{(x_i,\boldsymbol{L},\boldsymbol{a}_j,\boldsymbol{\lambda}_j)}(x)=:
			\widehat{U}^+_{(x_i,\boldsymbol{L},\boldsymbol{a}_j,\boldsymbol{\lambda}_j)}(x)+\chi_i(x) \phi_{(x_i,\boldsymbol{L},\boldsymbol{a}_j,\boldsymbol{\lambda}_j)}(x),
		\end{equation}
		where $\widehat{U}^+_{(x_i,\boldsymbol{L},\boldsymbol{a}_j,\boldsymbol{\lambda}_j)}\in\mathcal{C}^{2\sigma}(\mathbb R^n\setminus\Sigma)$ is the generalized bubble tower solution given by \eqref{standardbubbletower+deformed} and
		\begin{equation}\label{errorfunctionEF}
			\phi_{(x_i,\boldsymbol{L},\boldsymbol{a}_j,\boldsymbol{\lambda}_j)}(x)=\phi_{(\boldsymbol{L},\boldsymbol{a}_j,\boldsymbol{\lambda}_j)}(x-x_i) \quad {\rm and} \quad \chi_i(x)=\chi(x-x_i)\quad {\rm for \ all} \quad i\in\{1,\dots,N\}
		\end{equation}
		with $\phi_{(\boldsymbol{L},\boldsymbol{a}_j,\boldsymbol{\lambda}_j)}$ the error function from Lemma~\ref{lm:estimatesdelaunay}.    
		We say that $\bar{u}_{(\boldsymbol{x},\boldsymbol{L},\boldsymbol{a}_j,\boldsymbol{\lambda}_j)}\in \mathcal{C}^{\infty}(\mathbb R^n\setminus \Sigma)$ is an approximate solution to \eqref{ourintegralequation}, denote by $\bar{u}_{(\boldsymbol{x},\boldsymbol{L},\boldsymbol{a}_j,\boldsymbol{\lambda}_j)}\in {\rm Apx}_{\sigma}(\Sigma)$, whenever $(\boldsymbol{a}_j,\boldsymbol{\lambda}_j)\in{\rm Adm}_{\sigma}(\Sigma)$. 
		We then define the so-called perturbation map $\Upsilon_{\rm sol}: \ell^\infty_{\tau}(\mathbb R^{(n+1)N})\rightarrow \mathcal{C}^{\infty}(\mathbb R^n\setminus \Sigma)$ such that it associates balanced configurations with sequences of admissible perturbations.
	\end{definition}
	
	\subsection{Normalized approximate kernels}          
	In this subsection, we will use the aforementioned parameters to define a family of projections on the (normalized) approximate kernels. 
	At least for low Fourier eigenmodes, this family is entirely constructed by varying the parameters in the approximate solution.
	
	\begin{definition}\label{def:kernels}  
		Let $\sigma \in(1,+\infty)$, $n>2\sigma$ and $N\geqslant 2$.
		Assume that $(\boldsymbol{a}_j,\boldsymbol{\lambda}_j)\in{\rm Adm}_{\sigma}(\Sigma)$ is an admissible configuration as in Definition~\ref{def:balancedparameters} with $\bar{u}_{(\boldsymbol{x},\boldsymbol{L},\boldsymbol{a}_j,\boldsymbol{\lambda}_j)}\in{\rm Apx}_{\sigma}(\Sigma)$ their associated approximate solution as in Definition~\ref{def:approximatesolution}.  
		\begin{enumerate}
			\item[{\rm (a)}]  Let us introduce some notation of normalized approximate kernels.
			\begin{itemize}
				\item[{\rm (i)}] If $\ell=0$, we set 
				\begin{equation*}
					Z_{j, 0}^i(\boldsymbol{a}_j,\boldsymbol{\lambda}_j)=\partial_{r_j^i} U_{(x_i,L^i_j,\lambda_j^i,a_j^i)};
				\end{equation*}
				for the zero-frequency Fourier eigenmodes.
				\item[{\rm (ii)}] If $\ell\in\{1,\dots,n\}$, we set
				\begin{equation*}
					Z_{j, \ell}^i(\boldsymbol{a}_j,\boldsymbol{\lambda}_j)=\lambda_j^i \partial_{a_{j, \ell}^i} U_{(x_i,L^i_j,\lambda_j^i,a_j^i)}=-\lambda_j^i \partial_{x_\ell} U_{(x_i,L^i_j,\lambda_j^i,a_j^i)}.
				\end{equation*}
				for the low-frequency Fourier eigenmodes.
			\end{itemize}
			We denote by $\{Z_{j, \ell}^i(\boldsymbol{a}_j,\boldsymbol{\lambda}_j)\}_{(i,j,\ell)\in\mathcal{I}_\infty}\subset \mathcal{C}^0(\mathbb R^n\setminus\Sigma)$ the family of normalized approximate kernels.
			\item[{\rm (b)}] Let us introduce some notation of normalized approximate cokernels.
			\begin{itemize}
				\item[{\rm (i)}] If $\ell=0$, we set 
				\begin{equation*}
					\overline{Z}_{j, 0}^i(\boldsymbol{a}_j,\boldsymbol{\lambda}_j)=f^{\prime}_\sigma(U_{(x_i,L^i_j,\lambda_j^i,a_j^i)})Z_{j, 0}^i(\boldsymbol{a}_j,\boldsymbol{\lambda}_j);
				\end{equation*}
				for the zero-frequency Fourier eigenmodes.
				\item[{\rm (ii)}] If $\ell\in\{1,\dots,n\}$, we set
				\begin{equation*}
					\overline{Z}_{j, \ell}^i(\boldsymbol{a}_j,\boldsymbol{\lambda}_j)=f^{\prime}_\sigma(U_{(x_i,L^i_j,\lambda_j^i,a_j^i)})Z_{j, \ell}^i(\boldsymbol{a}_j,\boldsymbol{\lambda}_j)
				\end{equation*}
				for the low-frequency Fourier eigenmodes.
			\end{itemize}
			We denote by $\{\overline{Z}_{j, \ell}^i(\boldsymbol{a}_j,\boldsymbol{\lambda}_j)\}_{(i,j,\ell)\in\mathcal{I}_\infty}\subset \mathcal{C}^0(\mathbb R^n\setminus\Sigma)$ the family of normalized approximate cokernels.
		\end{enumerate}        
	\end{definition}
	
	These normalized kernels satisfy some orthogonality conditions, which will be important in applying a finite-dimensional reduction.         
	\begin{lemma}\label{lm:orthogonalityconditions}
		Let $\sigma \in(1,+\infty)$, $n>2\sigma$ and $N\geqslant 2$.
		Assume that $(\boldsymbol{a}_j,\boldsymbol{\lambda}_j)\in{\rm Adm}_{\sigma}(\Sigma)$ is an admissible configuration as in Definition~\ref{def:balancedparameters} with $\bar{u}_{(\boldsymbol{x},\boldsymbol{L},\boldsymbol{a}_j,\boldsymbol{\lambda}_j)}\in{\rm Apx}_{\sigma}(\Sigma)$ their associated approximate solution as in Definition~\ref{def:approximatesolution}.  
		Then, one has
		\begin{itemize}
			\item[{\rm (i)}] If $\ell\in\{1,\dots,n\}$, then 
			\begin{align}\label{orthogonalitycond1}
				\int_{\mathbb{R}^n}\overline{Z}_{j, \ell}^i(\boldsymbol{a}_j,\boldsymbol{\lambda}_j)\overline{Z}_{j^{\prime}, \ell^{\prime}}^i(\boldsymbol{a}_j,\boldsymbol{\lambda}_j) \ud x=\frac{4(n-2 \sigma)^2}{n}\left(\delta_{\ell, \ell^{\prime}}+\mathrm{o}(1)\right) e^{-(\gamma_{\sigma}+1)|t_j^i-t_{j^{\prime}}^i|},
			\end{align} 
			where $\delta_{\ell, \ell^{\prime}}$ is Kronecker's delta;
			\item[{\rm (ii)}] If $\ell=0$, then 
			\begin{align}\label{orthogonalitycond2}
				\int_{\mathbb{R}^n}\overline{Z}_{j, 0}^i (\boldsymbol{a}_j,\boldsymbol{\lambda}_j)Z_{j^{\prime}, 0}^i(\boldsymbol{a}_j,\boldsymbol{\lambda}_j) \ud x=C_0(1+\mathrm{o}(1)) e^{-\gamma_{\sigma}|t_j^i-t_{j^{\prime}}^i|}
			\end{align}
			for some $C_0>0$. 
		\end{itemize}
	\end{lemma}
	
	\begin{proof}
		Initially, let us observe that by Lemma \ref{lm:nondegeneracy},  the set of bounded solutions to 
		\begin{equation*}
			\phi-(-\Delta)^{-\sigma}(f^{\prime}_\sigma(U_{(x_i,L^i_j,\lambda_j^i,a_j^i)})\phi)=0 \quad {\rm in} \quad \mathbb{R}^n
		\end{equation*}
		is spanned by $\{\overline{Z}_{j, 0}^i(\boldsymbol{a}_j,\boldsymbol{\lambda}_j), \dots, \overline{Z}_{j, n}^i(\boldsymbol{a}_j,\boldsymbol{\lambda}_j)\}$ for any $i\in\{1,\dots,N\}$ and $j\in\mathbb N$.
		
		Without loss of generality, assume in the following that $x_i=0$. 
		For $\ell=0$, we will repeatedly use the following estimates
		\begin{equation}\label{kernelestimate1}
			\left|Z_{j, 0}^i(\boldsymbol{a}_j,\boldsymbol{\lambda}_j)(x)\right| \lesssim
			\begin{cases}
				|x|^{-\gamma_{\sigma}}V_{(x_i,L_i,a_j^i, \lambda^i_j)}(-\ln |x|), & {\rm if} \ |x| \leqslant 1, \\ 
				|x|^{-2\gamma_{\sigma}}(\lambda_j^i)^{\gamma_{\sigma}}, & {\rm if} \ |x| \geqslant 1.
			\end{cases}
		\end{equation}
		In addition, we have also have
		\begin{equation*}
			Z_{j, \ell}^i(\boldsymbol{a}_j,\boldsymbol{\lambda}_j)(x)=2\gamma_{\sigma}V_{a_j^i, L_i, t_i}(-\ln |x|)^{\frac{\gamma_{\sigma}+1}{\gamma_{\sigma}}}|x-a_j^i-x_i|^{-\gamma_{\sigma}-1}\left(x-a_j^i-x_i\right)_\ell,
		\end{equation*}
		Then, after recentering at $x_i=0$, it is easy to see that the following orthogonality conditions \eqref{orthogonalitycond1} are in force.  
		Similar estimates also hold true for $\ell=0$; the orthogonality condition in \eqref{orthogonalitycond2} is also satisfied.  
		
		The lemma is then proved.
	\end{proof}
	
	\subsection{Weighted functional spaces}
	It is convenient to define the suitable function spaces on which we will run our perturbation technique.
	
	\begin{definition}\label{def:fucntionspaces}
		Let $\alpha\in(0,1)$ and $\zeta_1,\zeta_2\in\mathbb R$ such that $\zeta_1<0$ and $\zeta_2>0$.
		We set the weighted norm
		\begin{equation*}
			\|u\|_{\mathcal{C}^{\alpha}_{\zeta_1, \zeta_2}(\mathbb R^n\setminus\Sigma)}=\|\operatorname{dist}(x,\Sigma)^{-\zeta_1} u\|_{\mathcal{C}^\alpha(B_1(\Sigma))}+\||x|^{-\zeta_2} u\|_{\mathcal{C}^{\alpha}(\mathbb{R}^n \setminus B_1(\Sigma))}.
		\end{equation*}
		In other words, one that $u\in\mathcal{C}^{\alpha}_{\zeta_1, \zeta_2}(\mathbb R^n\setminus\Sigma)$ if and only if 
		\begin{itemize}
			\item[{\rm (i)}] $($Near the singular set$)$ it is bounded by a constant times $|x-x_i|^{\zeta_1}$ and has its $\ell$-th order partial derivatives bounded by a constant times $|x-x_i|^{\zeta_1-\ell}$ for $\ell \leqslant \alpha$ near each singular point $x_i\in\Sigma$.
			\item[{\rm (ii)}] $($Away from the singular set$)$ it is bounded by $|x|^{\zeta_2}$ and has its $\ell$-th order partial derivatives bounded by a constant times $|x|^{\zeta_2-\ell}$ for $\ell \leqslant \alpha$.
		\end{itemize}
		Note that we are implicitly assuming that $0 \in \Sigma$, in order to simplify the notation.
	\end{definition}
	
	\begin{definition}\label{def:suitablespaces}
		Let $\sigma \in(1,+\infty)$, $n>2\sigma$ and $N\geqslant 2$. 
		We define the following weighted norms
		\begin{equation}\label{weightednorm1}
			\|u\|_{\mathcal{C}_{*,\tau}(\mathbb{R}^n\setminus\Sigma)}=\|u\|_{\mathcal{C}_{\min\{\zeta_1,-\gamma_{\sigma}+\tau\},-n-2\sigma}(\mathbb R)}
		\end{equation}
		and 
		\begin{equation}\label{weightednorm2}
			\|h\|_{\mathcal{C}_{**,\tau}(\mathbb{R}^n\setminus\Sigma)}=\|h\|_{\mathcal{C}_{n+\tau,-n+2\sigma}(\mathbb R)},
		\end{equation}
		where
		\begin{equation}\label{weightassumption}
			-\gamma_\sigma<\zeta_1<\min\left\{-\gamma_{\sigma}+2 \sigma, 0\right\}.
		\end{equation}
		Here $0<\tau\ll1$ small enough is given in the definition of the perturbation parameters \eqref{balancing5} and \eqref{balancing6}.   
		In this fashion, we denote by $\mathcal{C}_{*,\tau}(\mathbb{R}^n\setminus\Sigma)$ and $\mathcal{C}_{**,\tau}(\mathbb{R}^n\setminus\Sigma)$ the corresponding weighted Hölder spaces.         
		
	\end{definition}
	
	Let us make some observations regarding the last definition.
	\begin{remark}
		We emphasize that to simplify the notation, many times we will ignore the small perturbation and just the weight near the singular set as dist $(x, \Sigma)^{-\zeta_1}, \operatorname{dist}(x, \Sigma)^{2 \sigma-\zeta_1}$, respectively.
		The weights in Definition~\ref{def:suitablespaces} are suitably chosen to guarantee the invertibility and Fredholmness of the linearized operator around approximate solutions on weighted H\"older spaces; this will be clear in the reduction method we apply in the remaining subsections.
	\end{remark}
	
	\subsection{Perturbation of the approximate solution}
	This subsection is devoted to performing a perturbation method based on the approximated solution, which requires linearizing \eqref{ourintegralequationdual} around the approximate solution \eqref{approximatesolution} and estimating both the weighted norm of the right-inverse for the linearized operator given by and the associated remainder error.
	We emphasize that the balancing formulas and the orthogonality conditions for the normalized kernels discussed above will be building blocks of our construction.
	
	Let us explain our strategy in more detail.
	First, we consider the nonlinear operator defined
	$\mathscr{N}_\sigma: \mathcal{C}^0(\mathbb R^n\setminus\Sigma)\rightarrow \mathcal{C}^{2\sigma}(\mathbb R)$ given by 
	\begin{equation}\label{nonlinearoperator}
		\mathscr{N}_{\sigma}({u})={u}-(-\Delta)^{-\sigma}(f_\sigma\circ {u}).
	\end{equation}
	Notice that \eqref{ourintegralequationdual} can be reformulated as
	\begin{equation*}
		\mathscr{N}_{\sigma}({u})=0 \quad {\rm in} \quad \mathbb R^n\setminus\Sigma.
	\end{equation*}
	Next, by linearizing this operator around the approximate solution, we find a linear operator $\mathscr{L}_\sigma[\bar{u}_{(\boldsymbol{x},\boldsymbol{L},\boldsymbol{a}_j,\boldsymbol{\lambda}_j)}]: \mathcal{C}^{0}(\mathbb{R}^n\setminus\Sigma)\rightarrow \mathcal{C}^{2\sigma}(\mathbb{R}^n\setminus\Sigma)$ given by
	\begin{equation}
		\mathscr{L}_{\sigma}[\bar{u}_{(\boldsymbol{x},\boldsymbol{L},\boldsymbol{a}_j,\boldsymbol{\lambda}_j)}](\phi)=\phi-(-\Delta)^{-\sigma}(f^{\prime}_\sigma\circ \bar{u}_{(\boldsymbol{x},\boldsymbol{L},\boldsymbol{a}_j,\boldsymbol{\lambda}_j)})\phi.
	\end{equation}
	For the sake of simplicity, let us denote 
	\begin{equation*}
		\mathscr{L}_{\sigma}[\bar{u}_{(\boldsymbol{x},\boldsymbol{L},\boldsymbol{a}_j,\boldsymbol{\lambda}_j)}]:=\mathscr{L}_{\sigma}(\boldsymbol{x},\boldsymbol{L},\boldsymbol{a}_j,\boldsymbol{\lambda}_j).
	\end{equation*}
	
	\subsubsection{Quantitative estimates}
	Our first estimate concerns the nonlinear operator defined as \eqref{nonlinearoperator} applied to the approximate solution $\bar{u}_{(\boldsymbol{x},\boldsymbol{L},\boldsymbol{a}_j,\boldsymbol{\lambda}_j)}\in \mathcal{C}^{\infty}(\mathbb R^n\setminus \Sigma)$ given by \eqref{approximatesolution}, namely
	\begin{align*}
		\mathcal{N}_\sigma(\boldsymbol{x},\boldsymbol{L},\boldsymbol{a}_j,\boldsymbol{\lambda}_j)&:=\mathscr{N}_{\sigma}(\bar{u}_{(\boldsymbol{x},\boldsymbol{L},\boldsymbol{a}_j,\boldsymbol{\lambda}_j)})\\
		&=\bar{u}_{(\boldsymbol{x},\boldsymbol{L},\boldsymbol{a}_j,\boldsymbol{\lambda}_j)}-(-\Delta)^{-\sigma}(f_\sigma\circ \bar{u}_{(\boldsymbol{x},\boldsymbol{L},\boldsymbol{a}_j,\boldsymbol{\lambda}_j)}).
	\end{align*}            
	We emphasize that we must suitably choose the weighted norm in \eqref{weightednorm2} so that our following estimates have the correct decay.
	
	\begin{lemma}\label{lm:quantitativeestimates}
		Let $\sigma \in(1,+\infty)$, $n>2\sigma$ and $N\geqslant 2$.
		Assume that $(\boldsymbol{a}_j,\boldsymbol{\lambda}_j)\in{\rm Adm}_{\sigma}(\Sigma)$ is an admissible configuration as in Definition~\ref{def:balancedparameters} with $\bar{u}_{(\boldsymbol{x},\boldsymbol{L},\boldsymbol{a}_j,\boldsymbol{\lambda}_j)}\in{\rm Apx}_{\sigma}(\Sigma)$ their associated approximate solution as in Definition~\ref{def:approximatesolution}. 
		Then, there exists a weight $\zeta_1<0$ satisfying \eqref{weightassumption} such that
		\begin{equation}\label{estimateaproxximatesolution}
			\|\mathcal{N}_\sigma(\boldsymbol{x},\boldsymbol{L},\boldsymbol{a}_j,\boldsymbol{\lambda}_j)\|_{\mathcal{C}_{**,\tau}(\mathbb{R}^n\setminus\Sigma)} \lesssim e^{-\gamma_{\sigma}L(1+\xi)}
		\end{equation}
		for some $\xi>0$ uniformly on $L\gg1$ large.
	\end{lemma}
	
	\begin{proof}
		For the sake of simplicity, we shall prove the estimate in \eqref{estimateaproxximatesolution} for the $L^{\infty}-$norm.
		Namely, we need to quantitatively estimate the term $|\mathscr{N}_\sigma(\bar{u}_{(\boldsymbol{x}, \boldsymbol{L}, \boldsymbol{0},\boldsymbol{1})})|$ and then a applying a classical perturbation technique.
		
		The rest of the proof will be divided into two cases.
		
		\noindent{\bf Case 1:} $(\boldsymbol{a}_j,\boldsymbol{\lambda}_j)=(\boldsymbol{0},\boldsymbol{1})$ for all ${j\in\mathbb N}$.
		
		In this case, the approximate solution $\bar{u}_{(\boldsymbol{x}, \boldsymbol{L}, \boldsymbol{0},\boldsymbol{1})}\in \mathcal{C}^{\infty}(\mathbb R^n\setminus\Sigma)$ is given by
		\begin{equation*}
			\bar{u}_{(\boldsymbol{x}, \boldsymbol{L}, \boldsymbol{0},\boldsymbol{1})}(x)=\sum_{i=1}^N\left[\sum_{j\in\mathbb N} U_{(x_i, L_j^i,\boldsymbol{0},\lambda_j^i)}(x)+\chi_i(x) \phi_i(x-x_i)\right],
		\end{equation*}
		where
		\begin{equation*}
			U_{(x_i, L_j^i,\boldsymbol{0},\lambda_j^i)}(x):=\left(\frac{\lambda_j^i}{|\lambda_j^i|^2+|x-x_i|^2}\right)^{\gamma_{\sigma}}.
		\end{equation*} 
		
		Without loss of generality, assume $x_1=0$. Before we prove the estimate of $|\mathscr{N}_\sigma(\bar{u}_{(\boldsymbol{x}, \boldsymbol{L}, \boldsymbol{0},\boldsymbol{1})})|$, we first prove the following claim:
		
		\noindent{\bf Claim 1:} The following estimate holds
		\begin{align*}
			|\mathcal{D}_{\sigma}(\bar{u}_{(\boldsymbol{x}, \boldsymbol{L}, \boldsymbol{0},\boldsymbol{1})})|\lesssim
			\begin{cases}
				|x-x_i|^{\zeta_1-2\sigma} e^{-\gamma_\sigma L(1+\xi)}, \ &{\rm if} \ 0<\ud(x, \Sigma) <\frac{1}{2},\\
				e^{-\gamma_\sigma L(1+\xi)}, \ &{\rm if} \ \frac{1}{2} \leqslant\ud(x, \Sigma) < 1,\\
				|x|^{-(n+2 \sigma)} e^{-\gamma_\sigma L(1+\xi)}, \ &{\rm if} \ \ud(x, \Sigma) \geqslant 1,
			\end{cases}
		\end{align*}
		where
		\begin{align}\label{differenceterm}
			\mathcal{D}_{\sigma}(\bar{u}_{(\boldsymbol{x}, \boldsymbol{L}, \boldsymbol{0},\boldsymbol{1})}):=\sum_{i=1}^N \sum_{j\in\mathbb N}f_{\sigma}(U_{(x_i, L_j^i,\boldsymbol{0},\lambda_j^i)})-f_{\sigma}(\bar{u}_{(\boldsymbol{x}, \boldsymbol{L}, \boldsymbol{0},\boldsymbol{1})}).
		\end{align}
		
		\noindent As a matter of fact, we proceed by a direct estimate in terms in the asymptotic behavior of the bubble tower solution. 
		Without loss of generality, assume $x_1=0$. The proof will be divided into three steps: the exterior, transition, and interior region, respectively.
		
		\noindent{\bf Step 1:} If $\ud(x, \Sigma) \geqslant 1$, then $|\mathcal{D}_{\sigma}(\bar{u}_{(\boldsymbol{x}, \boldsymbol{L}, \boldsymbol{0},\boldsymbol{1})})|\lesssim|x|^{-(n+2 \sigma)} e^{-\gamma_\sigma L(1+\xi)}$.
		
		\noindent 
		In this region, we notice $\chi_i(x)=0$ for all $i\in\{2,\dots,N\}$ when $\ud(x,\Sigma)\geqslant 1$. 
		Next, using that
		\begin{equation*}
			|U_{(x_i, L_j^i,\boldsymbol{0},\lambda_j^i)}(x)| \sim\left(\lambda_j^i\right)^{\gamma_{\sigma}}|x|^{-(n-2\sigma)} \quad {\rm as} \quad |x|\rightarrow+\infty
		\end{equation*}
		and recalling the relation in \eqref{balancing0}, we have
		\begin{align*}
			\left|\mathcal{D}_{\sigma}(\bar{u}_{(\boldsymbol{x}, \boldsymbol{L}, \boldsymbol{0},\boldsymbol{1})})\right|&=c_{n,\sigma}\left|\left(\sum_{i=1}^N \sum_{j\in\mathbb N}U_{(x_i, L_j^i,\boldsymbol{0},\lambda_j^i)}\right)^\frac{n+2\sigma}{n-2\sigma}-\sum_{i=1}^N \sum_{j\in\mathbb N}\left(U_{(x_i, L_j^i,\boldsymbol{0},\lambda_j^i)}\right)^\frac{n+2\sigma}{n-2\sigma}\right|\\
			&\lesssim\left(e^{-\frac{(n-2 \sigma) L}{2}}|x|^{-(n-2 \sigma)}\right)^\frac{n+2\sigma}{n-2\sigma}\\
			&\lesssim e^{-\frac{(n+2 \sigma) L}{2}}|x|^{-(n+2 \sigma)},            
		\end{align*}
		which finishes the proof of the first step.
		
		\noindent{\bf Step 2:} If $\frac{1}{2} \leqslant|x| \leqslant 1$, then $|\mathcal{D}_{\sigma}(\bar{u}_{(\boldsymbol{x}, \boldsymbol{L}, \boldsymbol{0},\boldsymbol{1})})|\lesssim e^{-\gamma_\sigma L(1+\xi)}$ for some $\xi>0$.
		
		\noindent In this case, it is easy to verify the estimate 
		\begin{equation*}
			|\mathcal{D}_{\sigma}(\bar{u}_{(\boldsymbol{x}, \boldsymbol{L}, \boldsymbol{0},\boldsymbol{1})})| \lesssim e^{-\gamma_\sigma L(1+\xi)}
		\end{equation*}
		for some $\xi>0$.
		
		\noindent{\bf Step 3:} If $0<|x| \leqslant \frac{1}{2}$, then $ |\mathcal{D}_{\sigma}(\bar{u}_{(\boldsymbol{x}, \boldsymbol{L}, \boldsymbol{0},\boldsymbol{1})})|\lesssim |x-x_i|^{\zeta_1-2\sigma} e^{-\gamma_\sigma L(1+\xi)}$ for some $\xi>0$.
		
		\noindent Notice that $\chi_1(x)\equiv1$ and $\chi_i(x)0$ for $i \in\{2,\dots,N\}$. 
		By definition, it follows that 
		\begin{equation*}
			\bar{u}_{(\boldsymbol{x}, \boldsymbol{L}, \boldsymbol{0},\boldsymbol{1})}= {\widehat{U}}_{(x_1,L_j^1,0,\lambda_j^1)}-\left(1-\chi_1\right) \phi_1+\sum_{i =2}^N\left(\sum_{j\in\mathbb N} U_{(x_i, L_j^i,\boldsymbol{0},\lambda_j^i)}+\chi_i \phi_i\right)-\sum_{j\in\mathbb Z\setminus\mathbb N}U_{(x_1,L_j^1,0,\lambda_j^1)}.
		\end{equation*}
		Hence, by an easy computation, we obtain
		\begin{align}\label{quantitativest1}
			\nonumber
			\left|\mathcal{D}_{\sigma}(\bar{u}_{(\boldsymbol{x}, \boldsymbol{L}, \boldsymbol{0},\boldsymbol{1})})\right|
			&\lesssim \left|\sum_{j\in\mathbb N}(U_{(x_1,L_j^1,0,\lambda_j^1)})^\frac{n+2\sigma}{n-2\sigma}-\left(\sum_{j\in\mathbb N}U_{(x_1,L_j^1,0,\lambda_j^1)}+\mathcal{O}(e^{-\gamma_\sigma(1+\xi)})\right)^{\frac{n+2\sigma}{n-2\sigma}}\right|+\mathcal{O}(e^{-\gamma_\sigma(1+\xi)})\\
			&\lesssim \sum_{j\in\mathbb N}(U_{(x_1,L_j^1,0,\lambda_j^1)})^\frac{4\sigma}{n-2\sigma}e^{-\gamma_\sigma}+\mathcal{O}(e^{-\gamma_\sigma(1+\xi)}))\\\nonumber
			&\lesssim |x|^{-2\sigma}\sum_{j\in\mathbb N}V_{(x_1,L_j^1,0,\lambda_j^1)}^{\frac{4\sigma}{n-2\sigma}}e^{-\gamma_\sigma L}+\mathcal{O}(e^{-\gamma_\sigma(1+\xi)}),
		\end{align}
		where 
		\begin{equation*}
			V_{(x_1,L_j^1,0,\lambda_j^1)}(-\ln|x|):=V_{\rm sph}(-\ln|x|-L_1-2j L_1)
		\end{equation*}
		and we recall that the spherical solution $v_{\rm sph}$ is defined as \eqref{sphericalsolutionsEF}.
		
		\noindent Furthermore, it is straightforward to see that when $0<|x|\leqslant\frac{1}{2}$ there exists $\xi>0$ and $\zeta_1<0$ satisfying
		\begin{equation}\label{quantitativest2}
			|x|^{-\zeta_1}\left(\sum_{j\in\mathbb Z} V_{(x_1,L_j^1,0,\lambda_j^1)}(-\ln|x|)\right)^{\frac{4\sigma}{n-2\sigma}} \lesssim e^{-\xi L_1}.
		\end{equation}
		Indeed, if $-\infty<t<{L_1}$ there exists $C_1>0$ such that $|x|\leqslant C_1$ and 
		\begin{equation}\label{quantitativest3}
			\sum_{j\in\mathbb Z} V_{(x_1,L_j^1,0,\lambda_j^1)}(-\ln|x|) \leqslant C_1 e^{-\gamma_\sigma L_1}.
		\end{equation}
		Also, if ${L_1}\leqslant t < +\infty$, there exists $C_2>0$ such that $|x| \leqslant C_2 e^{-{L_1}/{2}}$ and 
		\begin{equation}\label{quantitativest4}
			\sum_{j\in\mathbb Z} V_{(x_1,L_j^1,0,\lambda_j^1)}(-\ln|x|) \leqslant C_2.
		\end{equation}
		Finally, combining \eqref{quantitativest1} and \eqref{quantitativest2} implies 
		\begin{align*}
			\left|\mathcal{D}_{\sigma}(\bar{u}_{(\boldsymbol{x}, \boldsymbol{L}, \boldsymbol{0},\boldsymbol{1})})\right|\lesssim |x|^{\zeta_1-2\sigma}e^{-\gamma_\sigma(1+\xi)},
		\end{align*}
		which finishes the proof of the first claim.
		
		We now proceed to the proof of our preliminary estimate.
		
		\noindent{\bf Claim 2:} The following estimates holds
		\begin{align*}
			|\mathscr{N}_\sigma(\bar{u}_{(\boldsymbol{x},\boldsymbol{L},\boldsymbol{a}_j,\boldsymbol{\lambda}_j)})|\lesssim
			\begin{cases}
				|x-x_i|^{\min\{\zeta_1-\tau,-\gamma_\sigma+\tau\}} e^{-\gamma_\sigma L(1+\xi)}, \ &{\rm if} \ 0<\ud(x, \Sigma) <\frac{1}{2},\\
				e^{-\gamma_\sigma L(1+\xi)}, \ &{\rm if} \ \frac{1}{2} \leqslant\ud(x, \Sigma) < 1,\\
				|x|^{2\sigma-n} e^{-\gamma_\sigma L(1+\xi)}, \ &{\rm if} \ \ud(x, \Sigma) \geqslant 1.
			\end{cases}
		\end{align*}
		
		\noindent 
		As before, the proof will be divided into three steps as follows:
		
		\noindent{\bf Step 1:} If $\ud(x, \Sigma) \geqslant 1$, then $|\mathscr{N}_\sigma(\bar{u}_{(\boldsymbol{x}, \boldsymbol{L}, \boldsymbol{0},\boldsymbol{1})})|\lesssim|x|^{2\sigma-n} e^{-\gamma_\sigma L(1+\xi)}$.
		
		\noindent  Notice that $\chi_i(x)\equiv0$ for all $i\in\{1,\dots,N\}$ and $x\in\mathbb R^n\setminus\Sigma$ such that $\ud(x, \Sigma) \geqslant 1$.
		From this, we get 
		\begin{align*}
			\mathscr{N}_{\sigma}(\bar{u}_{(\boldsymbol{x}, \boldsymbol{L}, \boldsymbol{0},\boldsymbol{1})})&= \bar{u}_{(\boldsymbol{x}, \boldsymbol{L}, \boldsymbol{0},\boldsymbol{1})}-(-\Delta)^{-\sigma}(f_\sigma(\bar{u}_{(\boldsymbol{x}, \boldsymbol{L}, \boldsymbol{0},\boldsymbol{1})}))\\
			& =\sum_{i=1}^N\sum_{j\in\mathbb N}U_{(x_i, L_j^i,\boldsymbol{0},\lambda_j^i)}(x)-(-\Delta)^{-\sigma}f_{\sigma}\left(\bar{u}_{(\boldsymbol{x}, \boldsymbol{L}, \boldsymbol{0},\boldsymbol{1})}\right)\\
			& =\int_{\mathbb R^n}\left(\sum_{i=1}^N \sum_{j\in\mathbb N}f_{\sigma}(U_{(x_i, L_j^i,\boldsymbol{0},\lambda_j^i)}(y))-f_{\sigma}(\bar{u}_{(\boldsymbol{x}, \boldsymbol{L}, \boldsymbol{0},\boldsymbol{1})}(y))\right)\mathcal{R}_\sigma(x-y)\ud y\\
			& =\int_{\mathbb R^n}\mathcal{D}_{\sigma}(\bar{u}_{(\boldsymbol{x}, \boldsymbol{L}, \boldsymbol{0},\boldsymbol{1})}(y))\mathcal{R}_\sigma(x-y)\ud y\\
			& =\left(\int_{|y|\leqslant 1}+\int_{1\leqslant |y|\leqslant |x|}+\int_{|y|\geqslant |x|}\right)\mathcal{D}_{\sigma}(\bar{u}_{(\boldsymbol{x}, \boldsymbol{L}, \boldsymbol{0},\boldsymbol{1})}(y))\mathcal{R}_\sigma(x-y)\ud y\\
			&=:I_{11}+I_{12}+I_{13}.
		\end{align*}
		where we recall that $\mathcal{D}_{\sigma}(\bar{u}_{(\boldsymbol{x}, \boldsymbol{L}, \boldsymbol{0},\boldsymbol{1})})$ is given by \eqref{differenceterm}.
		
		\noindent Applying Step 1 of Claim 1, we have
		\begin{equation*}
			|I_{11}|\lesssim e^{-\gamma_\sigma(1+\xi)}\int_{|y|\leqslant 1}{|x-y|^{2\sigma-n}}{|y|^{\zeta_1-2\sigma}}\ud y\lesssim |x|^{2\sigma-n}e^{-\gamma_\sigma(1+\xi)},
		\end{equation*}
		\begin{equation*}
			|I_{12}|\lesssim e^{-\gamma_\sigma(1+\xi)}\int_{1\leqslant |y|\leqslant |x|}{|x-y|^{n-2\sigma}}{|y|^{-(n+2\sigma)}}\ud y\lesssim |x|^{2\sigma-n}e^{-\gamma_\sigma L(1+\xi)},
		\end{equation*}
		and
		\begin{equation*}
			|I_{13}|\lesssim e^{-\gamma_\sigma(1+\xi)}\int_{|y|\geqslant |x|}{|x-y|^{2\sigma-n}}{|y|^{-(n+2\sigma)}}\ud y\lesssim |x|^{-n}e^{-\gamma_\sigma(1+\xi)}.
		\end{equation*}
		Combining the above estimates, we finish the proof of Step 1.
		
		\noindent{\bf Step 2:} If $\frac{1}{2} \leqslant|x| \leqslant 1$, then $|\mathscr{N}_\sigma(\bar{u}_{(\boldsymbol{x},\boldsymbol{L},\boldsymbol{a}_j,\boldsymbol{\lambda}_j)})|\lesssim e^{-\gamma_\sigma L(1+\xi)}$ for some $\xi>0$.
		
		\noindent In this case, it holds
		\begin{align*}
			\mathscr{N}_{\sigma}(\bar{u}_{(\boldsymbol{x}, \boldsymbol{L}, \boldsymbol{0},\boldsymbol{1})})&= \bar{u}_{(\boldsymbol{x}, \boldsymbol{L}, \boldsymbol{0},\boldsymbol{1})}-(-\Delta)^{-\sigma}(f_\sigma(\bar{u}_{(\boldsymbol{x}, \boldsymbol{L}, \boldsymbol{0},\boldsymbol{1})}))\\
			& =\sum_{i=1}^N\sum_{j\in\mathbb N} U_{(x_i, L_j^i,\boldsymbol{0},\lambda_j^i)}+\chi_1\phi_1-(-\Delta)^{-\sigma}f_{\sigma}\left(\bar{u}_{(\boldsymbol{x}, \boldsymbol{L}, \boldsymbol{0},\boldsymbol{1})}\right)\\
			& =\int_{\mathbb R^n}\mathcal{D}_{\sigma}(\bar{u}_{(\boldsymbol{x}, \boldsymbol{L}, \boldsymbol{0},\boldsymbol{1})})(y)\mathcal{R}_\sigma(x-y)\ud y+e^{-\gamma_\sigma L(1+\xi)}\\
			& =\left(\int_{|y|\leqslant \frac{|x|}{2}}+\int_{\frac{|x|}{2}\leqslant |y|\leqslant 1}+\int_{1\leqslant |y|\leqslant 2|x|}+\int_{|y|\geqslant 2|x|}\right)\mathcal{D}_{\sigma}(\bar{u}_{(\boldsymbol{x}, \boldsymbol{L}, \boldsymbol{0},\boldsymbol{1})})(y)\mathcal{R}_\sigma(x-y)\ud y+\mathcal{O}(e^{-\gamma_\sigma L(1+\xi)})\\
			&=:I_{21}+I_{22}+I_{23}+I_{24}+\mathcal{O}(e^{-\gamma_\sigma L(1+\xi)}).    
		\end{align*}
		
		\noindent Applying Step 2 of Claim 1, we get
		\begin{align*}
			|I_{21}|\lesssim e^{-\gamma_\sigma L(1+\xi)}\int_{|y|\leqslant \frac{|x|}{2}}{|x-y|^{2\sigma-n}}{|y|^{\zeta_1-2\sigma}}\ud y\lesssim |x|^{\zeta_1}e^{-\gamma_\sigma L(1+\xi)}\lesssim e^{-\gamma_\sigma L(1+\xi)},   
		\end{align*}
		\begin{align*}
			|I_{22}|
			\lesssim e^{-\gamma_\sigma L(1+\xi)}\int_{\frac{|x|}{2}\leqslant |y|\leqslant 1}{|x-y|^{2\sigma-n}}{|y|^{\zeta_1-2\sigma}}\ud y\lesssim |x|^{\zeta_1-2\sigma}e^{-\gamma_\sigma L(1+\xi)}\int_{|y-x|\leqslant \frac{3}{2}}{|x-y|^{2\sigma-n}}\ud y\lesssim e^{-\gamma_\sigma L(1+\xi)},   
		\end{align*}
		\begin{align*}
			|I_{23}|
			\lesssim e^{-\gamma_\sigma L(1+\xi)}\int_{1\leqslant |y|\leqslant 2|x|}{|x-y|^{2\sigma-n}}{|y|^{-(n+2\sigma)}}\ud y\lesssim e^{-\gamma_\sigma L(1+\xi)}\int_{|y-x|\leqslant 3|x|}{|x-y|^{2\sigma-n}}\ud y\lesssim e^{-\gamma_\sigma L(1+\xi)},   
		\end{align*}
		and
		\begin{align*}
			|I_{24}|
			\lesssim e^{-\gamma_\sigma L(1+\xi)}\int_{|y|\geqslant 2|x|}{|x-y|^{2\sigma-n}}{|y|^{-(n+2\sigma)}}\ud y\lesssim |x|^{2\sigma-n}e^{-\gamma_\sigma L(1+\xi)}\int_{|y|\geqslant 1}{|y|^{-(n+2\sigma)}}\ud y\lesssim e^{-\gamma_\sigma L(1+\xi)}.   
		\end{align*}
		Consequently, the proof of Step 2 follows.
		
		\noindent{\bf Step 3:} If $0<|x| \leqslant \frac{1}{2}$, then $|\mathscr{N}_\sigma(\bar{u}_{(\boldsymbol{x}, \boldsymbol{L}, \boldsymbol{0},\boldsymbol{1})})|\lesssim |x|^{\zeta_1-\tau} e^{-\gamma_\sigma L(1+\xi)}$.
		
		\noindent Similarly to the previous steps, we obtain
		\begin{align*}
			\mathscr{N}_{\sigma}(\bar{u}_{(\boldsymbol{x}, \boldsymbol{L}, \boldsymbol{0},\boldsymbol{1})})&= \bar{u}_{(\boldsymbol{x}, \boldsymbol{L}, \boldsymbol{0},\boldsymbol{1})}-(-\Delta)^{-\sigma}(f_\sigma(\bar{u}_{(\boldsymbol{x}, \boldsymbol{L}, \boldsymbol{0},\boldsymbol{1})}))\\
			& =\sum_{i=1}^N\sum_{j\in\mathbb N} U_{(x_i, L_j^i,\boldsymbol{0},\lambda_j^i)}+\phi_1-(-\Delta)^{-\sigma}f_{\sigma}\left(\bar{u}_{(\boldsymbol{x}, \boldsymbol{L}, \boldsymbol{0},\boldsymbol{1})}\right)\\
			& =\int_{\mathbb R^n}\mathcal{D}_{\sigma}(\bar{u}_{(\boldsymbol{x}, \boldsymbol{L}, \boldsymbol{0},\boldsymbol{1})}(y))\mathcal{R}_\sigma(x-y)\ud y+e^{-\gamma_\sigma L(1+\xi)}\\
			& =\left(\int_{|y|\leqslant \frac{|x|}{2}}+\int_{\frac{|x|}{2}\leqslant |y|\leqslant 2|x|}+\int_{2|x|\leqslant |y|\leqslant 1}+\int_{|y|\geqslant 1}\right)\mathcal{D}_{\sigma}(\bar{u}_{(\boldsymbol{x}, \boldsymbol{L}, \boldsymbol{0},\boldsymbol{1})}(y))\mathcal{R}_\sigma(x-y)\ud y+\mathcal{O}(e^{-\gamma_\sigma L(1+\xi)})\\
			&=:I_{31}+I_{32}+I_{33}+I_{34}+\mathcal{O}(e^{-\gamma_\sigma L(1+\xi)}).    
		\end{align*}
		
		\noindent Applying Step 3 of Claim 1, we get
		\begin{align*}
			|I_{31}|\lesssim e^{-\gamma_\sigma L(1+\xi)}\int_{|y|\leqslant \frac{|x|}{2}}{|x-y|^{2\sigma-n}}{|y|^{\zeta_1-2\sigma}}\ud y\lesssim |x|^{\zeta_1}e^{-\gamma_\sigma L(1+\xi)},   
		\end{align*}
		\begin{align*}
			|I_{32}|
			\lesssim e^{-\gamma_\sigma L(1+\xi)}\int_{\frac{|x|}{2}\leqslant |y|\leqslant 2|x|}{|x-y|^{2\sigma-n}}{|y|^{\zeta_1-2\sigma}}\ud y
			\lesssim |x|^{\zeta_1}e^{-\gamma_\sigma L(1+\xi)},   
		\end{align*}
		\begin{align*}
			|I_{33}|
			\lesssim e^{-\gamma_\sigma L(1+\xi)}\int_{2|x|\leqslant |y|\leqslant 1}{|x-y|^{2\sigma-n}}{|y|^{\zeta_1-2\sigma}}\ud y\lesssim e^{-\gamma_\sigma L(1+\xi)}|x|^{\zeta_1}\int_{2|x|\leqslant |y|\leqslant 1}{|y|^{-n}}\ud y\lesssim |x|^{\zeta_1-\tau}e^{-\gamma_\sigma L(1+\xi)},   
		\end{align*}
		and
		\begin{align*}
			|I_{34}|
			\lesssim e^{-\gamma_\sigma L(1+\xi)}\int_{|y|\geqslant 1}{|x-y|^{2\sigma-n}}{|y|^{-(n+2\sigma)}}\ud y\lesssim e^{-\gamma_\sigma L(1+\xi)}\int_{|y|\geqslant 1}{|y|^{-(n+2\sigma)}}\ud y\lesssim e^{-\gamma_\sigma L(1+\xi)}.   
		\end{align*}
		Therefore, for $|x|\leqslant \frac{1}{2}$, we conclude
		\begin{equation*}
			|\mathscr{N}_\sigma(\bar{u}_{(\boldsymbol{x}, \boldsymbol{L}, \boldsymbol{0},\boldsymbol{1})})|\lesssim |x|^{\zeta_1-\tau} e^{-\gamma_\sigma L(1+\xi)},
		\end{equation*}
		which gives us the desired estimate in Step 3.
		
		By combining the last three steps, the proof of the first case is concluded.
		
		Now we consider the case of a general configuration. 
		We will use a perturbation technique based on the last case in this situation.
		
		\noindent{\bf Case 2:} $(\boldsymbol{a}_j,\boldsymbol{\lambda}_j)\neq(\boldsymbol{0},\boldsymbol{1})$ for some ${j\in\mathbb N}$.
		
		Initially, we will prove the following decomposition.
		
		\noindent{\bf Claim 3:} It holds that 
		\begin{align*}
			&|\mathscr{N}_\sigma(\bar{u}_{(\boldsymbol{x},\boldsymbol{L},\boldsymbol{a}_j,\boldsymbol{\lambda}_j)})-\mathscr{N}_\sigma(\bar{u}_{(\boldsymbol{x}, \boldsymbol{L}, \boldsymbol{0},\boldsymbol{1})})|\\
			&\leqslant \sum_{i=1}^N \sum_{j\in\mathbb N} (-\Delta)^{-\sigma}\left[|f^{\prime}_\sigma(U_{(x_i, L_j^i,a_j^i,\lambda_j^i)})-f^{\prime}_\sigma(U_{(x_i, L_j^i,\boldsymbol{0},\lambda_j^i)})|\left(|\partial_{r_j^i}U_{(x_\ell, L_\ell^i,a_\ell^i,\lambda_\ell^i)}||r_j^i|
			+\sum_{\ell=1}^n|\partial_{a_{j,\ell}^i}U_{(x_\ell, L_\ell^i,a_\ell^i,\lambda_\ell^i)}||a_{j,\ell}^i|\right)\right]\\
			&=\sum_{i=1}^N \sum_{j\in\mathbb N}(-\Delta)^{-\sigma}\left[|\partial_{r_j^i}\widetilde{\mathcal{D}}^\prime_\sigma(\bar{u}_{(\boldsymbol{x}, \boldsymbol{L}, \boldsymbol{0},\boldsymbol{1})})||r_j^i|+\sum_{\ell=1}^n|\partial_{a_{j,\ell}^i}\Tilde{\mathcal{D}}^\prime_\sigma(\bar{u}_{(\boldsymbol{x}, \boldsymbol{L}, \boldsymbol{0},\boldsymbol{1})})||a_{j,\ell}^i|\right],
		\end{align*}
		where
		\begin{align}
			\widetilde{\mathcal{D}}_{\sigma}(\bar{u}_{(\boldsymbol{x}, \boldsymbol{L}, \boldsymbol{0},\boldsymbol{1})}):=\sum_{i=1}^N \sum_{j\in\mathbb N}f_{\sigma}(U_{(x_i, L_j^i,a_j^i,\lambda_j^i)})-f_{\sigma}(\bar{u}_{(\boldsymbol{x}, \boldsymbol{L}, \boldsymbol{0},\boldsymbol{1})}).
		\end{align}
		
		\noindent To prove this fact, we will differentiate $\mathscr{N}_\sigma(\bar{u}_{(\boldsymbol{x}, \boldsymbol{L}, \boldsymbol{0},\boldsymbol{1})})$ with respect to the parameters $r_j^i,a^i_{j,\ell}$. 
		Since the variation is linear in the displacements of the parameters, we vary the parameter of one point at one time. 
		First, with respect to $r_j^i$, we have
		\begin{align*}
			{\partial_{r_j^i}} \mathscr{N}_{\sigma}(\bar{u}_{(\boldsymbol{x}, \boldsymbol{L}, \boldsymbol{0},\boldsymbol{1})})&=\partial_{r_j^i}U_{(x_i, L_j^i,a_j^i,\lambda_j^i)}-(-\Delta)^{-\sigma} (f_{\sigma}^{\prime}(\bar{u}_{(\boldsymbol{x}, \boldsymbol{L}, \boldsymbol{0},\boldsymbol{1})}) {\partial_{r_j^i}}U_{(x_i, L_j^i,a_j^i,\lambda_j^i)})\\
			&=(-\Delta)^{-\sigma}[(f^{\prime}_\sigma(U_{(x_i, L_j^i,a_j^i,\lambda_j^i)})-f^{\prime}_\sigma(\bar{u}_{(\boldsymbol{x}, \boldsymbol{L}, \boldsymbol{0},\boldsymbol{1})}) {\partial_{r_j^i}}U_{(x_i, L_j^i,a_j^i,\lambda_j^i)}].
		\end{align*}
		Second, with respect to $a_{j,\ell}^i$, we obtain
		\begin{align*}
			{\partial_{a_{j,\ell}^i}} \mathscr{N}_{\sigma}(\bar{u}_{(\boldsymbol{x}, \boldsymbol{L}, \boldsymbol{0},\boldsymbol{1})})&=(-\Delta)^{-\sigma}\left[(f^{\prime}_\sigma(U_{(x_i, L_j^i,a_j^i,\lambda_j^i)})-f^{\prime}_\sigma(\bar{u}_{(\boldsymbol{x}, \boldsymbol{L}, \boldsymbol{0},\boldsymbol{1})})) \sum_{\ell=1}^n{\partial_{a_{j,\ell}^i}}U_{(x_\ell, L_\ell^i,a_\ell^i,\lambda_\ell^i)}\right].
		\end{align*}
		This fact concludes the proof of Claim 3.
		
		Next, we shall obtain $L^\infty$-estimate
		in the sense below.
		We first consider the case of the parameters $r_j^i$.
		
		\noindent{\bf Claim 4:} The following estimate holds
		\begin{align*}
			&\left|\partial_{r_j^i}\widetilde{\mathcal{D}}^\prime_\sigma(\bar{u}_{(\boldsymbol{x}, \boldsymbol{L}, \boldsymbol{0},\boldsymbol{1})})\right|\lesssim
			\begin{cases}
				\ud(x,\Sigma)^{\min\{\zeta_1,-\gamma_\sigma+\tau\}-2\sigma} e^{-\gamma_\sigma L(1+\xi)}, \ &{\rm if} \ 0<\ud(x, \Sigma)<1,\\
				|x|^{-(n+2\sigma)} e^{-\gamma_\sigma L(1+\xi)}, \ &{\rm if} \ \ud(x, \Sigma) \geqslant 1.
			\end{cases}
		\end{align*}
		
		\noindent As before, we consider two cases separately.
		
		\noindent{\bf Step 1:} If  $\ud(x, \Sigma) \geqslant 1$, then 
		\begin{equation*}
			[f^{\prime}_\sigma(U_{(x_i, L_j^i,a_j^i,\lambda_j^i)})-f^{\prime}_\sigma(\bar{u}_{(\boldsymbol{x}, \boldsymbol{L}, \boldsymbol{0},\boldsymbol{1})})] {\partial_{r_j^i}}U_{(x_i, L_j^i,a_j^i,\lambda_j^i)} \lesssim |x|^{-(n+2 \sigma)} e^{-\gamma_\sigma L(1+\xi)} e^{-\nu t_j^i},
		\end{equation*}
		for a suitable choice of $\nu>0$. 
		
		\noindent As a matter of fact, we have
		\begin{align*}
			[f^{\prime}_\sigma(U_{(x_i, L_j^i,a_j^i,\lambda_j^i)})-f^{\prime}_\sigma(\bar{u}_{(\boldsymbol{x}, \boldsymbol{L}, \boldsymbol{0},\boldsymbol{1})})] {\partial_{r_j^i}}U_{(x_i, L_j^i,a_j^i,\lambda_j^i)} & \lesssim\left(e^{-\gamma_{\sigma} L}|x|^{-(n-2 \sigma)}\right)^{\frac{4n}{n-2\sigma}} e^{-\gamma_{\sigma} L(2 j+1)}|x|^{-(n-2 \sigma)} \\
			& \lesssim |x|^{-(n+2 \sigma)} e^{-\gamma_{\sigma} L(1+\xi)} e^{-\nu t_j^i},
		\end{align*}
		which by \eqref{balancing5} and \eqref{balancing6} concludes the proof of this step.
		
		\noindent{\bf Step 2:} If  $0<\ud(x, \Sigma) < 1$, then 
		\begin{align*}
			[f^{\prime}_\sigma(U_{(x_i, L_j^i,a_j^i,\lambda_j^i)})-f^{\prime}_\sigma(\bar{u}_{(\boldsymbol{x}, \boldsymbol{L}, \boldsymbol{0},\boldsymbol{1})})] {\partial_{r_j^i}}U_{(x_i, L_j^i,a_j^i,\lambda_j^i)}\lesssim \ud(x, \Sigma)^{\min\{-\gamma_\sigma+\tau, \zeta_1\}-2 \sigma}e^{-\gamma_\sigma L(1+\xi)}
		\end{align*}
		for some $-\gamma_\sigma<\zeta_1<\min \{0,-\gamma_\sigma+2 \sigma\}$ and $0<\tau\ll 1$ small enough.
		
		\noindent In this situation, we may assume without loss of generality that $|x-x_i| \leqslant 1$ for $i\in\{2,\dots,N\}$.
		Hence, we proceed similarly to the proof of the estimates \eqref{quantitativest3} and \eqref{quantitativest4} to find
		\begin{align*}
			&[f^{\prime}_\sigma(U_{(x_i, L_j^i,a_j^i,\lambda_j^i)})-f^{\prime}_\sigma(\bar{u}_{(\boldsymbol{x}, \boldsymbol{L}, \boldsymbol{0},\boldsymbol{1})})] {\partial_{r_j^i}}U_{(x_i, L_j^i,a_j^i,\lambda_j^i)}\\
			& \lesssim |x-x_i|^{-2 \sigma}\left(\sum_{j\in\mathbb N} V_{(x_{i^\prime},L_{j^\prime}^{i^\prime}\lambda_{j^{\prime}}^{i^\prime},a_{j^{\prime}}^{i^\prime})}(-\ln |x-x_i|)\right)^{\frac{n+2\sigma}{n-2\sigma}} e^{-\gamma_\sigma L(2 j+1)} \\
			& \lesssim |x-x_i|^{\zeta_1-2 \sigma} e^{-\gamma_\sigma L(1+\xi)} e^{-\nu t_j^i}.
		\end{align*}
		for a suitable choice of $\nu>0$. 
		
		\noindent Again, we have more two cases to consider.
		If $|t-t^i_j|\geqslant L_1$, it follows  
		\begin{align*}
			&[f^{\prime}_\sigma(U_{(x_i, L_j^i,a_j^i,\lambda_j^i)})-f^{\prime}_\sigma(\bar{u}_{(\boldsymbol{x}, \boldsymbol{L}, \boldsymbol{0},\boldsymbol{1})})] {\partial_{r_j^i}}U_{(x_i, L_j^i,a_j^i,\lambda_j^i)}\\
			& \lesssim f^{\prime}_\sigma(U_{(x_i, L_j^i,a_j^i,\lambda_j^i)})\left(\sum_{\ell \neq j} f^{\prime}_\sigma(U_{(x_\ell, L_\ell^i,a_\ell^i,\lambda_\ell^i)})+e^{-\gamma_\sigma L}\right) \\
			& \lesssim |x|^{-\frac{n+2 \sigma}{2}} \sum_{\ell \neq j} V_{(x_i, L_j^i,a_j^i,\lambda_j^i)}^{\frac{4\sigma}{n-2\sigma}} V_{(x_\ell, L_\ell^i,a_\ell^i,\lambda_\ell^i)}+|x|^{-2 \sigma} V_{(x_{i^\prime},L_{j^\prime}^{i^\prime}\lambda_{j^{\prime}}^{i^\prime},a_{j^{\prime}}^{i^\prime})}^{\frac{4\sigma}{n-2\sigma}} e^{-\gamma_\sigma L} \\
			& \lesssim |x|^{-\frac{n+2 \sigma}{2}} e^{-\eta|t-t_j|} \sum_{\ell \neq j} e^{-(2 \sigma-\eta)|t-t_j^i|} e^{-\gamma_\sigma|t-t_\ell^i|}+|x|^{\zeta_1-2 \sigma}|x|^{\zeta_1} e^{-2 \sigma|t-t_j^i|} e^{-\gamma_\sigma L}\\
			& \lesssim\left(|x|^{-\frac{n+2 \sigma}{2}} e^{-\eta|t-t_j^i|}+|x|^{\zeta_1-2 \sigma} e^{-\nu t_j^i}\right) e^{-\gamma_\sigma L(1+\xi)},
		\end{align*}
		if $0<\eta<2 \sigma$ is chosen suitably.     
		Whereas, if $|t-t_\ell^i| \leqslant {L_1}$ for some $\ell \neq j$, one has
		\begin{align*}
			[f^{\prime}_\sigma(U_{(x_i, L_j^i,a_j^i,\lambda_j^i)})-f^{\prime}_\sigma(\bar{u}_{(\boldsymbol{x}, \boldsymbol{L}, \boldsymbol{0},\boldsymbol{1})})] {\partial_{r_j^i}}U_{(x_i, L_j^i,a_j^i,\lambda_j^i)}&\lesssim f^{\prime}_\sigma(U_{(x_i, L_j^i,a_j^i,\lambda_j^i)}){\partial_{r_j^i}}U_{(x_i, L_j^i,a_j^i,\lambda_j^i)} \\  
			& \lesssim |x|^{-\frac{n+2 \sigma}{2}} \sum_{\ell \neq j} V_{(x_\ell, L_\ell^i,a_\ell^i,\lambda_\ell^i)}^{\frac{4\sigma}{n-2\sigma}} V_{(x_i, L_j^i,a_j^i,\lambda_j^i)}\\
			& \lesssim |x|^{-\gamma_\sigma} e^{-\eta|t-t_j^i|} e^{\eta|t-t_j^i|} e^{-\gamma_\sigma|t-t_j|} e^{-2 \sigma|t-t_\ell^i|} \\
			& \lesssim|x|^{-\frac{n+2 \sigma}{2}} e^{-\eta|t-t_j^i|} e^{-\gamma_\sigma L(1+\xi)}
		\end{align*}
		if $0<\eta\ll\gamma_\sigma$ is chosen small enough. 
		
		\noindent In conclusion, by combining the above two estimates, we get 
		\begin{equation*}
			[f^{\prime}_\sigma(U_{(x_i, L_j^i,a_j^i,\lambda_j^i)})-f^{\prime}_\sigma(\bar{u}_{(\boldsymbol{x}, \boldsymbol{L}, \boldsymbol{0},\boldsymbol{1})})] {\partial_{r_j^i}}U_{(x_i, L_j^i,a_j^i,\lambda_j^i)} \lesssim |x|^{-\frac{n+2 \sigma}{2}} e^{-\tau t} e^{-\gamma_\sigma L(1+\xi)}+|x|^{\zeta_1-2 \sigma} e^{-\gamma_\sigma L(1+\xi)}
		\end{equation*}
		for $0<|x| \leqslant 1$, which implies
		\begin{align*}
			[f^{\prime}_\sigma(U_{(x_i, L_j^i,a_j^i,\lambda_j^i)})-f^{\prime}_\sigma(\bar{u}_{(\boldsymbol{x}, \boldsymbol{L}, \boldsymbol{0},\boldsymbol{1})})] {\partial_{r_j^i}}U_{(x_i, L_j^i,a_j^i,\lambda_j^i)}
			& \lesssim(\ud(x, \Sigma)^{-\frac{n+2 \sigma}{2}} \ud(x, \Sigma)^\tau+\ud(x, \Sigma)^{\zeta_1-2 \sigma})e^{-\gamma_\sigma(1+\xi)}\\
			& \lesssim \ud(x, \Sigma)^{\min\{-\gamma_\sigma+\tau, \zeta_1\}-2 \sigma}e^{-\gamma_\sigma(1+\xi)}.
		\end{align*}
		The proof of this step is concluded, and so is one of the claims.
		
		\noindent{\bf Claim 5:} The following estimate holds
		\begin{align*}
			&\left|\partial_{a_{j,\ell}^i}\widetilde{\mathcal{D}}^\prime_\sigma(\bar{u}_{(\boldsymbol{x}, \boldsymbol{L}, \boldsymbol{0},\boldsymbol{1})})\right|\lesssim
			\begin{cases}
				\ud(x,\Sigma)^{\min\{\zeta_1,-\gamma_\sigma+\tau\}-2\sigma} e^{-\gamma_\sigma L(1+\xi)}, \ &{\rm if} \ 0<\ud(x, \Sigma)<1,\\
				|x|^{-(n+2\sigma)} e^{-\gamma_\sigma L(1+\xi)}, \ &{\rm if} \ \ud(x, \Sigma) \geqslant 1.
			\end{cases}
		\end{align*}
		
		\noindent  The estimates are similar to the ones in the last claim, so we omit them here.
		
		As a combination of these estimates, we have our main conclusion. 
		
		\noindent{\bf Claim 6:} The following estimate holds
		\begin{align*}
			&\left|\mathscr{N}_\sigma(\bar{u}_{(\boldsymbol{x},\boldsymbol{L},\boldsymbol{a}_j,\boldsymbol{\lambda}_j)})-\mathscr{N}_\sigma(\bar{u}_{(\boldsymbol{x}, \boldsymbol{L}, \boldsymbol{0},\boldsymbol{1})})\right|\lesssim
			\begin{cases}
				\ud(x,\Sigma)^{\min\{\zeta_1-\tau,-\gamma_\sigma+\tau\}} e^{-\gamma_\sigma L(1+\xi)}, \ &{\rm if} \ 0<\ud(x, \Sigma)<1,\\
				|x|^{-(n-2\sigma)} e^{-\gamma_\sigma L(1+\xi)}, \ &{\rm if} \ \ud(x, \Sigma) \geqslant 1.
			\end{cases}
		\end{align*}
		
		\noindent To prove this claim, we plug Claims 4 and 5 into Claim 3 and proceed similarly to the proof of Claim 2.
		
		Finally, using the definitions of the weighted norms in Definition~\ref{def:suitablespaces}, it is straightforward to see that \eqref{estimateaproxximatesolution} is a direct consequence of the last claim.
		
		The lemma is finally proved.         
	\end{proof}
	
	\subsubsection{Finite-dimensional reduction}
	We apply a finite-dimensional Lyapunov--Schmidt reduction to solve an auxiliary linearized equation around an approximate solution.
	As usual in this method, we use the orthogonality properties of the normalized approximate kernels and cokernels from Lemma~\ref{lm:orthogonalityconditions}.
	
	\begin{lemma}\label{lm:fredholmtheory}
		Let $\sigma \in(1,+\infty)$, $n>2\sigma$ and $N\geqslant 2$.
		Assume that $(\boldsymbol{a}_j,\boldsymbol{\lambda}_j)\in{\rm Adm}_{\sigma}(\Sigma)$ is an admissible configuration as in Definition~\ref{def:balancedparameters} with $\bar{u}_{(\boldsymbol{x},\boldsymbol{L},\boldsymbol{a}_j,\boldsymbol{\lambda}_j)}\in{\rm Apx}_{\sigma}(\Sigma)$ their associated approximate solution as in Definition~\ref{def:approximatesolution}.  
		Then, there exists a weight $\zeta_1<0$ satisfying \eqref{weightassumption} such that for any $h\in\mathcal{C}_{**,\tau}(\mathbb{R}^n\setminus\Sigma)$, 
		there exists $\{c_{j, \ell}^i(\boldsymbol{a}_j,\boldsymbol{\lambda}_j)\}_{(i,j,\ell)\in\mathcal{I}_\infty}\subset\mathbb R$ and a unique solution            
		$\phi\in \mathcal{C}_{*,\tau}(\mathbb{R}^n\setminus\Sigma)$ to 
		the following linearized equation 
		\begin{align}\label{auxiliarysystemdual}\tag{$\mathcal{L}_{2\sigma,\boldsymbol{a},\boldsymbol{\lambda}}^{\prime}$}
			\left\{\begin{array}{l}
				\mathscr{L}_{\sigma}(\boldsymbol{x},\boldsymbol{L},\boldsymbol{a}_j,\boldsymbol{\lambda}_j)(\phi)=h+ \displaystyle\displaystyle\sum_{i=1}^N  \displaystyle\sum_{j\in\mathbb N}  \displaystyle\sum_{\ell=0}^n c_{j, \ell}^i(\boldsymbol{a}_j,\boldsymbol{\lambda}_j) {Z}_{j, \ell}^i(\boldsymbol{a}_j,\boldsymbol{\lambda}_j) \quad {\rm in} \quad \mathbb{R}^n\setminus\Sigma, \\
				\displaystyle\int_{\mathbb{R}^n} \phi\overline{Z}_{j, \ell}^i(\boldsymbol{a}_j,\boldsymbol{\lambda}_j) \ud x=0  \quad {\rm for} \quad (i,j,\ell)\in \mathcal{I}_{\infty}.
			\end{array}\right.                   
		\end{align}
		Moreover, one has the estimate
		\begin{equation*}
			\|\phi\|_{\mathcal{C}_{*,\tau}(\mathbb{R}^n\setminus\Sigma)} \lesssim \|h\|_{\mathcal{C}_{**,\tau}(\mathbb{R}^n\setminus\Sigma)}.
		\end{equation*}
		uniformly on $\lambda\ll1$ large.
		In what follows, we shall denote this error function by $\phi_{(\boldsymbol{x},\boldsymbol{L},\boldsymbol{a}_j,\boldsymbol{\lambda}_j)}$.
	\end{lemma}
	
	\begin{proof}
		First, by multiplying equation \eqref{auxiliarysystemdual} by the normalized approximate cokernels 
		$\overline{Z}_{j^{\prime}, \ell^{\prime}}^{i^\prime}(\boldsymbol{a}_j,\boldsymbol{\lambda}_j)$ given by Definition~\ref{def:kernels}, and integrating over $\mathbb{R}^n$, it follows
		\begin{align}\label{linearizedtimeskernels}
			&\int_{\mathbb{R}^n}\left[\phi-(-\Delta)^{-\sigma}(f^{\prime}_\sigma(\bar{u}_{(\boldsymbol{x},\boldsymbol{L},\boldsymbol{a}_j,\boldsymbol{\lambda}_j)}) \phi)\right] f_{\sigma}^{\prime}(U_{(x_{i^\prime},L_{j^\prime}^{i^\prime}\lambda_{j^{\prime}}^{i^\prime},a_{j^{\prime}}^{i^\prime})})Z_{j^{\prime}, \ell^{\prime}}^{i^\prime}(\boldsymbol{a}_j^\prime,\boldsymbol{\lambda}_j^\prime) \ud x\\\nonumber
			&=\int_{\mathbb{R}^n} h f_{\sigma}^{\prime}(U_{(x_{i^\prime},L_{j^\prime}^{i^\prime}\lambda_{j^{\prime}}^{i^\prime},a_{j^{\prime}}^{i^\prime})})Z_{j^{\prime}, \ell^{\prime}}^{i^\prime}(\boldsymbol{a}_j^\prime,\boldsymbol{\lambda}_j^\prime) \ud x\\\nonumber
			&+\displaystyle\sum_{i=1}^N  \displaystyle\sum_{j\in\mathbb N}  \displaystyle\sum_{\ell=0}^n c_{j, \ell}^i(\boldsymbol{a}_j^\prime,\boldsymbol{\lambda}_j^\prime)\int_{\mathbb R^n} f_{\sigma}^{\prime}(U_{(x_{i^\prime},L_{j^\prime}^{i^\prime}\lambda_{j^{\prime}}^{i^\prime},a_{j^{\prime}}^{i^\prime})}){Z}_{j, \ell}^i(\boldsymbol{a}_j^\prime,\boldsymbol{\lambda}_j^\prime) Z_{j^{\prime}, \ell^{\prime}}^{i^\prime}(\boldsymbol{a}^\prime_j,\boldsymbol{\lambda}_j^\prime)\ud x.
		\end{align}
		They simplify our notation, let us set 
		\begin{align*}
			I_0&=\int_{\mathbb{R}^n}\left[\phi-(-\Delta)^{-\sigma}(f^{\prime}_\sigma(\bar{u}_{(\boldsymbol{x},\boldsymbol{L},\boldsymbol{a}_j,\boldsymbol{\lambda}_j)}) \phi)\right] f_{\sigma}^{\prime}(U_{(x_{i^\prime},L_{j^\prime}^{i^\prime}\lambda_{j^{\prime}}^{i^\prime},a_{j^{\prime}}^{i^\prime})})Z_{j^{\prime}, \ell^{\prime}}^{i^\prime}(\boldsymbol{a}_j^\prime,\boldsymbol{\lambda}_j^\prime) \ud x
		\end{align*}
		and
		\begin{align*}
			I_1&=\int_{\mathbb{R}^n} h f_{\sigma}^{\prime}(U_{(x_{i^\prime},L_{j^\prime}^{i^\prime}\lambda_{j^{\prime}}^{i^\prime},a_{j^{\prime}}^{i^\prime})})Z_{j^{\prime}, \ell^{\prime}}^{i^\prime}(\boldsymbol{a}_j^\prime,\boldsymbol{\lambda}_j^\prime) \ud x.
		\end{align*}
		In the next claims, we will estimate the two terms above based on the orthogonality conditions from Lemma~\ref{lm:orthogonalityconditions}.
		
		\noindent{\bf Claim 1:} The following estimate holds
		\begin{equation*}
			|I_0|\lesssim  \|\phi\|_{\mathcal{C}_{*,\tau}(\mathbb{R}^n\setminus\Sigma)}e^{-\gamma_\sigma L(1+\xi)}e^{-(\zeta_1-\tau+\gamma_\sigma)t_{j^\prime}^i}.
		\end{equation*}
		
		\noindent Indeed, it is not hard to check that the approximate kernel $Z_{j^{\prime}, \ell^{\prime}}^{i^{\prime}}(\boldsymbol{a}_j^\prime,\boldsymbol{\lambda}_j^\prime)$ satisfies the linearized equation below
		\begin{equation*}
			(-\Delta)^\sigma Z_{j^{\prime}, \ell^{\prime}}^{i^{\prime}}(\boldsymbol{a}_j^\prime,\boldsymbol{\lambda}_j^\prime)-f_{\sigma}^{\prime}(U_{(x_{i^\prime},L_{j^\prime}^{i^\prime}\lambda_{j^{\prime}}^{i^\prime},a_{j^{\prime}}^{i^\prime})})Z_{j^{\prime}, \ell^{\prime}}^{i^\prime}(\boldsymbol{a}_j^\prime,\boldsymbol{\lambda}_j^\prime)=0 \quad \quad \mathbb R^n\setminus\Sigma,
		\end{equation*}
		we have  
		\begin{align*}
			I_0
			&= \int_{\mathbb{R}^n}f_{\sigma}^{\prime}(U_{(x_{i^\prime},L_{j^\prime}^{i^\prime}\lambda_{j^{\prime}}^{i^\prime},a_{j^{\prime}}^{i^\prime})})\phi Z_{j^{\prime}, \ell^{\prime}}^i(\boldsymbol{a}_j,\boldsymbol{\lambda}_j)-(-\Delta)^{-\sigma}(f^{\prime}_\sigma(\bar{u}_{(\boldsymbol{x},\boldsymbol{L},\boldsymbol{a}_j,\boldsymbol{\lambda}_j)}) \phi) (-\Delta)^\sigma Z_{j^{\prime}, \ell^{\prime}}^{i^\prime}(\boldsymbol{a}_j^\prime,\boldsymbol{\lambda}_j^\prime) \ud x \\
			&= \int_{\mathbb{R}^n}\left[f_{\sigma}^{\prime}(U_{(x_{i^\prime},L_{j^\prime}^{i^\prime}\lambda_{j^{\prime}}^{i^\prime},a_{j^{\prime}}^{i^\prime})})-f^{\prime}_\sigma(\bar{u}_{(\boldsymbol{x},\boldsymbol{L},\boldsymbol{a}_j,\boldsymbol{\lambda}_j)}) \right]\phi Z_{j^{\prime}, \ell^{\prime}}^{i^\prime}(\boldsymbol{a}_j^\prime,\boldsymbol{\lambda}_j^\prime) \ud x \\
			&= {\left[\int_{B_1(x_{i^{\prime}})}+\sum_{i \neq i^{\prime}} \int_{B_1(x_i)}+\int_{\mathbb{R}^n \backslash \bigsqcup\limits_{i=1}^N B_1(x_i)}\right]\left[f_{\sigma}^{\prime}(U_{(x_{i^\prime},L_{j^\prime}^{i^\prime}\lambda_{j^{\prime}}^{i^\prime},a_{j^{\prime}}^{i^\prime})})-f^{\prime}_\sigma(\bar{u}_{(\boldsymbol{x},\boldsymbol{L},\boldsymbol{a}_j,\boldsymbol{\lambda}_j)}) \right]\phi Z_{j^{\prime}, \ell^{\prime}}^{i^\prime}(\boldsymbol{a}_j^\prime,\boldsymbol{\lambda}_j^\prime) \ud x } \\
			&=: I_{01}+ I_{02}+ I_{03} .
		\end{align*}
		
		\noindent Without loss of generality, assume that $i^\prime=1$ and $x_1=0$. First, we consider the case when $\ell^\prime=0$. Recalling the estimates for $Z_{j^{\prime}, 0}^{i^\prime}(\boldsymbol{a}_j,\boldsymbol{\lambda}_j)$ from \eqref{kernelestimate1}, we get
		\begin{align*}
			|I_{01}|
			&=\left|\int_{B_1}\left[f_{\sigma}^{\prime}(U_{(x_{i^\prime},L_{j^\prime}^{i^\prime}\lambda_{j^{\prime}}^{i^\prime},a_{j^{\prime}}^{i^\prime})})-f^{\prime}_\sigma(\bar{u}_{(\boldsymbol{x},\boldsymbol{L},\boldsymbol{a}_j,\boldsymbol{\lambda}_j)}) \right]\phi Z_{j^{\prime}, \ell^{\prime}}^{i^\prime}(\boldsymbol{a}_j^\prime,\boldsymbol{\lambda}_j^\prime) \ud x \right|\\
			&\lesssim \|\phi\|_{\mathcal{C}_{*,\tau}(\mathbb{R}^n\setminus\Sigma)}\int_{B_1}\left|f_{\sigma}^{\prime}(U_{(x_{i^\prime},L_{j^\prime}^{i^\prime}\lambda_{j^{\prime}}^{i^\prime},a_{j^{\prime}}^{i^\prime})})-f^{\prime}_\sigma(\bar{u}_{(\boldsymbol{x},\boldsymbol{L},\boldsymbol{a}_j,\boldsymbol{\lambda}_j)})\right| |x|^{\zeta_1}|Z_{j^{\prime}, \ell^{\prime}}^{i^\prime}(\boldsymbol{a}_j^\prime,\boldsymbol{\lambda}_j^\prime)| \ud x\\
			&\lesssim \|\phi\|_{\mathcal{C}_{*,\tau}(\mathbb{R}^n\setminus\Sigma)}\int_{B_1}|x|^{\zeta_1-\frac{n+2\sigma}{2}}V_{(x_{i^\prime},L_{j^\prime}^{i^\prime}\lambda_{j^{\prime}}^{i^\prime},a_{j^{\prime}}^{i^\prime})}^{\frac{4\sigma}{n-2\sigma}}\sum_{j\neq j^\prime}V_{(x_{i^\prime},L_{j^\prime}^{i^\prime}\lambda_{j^{\prime}}^{i^\prime},a_{j^{\prime}}^{i^\prime})} \ud x&\\
			&\lesssim \|\phi\|_{\mathcal{C}_{*,\tau}(\mathbb{R}^n\setminus\Sigma)}\int_{0}^{+\infty}e^{-(\zeta_1+\gamma_\sigma)t}v_{j^\prime}^{\frac{4\sigma}{n-2\sigma}}\sum_{j\neq j^\prime}V_{(x_i,L_i,a_j^i, \lambda^i_j)}\ud t\\
			&\lesssim \|\phi\|_{\mathcal{C}_{*,\tau}(\mathbb{R}^n\setminus\Sigma)}e^{-\gamma_\sigma L(1+\xi)}e^{-(\zeta_1+\gamma_\sigma)t_{j^\prime}^{i^\prime}},
		\end{align*}
		since $\zeta_1>-\gamma_\sigma$. 
		Next, it holds
		\begin{align*}
			|I_{02}|
			&=\left|\sum_{i\neq 1}\int_{B_1(x_i)}\left[f_{\sigma}^{\prime}(U_{(x_{i^\prime},L_{j^\prime}^{i^\prime}\lambda_{j^{\prime}}^{i^\prime},a_{j^{\prime}}^{i^\prime})})-f^{\prime}_\sigma(\bar{u}_{(\boldsymbol{x},\boldsymbol{L},\boldsymbol{a}_j,\boldsymbol{\lambda}_j)}) \right]\phi Z_{j^{\prime}, \ell^{\prime}}^{i^\prime}(\boldsymbol{a}_j^\prime,\boldsymbol{\lambda}_j^\prime) \ud x\right|\\
			&\lesssim \|\phi\|_{\mathcal{C}_{*,\tau}(\mathbb{R}^n\setminus\Sigma)}\sum_{i\neq 1}\int_{B_1(x_i)}\left|f_{\sigma}^{\prime}(U_{(x_{i^\prime},L_{j^\prime}^{i^\prime}\lambda_{j^{\prime}}^{i^\prime},a_{j^{\prime}}^{i^\prime})})-f^{\prime}_\sigma(\bar{u}_{(\boldsymbol{x},\boldsymbol{L},\boldsymbol{a}_j,\boldsymbol{\lambda}_j)}) \right||Z_{j^{\prime}, \ell^{\prime}}^{i^\prime}(\boldsymbol{a}_j^\prime,\boldsymbol{\lambda}_j^\prime)||x-x_i|^{\zeta_1}\ud x\\
			&\lesssim \|\phi\|_{\mathcal{C}_{*,\tau}(\mathbb{R}^n\setminus\Sigma)}\sum_{i\neq 1}\int_{B_1(x_1)} |x-x_i|^{\zeta_1-2\sigma}(\lambda^{i^\prime}_{j^\prime})^{\gamma_\sigma}\left(\sum_{j\in\mathbb N}V_{(x_i,L_i,a_j^i, \lambda^i_j)}(-\ln |x|)\right)\ud x\\
			&\lesssim \|\phi\|_{\mathcal{C}_{*,\tau}(\mathbb{R}^n\setminus\Sigma)}(\lambda^{i^\prime}_{j^\prime})^{\gamma_\sigma}e^{-(n+\zeta_1-2\sigma)L}\\
			&\lesssim \|\phi\|_{\mathcal{C}_{*,\tau}(\mathbb{R}^n\setminus\Sigma)}e^{\zeta_1 t^{i^\prime}_{j^\prime}-2\gamma_\sigma L-\zeta_1 L}e^{-(\zeta_1+\gamma_\sigma)t^{i^\prime}_{j^\prime}}\\
			&\lesssim \|\phi\|_{\mathcal{C}_{*,\tau}(\mathbb{R}^n\setminus\Sigma)}e^{-\gamma_\sigma L(1+\xi)}e^{-(\zeta_1+\gamma_\sigma)t_{j^\prime}^{i^\prime}}.
		\end{align*}
		In addition, one has
		\begin{align*}
			|I_{03}|
			&=\left|\int_{\mathbb{R}^n \backslash \bigsqcup\limits_{i=1}^N B_1(x_i)}\left[f_{\sigma}^{\prime}(U_{(x_{i^\prime},L_{j^\prime}^{i^\prime}\lambda_{j^{\prime}}^{i^\prime},a_{j^{\prime}}^{i^\prime})})-f^{\prime}_\sigma(\bar{u}_{(\boldsymbol{x},\boldsymbol{L},\boldsymbol{a}_j,\boldsymbol{\lambda}_j)}) \right]\phi Z_{j^{\prime}, \ell^{\prime}}^{i^\prime}(\boldsymbol{a}_j^\prime,\boldsymbol{\lambda}_j^\prime) \ud x \right|\\
			&\lesssim \|\phi\|_{\mathcal{C}_{*,\tau}(\mathbb{R}^n\setminus\Sigma)}\int_{\mathbb{R}^n \backslash \bigsqcup\limits_{i=1}^N B_1(x_i)} |x|^{-(n-2\sigma)}|x|^{-(n+2\sigma)}(\lambda^{i^\prime}_{j^\prime})^{\gamma_\sigma} e^{-2\sigma L} \ud x\\
			&\lesssim \|\phi\|_{\mathcal{C}_{*,\tau}(\mathbb{R}^n\setminus\Sigma)} e^{\zeta_1 t_{j^\prime}^{i^\prime}-2\sigma L}e^{-(\zeta_1+\gamma_\sigma)t_{j^\prime}^{i^\prime}}\\
			&\lesssim \|\phi\|_{\mathcal{C}_{*,\tau}(\mathbb{R}^n\setminus\Sigma)}e^{-\gamma_\sigma L(1+\xi)}e^{-(\zeta_1+\gamma_\sigma)t_{j^\prime}^{i^\prime}},
		\end{align*}
		where we have used $-\gamma_\sigma<\zeta_1<-\gamma_\sigma+2\sigma$.
		
		\noindent On the other hand, from  \eqref{kernelestimate1}, we recall  
		\begin{equation*}
			Z_{j^{\prime}, \ell^{\prime}}^{i^\prime}(\boldsymbol{a}_j^\prime,\boldsymbol{\lambda}_j^\prime)=\mathcal{O}(|x-x_{i^\prime}|^{-\gamma_\sigma})V_{(x_{i^{\prime}},L_{j^{\prime}}^{i^{\prime}},a_{j^{\prime}}^{i^{\prime}}, \lambda_{i^{\prime}})}(-\ln |x|)^{1+\frac{2}{n-2\sigma}} \quad {\rm for} \quad \ell^\prime\in \{1, \dots, n\}.
		\end{equation*}
		Using the last identity, one can get similar estimates to the ones above. 
		In conclusion, it is straightforward to check 
		\begin{align*}
			|I_0|&=\left|\int_{\mathbb{R}^n}\left[\phi-(-\Delta)^{-\sigma}(f^{\prime}_\sigma(\bar{u}_{(\boldsymbol{x},\boldsymbol{L},\boldsymbol{a}_j,\boldsymbol{\lambda}_j)}) \phi)\right] f_{\sigma}^{\prime}(U_{(x_{i^\prime},L_{j^\prime}^{i^\prime}\lambda_{j^{\prime}}^{i^\prime},a_{j^{\prime}}^{i^\prime})})Z_{j^{\prime}, \ell^{\prime}}^{i^\prime}(\boldsymbol{a}_j^\prime,\boldsymbol{\lambda}_j^\prime) \ud x\right|\\
			&\lesssim  \|\phi\|_{\mathcal{C}_{*,\tau}(\mathbb{R}^n\setminus\Sigma)}e^{-\gamma_\sigma L(1+\xi)}e^{-(\zeta_1+\gamma_\sigma)t_{j^\prime}^{i^\prime}},
		\end{align*}
		which proves the claim.
		
		\noindent{\bf Claim 2:} The following estimate holds
		\begin{equation*}
			|I_1|\lesssim \|h\|_{\mathcal{C}_{**,\tau}(\mathbb{R}^n\setminus\Sigma)}e^{-(\zeta_1+\gamma_\sigma)t_{j^\prime}^{i^\prime}}.
		\end{equation*}
		In fact, we have
		\begin{align*}
			|I_1|
			&=\left|\int_{\mathbb{R}^n} h f_{\sigma}^{\prime}(U_{(x_{i^\prime},L_{j^\prime}^{i^\prime}\lambda_{j^{\prime}}^{i^\prime},a_{j^{\prime}}^{i^\prime})})Z_{j^{\prime}, \ell^{\prime}}^{i^\prime}(\boldsymbol{a}_j^\prime,\boldsymbol{\lambda}_j^\prime) \ud x\right|\\
			&\lesssim \int_{B_1(x_{i^{\prime}}}\|h\|_{\mathcal{C}_{**,\tau}(\mathbb{R}^n\setminus\Sigma)}|x-x_{i^\prime}|^{\zeta_1-\tau} |x-x_{i^\prime}|^{-\gamma_\sigma}\left(e^{-\gamma_\sigma t_{j^\prime}^{i^\prime}}+e^{-(\gamma_\sigma+1) t_{j^\prime}^{i^\prime}}\right)\ud x\\
			&+\sum_{i \neq i^{\prime}} \int_{B_1(x_i)}\|h\|_{\mathcal{C}_{**,\tau}(\mathbb{R}^n\setminus\Sigma)}|x-x_i|^{\zeta_1-\tau}e^{-\gamma_\sigma t_{j^\prime}^{i^\prime}}\ud x\\
			&+\int_{\mathbb{R}^n \backslash \bigsqcup\limits_{i=1}^N B_1(x_i)}\|h\|_{\mathcal{C}_{**,\tau}(\mathbb{R}^n\setminus\Sigma)}|x|^{-(n-2\sigma)}|x|^{-4\sigma}|x|^{-(n-2\sigma)}e^{-\gamma_\sigma t_{j^\prime}^{i^\prime}} \ud x\\
			&\lesssim \|h\|_{\mathcal{C}_{**,\tau}(\mathbb{R}^n\setminus\Sigma)}e^{-(\zeta_1-
				\tau+\gamma_\sigma) t_{j^\prime}^{i^\prime}},
		\end{align*}
		which proves the desired estimate.
		
		\noindent{\bf Claim 3:} The following estimate holds
		\begin{equation*}
			\left\|\sum_{i=1}^N  \displaystyle\sum_{j\in\mathbb N}  \displaystyle\sum_{\ell=0}^n c^{i}_{j,\ell}(\boldsymbol{a}_j,\boldsymbol{\lambda}_j)Z^{i}_{j,\ell}(\boldsymbol{a}_j,\boldsymbol{\lambda}_j)\right\|_{\mathcal{C}_{**,\tau}(\mathbb{R}^n\setminus\Sigma)}\lesssim e^{-\gamma_\sigma L(1+\xi)}\|\phi\|_{\mathcal{C}_{*,\tau}(\mathbb{R}^n\setminus\Sigma)}+\|h\|_{\mathcal{C}_{**,\tau}(\mathbb{R}^n\setminus\Sigma)}.
		\end{equation*}
		
		\noindent As a matter of fact, we first isolate the term $c^{i}_{j,\ell}(\boldsymbol{a}_j,\boldsymbol{\lambda}_j)$ in \eqref{linearizedtimeskernels} by inverting the matrix
		\begin{equation*}
			\int_{\mathbb R^n} f_{\sigma}^{\prime}(U_{(x_{i^\prime},L_{j^\prime}^{i^\prime}\lambda_{j^{\prime}}^{i^\prime},a_{j^{\prime}}^{i^\prime})}){Z}_{j, \ell}^i(\boldsymbol{a}_j,\boldsymbol{\lambda}_j) Z_{j^{\prime}, \ell^{\prime}}^{i^\prime}(\boldsymbol{a}_j^\prime,\boldsymbol{\lambda}_j^\prime)\ud x.
		\end{equation*}
		For this, recall the orthogonality estimates from \eqref{orthogonalitycond1} and \eqref{orthogonalitycond2}, which yields
		\begin{align*}
			\int_{\mathbb{R}^n}f^{\prime}_\sigma(U_{(x_i,L^i_j,\lambda_j^i,a_j^i)})Z_{j, \ell}^i(\boldsymbol{a}_j,\boldsymbol{\lambda}_j) Z_{j, \ell^{\prime}}^i(\boldsymbol{a}_j,\boldsymbol{\lambda}_j) \ud x=C_0\delta_{\ell, \ell^{\prime}} \quad {\rm for} \quad \ell^\prime\in \{1, \dots, n\},
		\end{align*} 
		and
		\begin{align*}
			\int_{\mathbb{R}^n}f^{\prime}_\sigma(U_{(x_i,L^i_j,\lambda_j^i,a_j^i)})Z_{j, \ell}^i(\boldsymbol{a}_j,\boldsymbol{\lambda}_j) Z_{j^{\prime}, \ell^{\prime}}^i(\boldsymbol{a}_j,\boldsymbol{\lambda}_j) \ud x=\mathcal{O}(e^{-\gamma_{\sigma}|t_j^i-t_{j^{\prime}}^i|})\quad  \text{ if } \quad \ell\neq \ell^\prime,
		\end{align*} 
		plus a tiny error. Then using  in \cite[Lemma A.6]{MR2522830} for the inversion of a Toepliz-type operator, one has from \eqref{linearizedtimeskernels} that
		\begin{align*}
			|c_{j,\ell}^i(\boldsymbol{a}_j,\boldsymbol{\lambda}_j)|&\lesssim [e^{-\gamma_\sigma L(1+\xi)}\|\phi\|_{\mathcal{C}_{*,\tau}(\mathbb{R}^n\setminus\Sigma)}
			+\|h\|_{\mathcal{C}_{**,\tau}(\mathbb{R}^n\setminus\Sigma)}]e^{-(\zeta_1-\tau+\gamma_\sigma)t_j^i}\\
			&+\sum_{j'\neq j}[e^{-\gamma_\sigma L(1+\xi)}\|\phi\|_{\mathcal{C}_{*,\tau}(\mathbb{R}^n\setminus\Sigma)}
			+\|h\|_{\mathcal{C}_{**,\tau}(\mathbb{R}^n\setminus\Sigma)}]e^{-\gamma_\sigma(1+o(1))|t_j-t_{j'}|}e^{-(\zeta_1-\tau+\gamma_\sigma)t_{j'}^i}\\
			&\lesssim [e^{-\gamma_\sigma L(1+\xi)}\|\phi\|_{\mathcal{C}_{*,\tau}(\mathbb{R}^n\setminus\Sigma)}
			+\|h\|_{\mathcal{C}_{**,\tau}(\mathbb{R}^n\setminus\Sigma)}]
			e^{-\gamma_\sigma(1+o(1))|t_j-t_{j'}|}e^{-(\zeta_1-\tau+\gamma_\sigma)t_{j}^i}.
		\end{align*}
		
		\noindent Using the estimates \eqref{kernelestimate1} of $Z^{i}_{j,\ell}(\boldsymbol{a}_j,\boldsymbol{\lambda}_j)$ and its equation, we split the integrals as in Step 3 in the proof of Claim 2 in Lemma \ref{lm:quantitativeestimates} and get in $B_1(p_i)$,
		\begin{align*}
			|Z^{i}_{j,\ell}(\boldsymbol{a}_j,\boldsymbol{\lambda}_j)|&=|(-\Delta)^{-\sigma}(f_{\sigma}^{\prime}(U_{(x_i,L^i_j,\lambda_j^i,a_j^i)}) {Z}_{j, \ell}^i(\boldsymbol{a}_j,\boldsymbol{\lambda}_j))|\\
			&\lesssim |x-x_i|^{-\gamma_\sigma+2\sigma}e^{-(\gamma_\sigma+2\sigma)|t^{i}_{j^\prime}-t^{i}_{j}|}\lesssim |x-x_i|^{\zeta_1-\tau}e^{-(\gamma_\sigma+2\sigma)|t^{i}_{j^\prime}-t^{i}_{j}|}.
		\end{align*}
		
		\noindent The above two estimates yield that
		\begin{align*}
			|c_{j,\ell}^{i}(\boldsymbol{a}_j,\boldsymbol{\lambda}_j)Z_{j,\ell}^{i}(\boldsymbol{a}_j,\boldsymbol{\lambda}_j)|\lesssim |x-x_i|^{\zeta_1-\tau}[e^{-\gamma_\sigma L(1+\xi)}\|\phi\|_{\mathcal{C}_{*,\tau}(\mathbb{R}^n\setminus\Sigma)}
			+\|h\|_{\mathcal{C}_{**,\tau}(\mathbb{R}^n\setminus\Sigma)}]e^{-\zeta|t^{i}_{j^\prime}-t^{i}_{j}|}
		\end{align*}
		for some $\zeta>0$.
		
		\noindent For $x\in \mathbb{R}^n \backslash \bigsqcup\limits_{i=1}^N B_1(x_i)$, one has
		\begin{align*}
			|c_{j,\ell}^{i}(\boldsymbol{a}_j,\boldsymbol{\lambda}_j)Z_{j,\ell}^{i}(\boldsymbol{a}_j,\boldsymbol{\lambda}_j)|&\lesssim (\lambda_j^i)^{\gamma_\sigma}|x|^{-(n-2\sigma)}|c_{j,\ell}^{i}(\boldsymbol{a}_j,\boldsymbol{\lambda}_j)|\\
			&\lesssim |x|^{-(n-2\sigma)}[e^{-\gamma_\sigma L(1+\xi)}\|\phi\|_{\mathcal{C}_{*,\tau}(\mathbb{R}^n\setminus\Sigma)}
			+\|h\|_{\mathcal{C}_{**,\tau}(\mathbb{R}^n\setminus\Sigma)}]e^{-\zeta|t^{i}_{j^\prime}-t^{i}_{j}|}.
		\end{align*}
		
		\noindent Combining the above two estimates yields
		\begin{equation*}
			\left\|\sum_{i=1}^N  \displaystyle\sum_{j\in\mathbb N}  \displaystyle\sum_{\ell=0}^n c^{i}_{j,\ell}(\boldsymbol{a}_j,\boldsymbol{\lambda}_j)Z^{i}_{j,\ell}(\boldsymbol{a}_j,\boldsymbol{\lambda}_j)\right\|_{\mathcal{C}_{**,\tau}(\mathbb{R}^n\setminus\Sigma)}\lesssim e^{-\gamma_\sigma L(1+\xi)}\|\phi\|_{\mathcal{C}_{*,\tau}(\mathbb{R}^n\setminus\Sigma)}+\|h\|_{\mathcal{C}_{**,\tau}(\mathbb{R}^n\setminus\Sigma)}.
		\end{equation*}
		The proof of the claim is concluded.
		
		\noindent{\bf Claim 4:} 
		It holds that 
		\begin{equation*}
			\|\phi\|_{\mathcal{C}_{*,\tau}(\mathbb{R}^n\setminus\Sigma)} \lesssim \|\bar{h}\|_{\mathcal{C}_{**,\tau}(\mathbb{R}^n\setminus\Sigma)},
		\end{equation*}
		uniformly on $L\gg1$, where 
		\begin{equation*}
			\bar{h}=h+ \displaystyle\sum_{i=1}^N  \displaystyle\sum_{j\in\mathbb N}  \displaystyle\sum_{\ell=0}^n c_{j, \ell}^i(\boldsymbol{a}_j,\boldsymbol{\lambda}_j){Z}_{j, \ell}^i(\boldsymbol{a}_j,\boldsymbol{\lambda}_j).
		\end{equation*}
		
		\noindent We suppose by contradiction that there exist sequences of functions $\{\bar{h}_k\}_{k\in\mathbb N}\subset \mathcal{C}_{**,\tau}(\mathbb{R}^n\setminus\Sigma)$ and $\{\phi_k\}_{k\in\mathbb N}\subset \mathcal{C}_{*,\tau}(\mathbb{R}^n\setminus\Sigma)$, where $\phi_k=(\mathscr{L}_{\sigma}(\boldsymbol{a},\boldsymbol{\lambda}))^{-1}(\bar{h}_k)$ for all $k\in\mathbb N$
		such that $\|\phi_k\|_{\mathcal{C}_{*,\tau}(\mathbb{R}^n\setminus\Sigma)}=1$ and
		\begin{equation}\label{contradictionassumption}
			\|\bar{h}_k\|_{\mathcal{C}_{**,\tau}(\mathbb{R}^n\setminus\Sigma)} \rightarrow 0 \quad {\rm as} \quad k\rightarrow+\infty.
		\end{equation}
		Here we can write
		\begin{equation*}
			\bar{h}_k=h_k+ \displaystyle\sum_{i=1}^N  \displaystyle\sum_{j\in\mathbb N}  \displaystyle\sum_{\ell=0}^n c_{j, \ell}^{i,k}(\boldsymbol{a}_j,\boldsymbol{\lambda}_j){Z}_{j, \ell}^{i,k}(\boldsymbol{a}_j,\boldsymbol{\lambda}_j),
		\end{equation*}
		where $\{c_{j, \ell}^{i,k}\}_{k\in\mathbb N}\subset \mathcal{C}^\infty({\rm Adm}_\sigma(\Sigma))$, 
		$\{{h}_k\}_{k\in\mathbb N}\subset \mathcal{C}_{**,\tau}(\mathbb{R}^n\setminus\Sigma)$, and $\{\boldsymbol{L}_k\}_{k\in\mathbb N}\subset\mathbb R^N$ such that
		\begin{equation*}
			\max_{1\leqslant i \leqslant N}L^i_k =:|\boldsymbol{L}_k|\rightarrow +\infty \quad {\rm as} \quad k\rightarrow +\infty
		\end{equation*}
		is a sequence of parameters.
		
		\noindent Notice that 
		\begin{equation*}
			\phi_k=(-\Delta)^{-\sigma}(f^{\prime}_\sigma(\bar{u}_{(\boldsymbol{x},\boldsymbol{L},\boldsymbol{a}_j,\boldsymbol{\lambda}_j)}) \phi_k)+\bar{h}_k \quad {\rm in} \quad \mathbb R^n\setminus\Sigma.
		\end{equation*}
		Thus, we need to estimate the first term on the right-hand side of the last equation.
		
		\noindent{\bf Step 1:} If $\ud(x, \Sigma) \geqslant 1$, then 
		\begin{equation*}
			|(-\Delta)^{-\sigma}(f^{\prime}_\sigma(\bar{u}_{(\boldsymbol{x},\boldsymbol{L},\boldsymbol{a}_j,\boldsymbol{\lambda}_j)}) \phi_k)| \leqslant \mathrm{o}(1)\|\phi_k\|_{\mathcal{C}_{*,\tau}(\mathbb{R}^n\setminus\Sigma)}|x|^{-(n-2 \sigma)} \quad {\rm as} \quad L\rightarrow +\infty
		\end{equation*}
		
		\noindent Indeed, notice that 
		\begin{align*}
			(-\Delta)^{-\sigma}(f^{\prime}_\sigma(\bar{u}_{(\boldsymbol{x},\boldsymbol{L},\boldsymbol{a}_j,\boldsymbol{\lambda}_j)}) \phi_k)&=\left[\int_{\ud(y, \Sigma) \leqslant 1}+\int_{\ud(y, \Sigma) \geqslant 1}\right] f^{\prime}_\sigma(\bar{u}_{(\boldsymbol{x},\boldsymbol{L},\boldsymbol{a}_j,\boldsymbol{\lambda}_j)}) \phi_k \mathcal{R}_{\sigma}(x-y) \ud y\\
			&=: I_{1}+I_{2}.
		\end{align*}
		Let us start with estimating the second term on the right-hand side above.
		First, by Lemma~\ref{lm:estimatesdelaunay}, we have 
		\begin{equation*}
			\bar{u}_{(\boldsymbol{x},\boldsymbol{L},\boldsymbol{a}_j,\boldsymbol{\lambda}_j)}(y)=\mathcal{O}(e^{-\gamma_\sigma L})|y|^{-(n-2 \sigma)} \quad {\rm for} \quad \ud(y, \Sigma) \geqslant 1,
		\end{equation*}
		from which we conclude
		\begin{align}\label{estimatereduction1}
			I_{2} &\lesssim e^{-\sigma L}\|\phi_k\|_{\mathcal{C}_{*,\tau}(\mathbb{R}^n\setminus\Sigma)}\int_{\ud(y, \Sigma) \geqslant 1}|y|^{-(n+2 \sigma)} \mathcal{R}_{\sigma}(x-y) \ud y\\\nonumber
			&\lesssim \mathrm{o}(1)\|\phi_k\|_{\mathcal{C}_{*,\tau}(\mathbb{R}^n\setminus\Sigma)}|x|^{-(n-2 \sigma)}.
		\end{align}
		For the first term, we get 
		\begin{align*}\
			I_{1} & \lesssim \sum_{i=1}^N \int_{|y-x_i| \leqslant 1}|y-x_i|^{-2 \sigma}\left(\sum_{j\in\mathbb N} V_{(x_i,L_i,a_j^i, \lambda^i_j)}(-\ln |x|)\right)^{\frac{2n}{n-2\sigma}}\|\phi_k\|_{\mathcal{C}_{*,\tau}(\mathbb{R}^n\setminus\Sigma)}|y-x_i|^{\zeta_1}|x-y|^{-(n-2 \sigma)} \ud y \\
			& \lesssim \|\phi\|_{\mathcal{C}_{*,\tau}(\mathbb{R}^n\setminus\Sigma)}|x|^{-(n-2 \sigma)} \int_{|y-x_i|<1}|y-x_i|^{\zeta_1-2 \sigma}\left(\sum_{j\in\mathbb N} V_{(x_i,L_i,a_j^i, \lambda^i_j)}(-\ln |x|)\right)^{\frac{2n}{n-2\sigma}} \ud y \\
			& \lesssim \|\phi\|_{\mathcal{C}_{*,\tau}(\mathbb{R}^n\setminus\Sigma)}|x|^{-(n-2 \sigma)} \int_0^{+\infty} e^{-(n+\zeta_1-2 \sigma) t}\left(\sum_{j\in\mathbb N} V_{(x_i,L_i,a_j^i, \lambda^i_j)}(-\ln |x|)\right)^{\frac{2n}{n-2\sigma}} \ud t,
		\end{align*}
		which implies
		\begin{align}\label{estimatereduction2}
			I_{1} &\lesssim e^{-(n+\zeta_1-2 \sigma) L}|x|^{-(n-2 \sigma)}\|\phi_k\|_{\mathcal{C}_{*,\tau}(\mathbb{R}^n\setminus\Sigma)}\lesssim \mathrm{o}(1)\|\phi_k\|_{\mathcal{C}_{*,\tau}(\mathbb{R}^n\setminus\Sigma)}|x|^{-(n-2 \sigma)}.
		\end{align}     
		Since $\zeta_1>-\gamma_\sigma$, by combining estimates \eqref{estimatereduction1} and \eqref{estimatereduction2} one concludes the proof Step 1. 
		
		\noindent Subsequently, using Step 1, we also observe that by the estimates above, it holds 
		\begin{equation}
			\sup_{\ud(x, \Sigma) \geqslant 1} |x|^{n-2 \sigma}|\phi_k(x)| \lesssim \|\bar{h}_k\|_{\mathcal{C}_{**,\tau}(\mathbb{R}^n\setminus\Sigma)}+\mathrm{o}(1)\|\phi_k\|_{\mathcal{C}_{*,\tau}(\mathbb{R}^n\setminus\Sigma)} \rightarrow 0 \quad \text { as } \quad L \rightarrow +\infty,
		\end{equation}
		where we also used our contradiction assumption \eqref{contradictionassumption}. Hence, one can find $x_i\in \Sigma$ for some $i\in\{1,\dots,N\}$ such that
		\begin{equation}\label{contradictionreduction}
			\sup_{|x| \leqslant 1}|x-x_i|^{-\zeta_1} \phi_k(x) \geqslant \frac{1}{2} \quad {\rm for \ all} \quad k\in\mathbb N.
		\end{equation}
		In the next step, we prove an estimate contradicting the lower bound above.
		To simplify the notation, we assume that $x_i=0$ and so $|x|<1$.
		
		\noindent{\bf Step 2:} If $|x| \leqslant 1$, then one find $R\gg1$ large enough such that 
		\begin{equation*}
			|(-\Delta)^{-\sigma}(f^{\prime}_\sigma(\bar{u}_{(\boldsymbol{x},\boldsymbol{L},\boldsymbol{a}_j,\boldsymbol{\lambda}_j)}) \phi_k)| \lesssim \mathrm{o}(1)\|\phi_k\|_{\mathcal{C}_{*,\tau}(\mathbb{R}^n\setminus\Sigma)}|x|^{\zeta_1}+e^{-\gamma_\sigma R}+e^{-2R}|x|^{\zeta_1}\quad {\rm as} \quad L\rightarrow +\infty.
		\end{equation*}
		
		\noindent As a matter of fact, similar to before, we have 
		\begin{align*}
			|(-\Delta)^{-\sigma}(f^{\prime}_\sigma(\bar{u}_{(\boldsymbol{x},\boldsymbol{L},\boldsymbol{a}_j,\boldsymbol{\lambda}_j)}) \phi_k)|&=\left[\int_{\ud(y, \Sigma) \leqslant 1}+\int_{\ud(y, \Sigma) \geqslant 1}\right] f^{\prime}_\sigma(\bar{u}_{(\boldsymbol{x},\boldsymbol{L},\boldsymbol{a}_j,\boldsymbol{\lambda}_j)}) \phi_k \mathcal{R}_{\sigma}(x-y) \ud y=: I_{1}+I_{2}.
		\end{align*}
		In the same spirit of the estimates for $\bar{h}$ above, it holds
		\begin{align*}
			I_{1} \lesssim \int_{\ud(y, \Sigma) \geqslant 1} e^{-\sigma L}|y|^{-4 \sigma} {|x-y|^{2 \sigma-n}}\|\phi_k\|_{\mathcal{C}_{*,\tau}(\mathbb{R}^n\setminus\Sigma)}|y|^{-(n-2 \sigma)} \ud y \lesssim \mathrm{o}(1)\|\phi_k\|_{\mathcal{C}_{*,\tau}(\mathbb{R}^n\setminus\Sigma)}|x|^{\zeta_1}.
		\end{align*}
		
		\noindent For the second term, the computation is slightly more involved. 
		We proceed by performing a standard blow-up method.
		Namely, let us consider the family of rescaled functions $\widehat{\phi}_{j}^{i,k}: \mathcal{A}^{i,k}_{j}\rightarrow\mathbb R$ defined on the annular 
		region as
		\begin{equation*}
			\widehat{\phi}_{j}^{i,k}(\hat{x})=\left(\lambda_j^i\right)^{-\zeta_1} \phi_k(\lambda_j^i \hat{x}) \quad {\rm for} \quad k\in\mathbb N,
		\end{equation*}
		where 
		\begin{align*}
			\mathcal{A}^{i,k}_{j}:=\{x\in \mathbb R^n : \tilde{A}^{i,k}_j<|x|<\tilde{A}^{i,k}_{j-1}\}.
		\end{align*}            
		Here we set        
		$\tilde{A}^{i,k}_j={A^{i,k}_j}/{\lambda_j^{i,k}}$, 
		where $A^{i,k}_j=\sqrt{\lambda_{j+1}^{i,k} \lambda_j^{i,k}}$ for $k\in\mathbb N$ and observe 
		\begin{equation*}
			|\mathbb R^n_+\setminus \mathcal{A}^{i,k}_{j}|\rightarrow0 \quad {\rm as} \quad k \rightarrow +\infty.
		\end{equation*}
		Furthermore, it is not hard to check that $\widehat{\phi}_{j}^{i,k}\in\mathcal{C}^{0}(\mathcal{A}^{i,k}_{j})$ satisfy the following rescaled equation
		\begin{equation*}
			\left\{\begin{array}{c}
				\widehat{\phi}_{j}^{i,k}-c_{n, \sigma} \frac{n+2\sigma}{n-2\sigma}
				\displaystyle\int_{\mathbb R^n}\left[\frac{\widehat{\phi}_{j}^{i,k}}{(1+|\hat{x}|^2)^{2 \sigma}} \right]\mathcal{R}_{\sigma}(x-y)\ud y(1+\mathrm{o}(1))=(\lambda_j^{i,k})^{2 \sigma-\zeta_1} \bar{h}(\lambda_j^{i,k} \hat{x}) \quad \text { in } \quad \mathcal{A}^{i,k}_{j}, \\
				\displaystyle\int_{\mathbb{R}^n} \widehat{\phi}_{j}^{i,k} [f_{\sigma}^{\prime}(U_{(0,L_j^{i,k},\lambda_j^{i,k},a_j^{i,k})}) {Z}_{j, \ell}^{i,k}(\lambda_j^{i,k},a_j^{i,k})](\lambda_j^i \hat{x}) \ud \hat{x}=0  \ {\rm for} \ i\in\{1,\dots,N\}, \ j,k\in\mathbb N, \ {\rm and} \  \ell\in\{0,\dots,n\}.
			\end{array}\right.
		\end{equation*}
		
		\noindent Now, we observe the estimate below holds
		\begin{equation*}
			|h_k| \lesssim \|h_k\|_{\mathcal{C}_{**,\tau}(\mathbb{R}^n\setminus\Sigma)}|\lambda_j^{i,k} \hat{x}|^{\zeta_1-2 \sigma} \quad {\rm as} \quad k \rightarrow +\infty.
		\end{equation*}
		Then, there exists $\widehat{\phi}_{j}^{i,\infty}\in\mathcal{C}^{0}(\mathcal{A}^{i,\infty}_{j})$ solution to the following blow-up limit equation
		\begin{align*}
			\begin{cases}
				&\widehat{\phi}_{j}^{i,\infty}-c_{n, \sigma}  \displaystyle\int_{\mathbb R^n} \left[f_\sigma^\prime(U_{(0,1)}(y))\widehat{\phi}_{j}^{i,\infty}(y)\right]\mathcal{R}_{\sigma}(x-y)\ud y=0 \quad \text { in } \quad \mathcal{A}^{i,\infty}_{j}, \\
				&\displaystyle\int_{\mathbb{R}^n} \widehat{\phi}_{j}^{i,\infty} f_\sigma^\prime(U_{(0,1)}) Z_{j,\ell}^{i,\infty}(0,1) \ud x=0.
			\end{cases}
		\end{align*}
		Here $\mathcal{A}^{i,\infty}_{j}=\cup_{k\in\mathbb N}\mathcal{A}^{i,k}_{j}$ is such that $\widehat{\phi}_{j}^{i,k} \rightarrow \widehat{\phi}_{j}^{i,\infty}$ as $k\rightarrow+\infty$ in $\mathcal{A}_R$, where the annular region $\mathcal{A}_R:=\{x\in\mathbb R^n : R^{-1} \leqslant|\hat{x}| \leqslant R\}$ is such that $\mathcal{A}_R\subset\mathcal{A}^{i,\infty}_{j}$  for $R\gg1$ large enough which will be chosen suitably later, where we recall that $U_{(0,1)}=u_{\rm sph}$ is the standard bubble tower solution given by \eqref{standardbubbletower} and $Z_{j,\ell}^{i,k}(\boldsymbol{a}_j,\boldsymbol{\lambda}_j)$ for $\ell\in\{0, \ldots, n\}$, are the corresponding kernels in Definition~\ref{def:kernels}. 
		Therefore, by the non-degeneracy of the standard bubble in Lemma~\ref{lm:nondegeneracy}, we conclude that the blow-up limit is trivial, that is, $\widehat{\phi}_{j}^{i,\infty}\equiv0$ and $\widehat{\phi}_{j}^{i,k} \rightarrow 0$ as $k\rightarrow+\infty$ in $\mathcal{A}_R$.

		\noindent As a consequence, if we consider the original ${\phi}_{k}$, this is equivalent to the uniform convergence
		\begin{equation}\label{convergenceredcution}
			|x|^{-\zeta_1} {\phi}_{k}(x) \rightarrow 0 \quad \text { in } \quad \mathcal{A}^{i,k}_{\infty} \quad \text { as } \quad k \rightarrow +\infty,
		\end{equation}
		where $\mathcal{A}^{i,k}_{\infty}:=\cup_{j\in\mathbb N}\mathcal{A}^{i,k}_{j}$ and $\mathcal{A}^{i,k}_{j}:=\{{R}^{-1}\lambda_j^{i,k}<|x|<R \lambda_j^{i,k}\}$.
		Using the convergence above, we can now estimate the remaining term
		\begin{align*}
			I_{2} & =\int_{\ud(y, \Sigma) \leqslant 1} f^{\prime}_\sigma(\bar{u}_{(\boldsymbol{x},\boldsymbol{L},\boldsymbol{a}_j,\boldsymbol{\lambda}_j)}) \phi_k \mathcal{R}_{\sigma}(x-y) \ud y\\
			& \lesssim \sum_{j\in\mathbb N}\left[\int_{\mathcal{A}^{i,k}_{j}}+\int_{(\mathcal{A}^{i,k}_{j})^c}\right] f^{\prime}_\sigma(\bar{u}_{(\boldsymbol{x},\boldsymbol{L},\boldsymbol{a}_j,\boldsymbol{\lambda}_j)}) \phi_k \mathcal{R}_{\sigma}(x-y)\ud y =: I_{21}+I_{22},
		\end{align*}
		where $(\mathcal{A}^{i,k}_{j})^c:=\{y \in \mathbb R^n : 0<\ud(y, \Sigma) \leqslant 1\} \setminus \mathcal{A}^{i,k}_{j}$.
		
		\noindent First, again from Lemma~\ref{lm:estimatesdelaunay}, we know
		\begin{equation*}
			\bar{u}_{(\boldsymbol{x},\boldsymbol{L},\boldsymbol{a}_j,\boldsymbol{\lambda}_j)}(y)=|y|^{-\gamma_\sigma}\left(\sum_{j\in\mathbb N} V_{(x_i,L_i,a_j^i, \lambda^i_j)}(-\ln |x|)\right)(1+\mathrm{o}(1)) \quad {\rm for} \quad 0<\ud(y, \Sigma) < 1,
		\end{equation*}
		from which we get
		\begin{equation*}
			\sum_{j\in\mathbb N} V_{(x_i,L_i,a_j^i, \lambda^i_j)}(-\ln |x|) \lesssim e^{-\gamma_\sigma R} \quad {\rm in} \quad (\mathcal{A}^{i,k}_{\infty})^c.
		\end{equation*}        
		Hence, the summation on the left-hand side of the last equation can be made small enough by choosing $R\gg1$ large enough but uniform on $k\gg1$, which in turn implies
		\begin{equation*}
			I_{22} \lesssim e^{-2 R} \int_{(\mathcal{A}^{i,k}_{\infty})^c}|y|^{-2 \sigma}\|\phi\|_{\mathcal{C}_{*,\tau}(\mathbb{R}^n\setminus\Sigma)}|y|^{\zeta_1} {|x-y|^{n-2 \sigma}} \ud y \lesssim e^{-2 R}|x|^{\zeta_1}.
		\end{equation*}
		Additionally, from \eqref{convergenceredcution}, it is direct to see 
		\begin{align*}
			I_{21} & \lesssim \sum_{j\in\mathbb N}\int_{\mathcal{A}^{i,k}_{j}}
			|\phi_k||y|^{-\zeta_1} |y|^{\zeta_1-2 \sigma}|x-y|^{2 \sigma-n}\left(\sum_{j\in\mathbb N} V_{(x_i,L_i,a_j^i, \lambda^i_j)}(-\ln |x|)\right)^{\frac{2n}{n-2\sigma}} \ud y \\
			& \lesssim \mathrm{o}(1) \int_{\mathcal{A}^{i,k}_{j}} {|y|^{\zeta_1-2 \sigma}}{|x-y|^{2 \sigma-n}} \ud y\\
			&\lesssim \mathrm{o}(1)|x|^{\zeta_1} .
		\end{align*}
		The proof of this step is then finished.
		
		\noindent Finally, from Step 2, we must have $|x|^{-\zeta_1} \phi_k(x)=\mathrm{o}(1)$ as $k \rightarrow +\infty,$
		which is a contradiction with \eqref{contradictionreduction}. 
		This completes the proof of Claim 4.
		
		\noindent{\bf Claim 5:} For any $\bar{h}\in\mathcal{C}_{**,\tau}(\mathbb{R}^n\setminus\Sigma)$, one can find a unique solution $\phi\in \mathcal{C}_{*,\tau}(\mathbb{R}^n\setminus\Sigma)$ to \eqref{auxiliarysystemdual}.
		
		\noindent First, we consider the space
		\begin{equation*}
			\mathscr{H}^{\perp}(\mathbb{R}^n)=\left\{\phi \in H^\sigma(\mathbb{R}^n): \int_{\mathbb{R}^n} \phi\overline{Z}_{j, \ell}^i(\boldsymbol{a}_j,\boldsymbol{\lambda}_j) \ud x=0 \ {\rm for} \ (i,j,\ell)\in\mathcal{I}_{\infty}\right\}.
		\end{equation*}
		Notice that Eq. \eqref{auxiliarysystemdual} may be reformulated in terms of $\phi$ to become
		\begin{equation}\label{adjointequation}
			\phi+\mathscr{K}(\phi)=\bar{h} \quad \text { in } \quad \mathscr{H}^{\perp}(\mathbb{R}^n),
		\end{equation}
		where $\bar{h}$ is defined by duality and $\mathscr{K}: \mathscr{H}^{\perp}(\mathbb{R}^n) \rightarrow \mathscr{H}^{\perp}(\mathbb{R}^n)$ is a linear compact operator. 
		Using Fredholm alternative, showing that \eqref{auxiliarysystemdual} has a unique solution for each $\bar{h}$ is equivalent to finding a unique solution for $\bar{h}=0$ to \eqref{adjointequation}, which in turn follows from Claim 4. 
		
		The proof is now a consequence of Claims 4 and 5.
	\end{proof}
	
	
	As a consequence of the last result, we can state the last lemma 
	
	\begin{lemma}\label{lm:solutionoperator}
		Let $\sigma \in(1,+\infty)$, $n>2\sigma$ and $N\geqslant 2$.
		Assume that $(\boldsymbol{a}_j,\boldsymbol{\lambda}_j)\in       {\rm Adm}_{\sigma}(\Sigma)$ is an admissible configuration as in Definition~\ref{def:balancedparameters} with $\bar{u}_{(\boldsymbol{x},\boldsymbol{L},\boldsymbol{a}_j,\boldsymbol{\lambda}_j)}\in{\rm Apx}_{\sigma}(\Sigma)$ their associated approximate solution as in Definition~\ref{def:approximatesolution}.      
		Then, there exists a bounded right-inverse for the linearized operator
		$(\mathscr{L}_{\sigma}(\boldsymbol{a}_j,\boldsymbol{\lambda}_j))^{-1}:\mathcal{C}_{**,\tau}(\mathbb{R}^n\setminus\Sigma)\rightarrow \mathcal{C}_{*,\tau}(\mathbb{R}^n\setminus\Sigma)$.
		Moreover, the following estimate holds
		\begin{equation*}
			\|\phi\|_{\mathcal{C}_{*,\tau}(\mathbb{R}^n\setminus\Sigma)} \lesssim \|\mathscr{L}_{\sigma}(\boldsymbol{a}_j,\boldsymbol{\lambda}_j)^{-1}(\phi)\|_{\mathcal{C}_{**,\tau}(\mathbb{R}^n\setminus\Sigma)}.
		\end{equation*}
		uniformly on $L\gg1$ large.     
	\end{lemma}
	
	\subsubsection{Fixed-point argument}
	
	We prove our main result using a standard perturbation method.
	The main idea is to apply a contraction theorem for the operator $ \mathcal{N}_\sigma(\boldsymbol{x},\boldsymbol{L},\boldsymbol{a}_j,\boldsymbol{\lambda}_j)(\phi)=\mathcal{N}_\sigma(\bar{u}_{(\boldsymbol{x},\boldsymbol{L},\boldsymbol{a}_j,\boldsymbol{\lambda}_j)}+\phi)$ on the suitably weighted norms introduced in Definition~\ref{def:suitablespaces}. 
	
	\begin{proposition}\label{prop:lyapunovschmidtreduction}
		Let $\sigma \in(1,+\infty)$, $n>2\sigma$ and $N\geqslant 2$.
		Assume that $(\boldsymbol{a}_j,\boldsymbol{\lambda}_j)\in{\rm Adm}_{\sigma}(\Sigma)$ is an admissible configuration as in Definition~\ref{def:balancedparameters} with $\bar{u}_{(\boldsymbol{x},\boldsymbol{L},\boldsymbol{a}_j,\boldsymbol{\lambda}_j)}\in{\rm Apx}_{\sigma}(\Sigma)$ their associated approximate solution as in Definition~\ref{def:approximatesolution}.    
		Then, for $L\gg1$ large enough and $\zeta_1<0$ satisfying \eqref{weightassumption}, there exists $\{c_{j, \ell}^i(\boldsymbol{a}_j,\boldsymbol{\lambda}_j)\}_{(i,j,k)\in\mathcal{I}_\infty}\subset\mathbb R$ and a solution            
		$\phi\in \mathcal{C}_{*,\tau}(\mathbb{R}^n\setminus\Sigma)$ to 
		\begin{align}\tag{$\mathcal{Q}^{\prime}_{2\sigma,\boldsymbol{a},\boldsymbol{\lambda}}$}\label{ourequationprojected}
			\left\{\begin{array}{l}
				\mathcal{N}_\sigma(\boldsymbol{x},\boldsymbol{L},\boldsymbol{a}_j,\boldsymbol{\lambda}_j)(\phi)=\displaystyle\displaystyle\sum_{i=1}^N  \displaystyle\sum_{j\in\mathbb N}  \displaystyle\sum_{\ell=0}^n  c_{j, \ell}^i(\boldsymbol{a}_j,\boldsymbol{\lambda}_j)\overline{Z}_{j, \ell}^i(\boldsymbol{a}_j,\boldsymbol{\lambda}_j) \quad {\rm in} \quad \mathbb{R}^n\setminus\Sigma, \\
				\displaystyle\int_{\mathbb{R}^n} \phi\overline{Z}_{j, \ell}^i(\boldsymbol{a}_j,\boldsymbol{\lambda}_j) \ud x=0 \quad {\rm for} \quad (i,j,\ell)\in\mathcal{I}_\infty,
			\end{array}\right.                   
		\end{align}
		where $\{\overline{Z}_{j, \ell}^i(\boldsymbol{a}_j,\boldsymbol{\lambda}_j)\}_{(i,j,\ell)\in\mathcal{I}_\infty}\subset \mathcal{C}^{2\sigma}(\mathbb R^n\setminus\Sigma)$ is the family of approximate normalized corkernels given by Definition~\ref{def:kernels}.
		Moreover, one has the estimate
		\begin{equation*}
			\|\phi_{(\boldsymbol{x},\boldsymbol{L},\boldsymbol{a}_j,\boldsymbol{\lambda}_j)}\|_{\mathcal{C}_{*,\tau}(\mathbb{R}^n\setminus\Sigma)} \lesssim e^{-\gamma_\sigma L(1+\xi)}
		\end{equation*}
		for some $\xi>0$ depending uniformly on $L\gg1$ large.
	\end{proposition}
	
	\begin{proof}
		According to Lemma~\ref{lm:fredholmtheory}, the solution operator $(\mathscr{L}_{\sigma}(\boldsymbol{a},\boldsymbol{\lambda}))^{-1}:\mathcal{C}^0(\mathbb{R}^n\setminus\Sigma)\rightarrow\mathcal{C}^{2\sigma}(\mathbb{R}^n\setminus\Sigma)$ defined in Lemma~\ref{lm:solutionoperator} is well-defined.
		Notice that $\bar{u}_{(\boldsymbol{x},\boldsymbol{L},\boldsymbol{a}_j,\boldsymbol{\lambda}_j)}+\phi$ with $\phi\in \mathcal{C}_*(\mathbb R^n\setminus\Sigma)$ solves equation \eqref{ourequationprojected}, if and only if, it solves the fixed-point problem below
		\begin{equation*}
			\phi=\mathscr{B}_{\sigma}(\boldsymbol{a}_j,\boldsymbol{\lambda}_j)(\phi) \quad {\rm in} \quad \mathbb R^n\setminus\Sigma.
		\end{equation*}
		Here $\mathscr{B}_{\sigma}(\boldsymbol{a}_j,\boldsymbol{\lambda}_j):\mathcal{C}^0(\mathbb{R}^n\setminus\Sigma)\rightarrow\mathcal{C}^{2\sigma}(\mathbb{R}^n\setminus\Sigma)$ is given by
		\begin{equation}\label{contractionmap}
			\mathscr{B}_{\sigma}(\boldsymbol{a}_j,\boldsymbol{\lambda}_j)(\phi):=-(\mathscr{L}_{\sigma}(\boldsymbol{a},\boldsymbol{\lambda}))^{-1}(\mathscr{N}_\sigma(\bar{u}_{(\boldsymbol{x},\boldsymbol{L},\boldsymbol{a}_j,\boldsymbol{\lambda}_j)}))+(\mathscr{L}_{\sigma}(\boldsymbol{a},\boldsymbol{\lambda}))^{-1}(\mathscr{R}_{\sigma}(\boldsymbol{a}_j,\boldsymbol{\lambda}_j)(\phi))
		\end{equation}
		and $\mathscr{R}_{\sigma}(\boldsymbol{a}_j,\boldsymbol{\lambda}_j):\mathcal{C}^0(\mathbb{R}^n\setminus\Sigma)\rightarrow\mathcal{C}^{2\sigma}(\mathbb{R}^n\setminus\Sigma)$ is given by
		\begin{equation}\label{errortermfixedpoint}
			\mathscr{R}_{\sigma}(\boldsymbol{a}_j,\boldsymbol{\lambda}_j)(\phi):=(-\Delta)^{-\sigma}[\mathscr{Q}_{\sigma}(\boldsymbol{a}_j,\boldsymbol{\lambda}_j)(\phi)],
		\end{equation}
		where
		\begin{equation}\label{errortermfixedpointintegrand}
			\mathscr{Q}_{\sigma}(\boldsymbol{a}_j,\boldsymbol{\lambda}_j)(\phi):=|f_\sigma(\bar{u}_{(\boldsymbol{x},\boldsymbol{L},\boldsymbol{a}_j,\boldsymbol{\lambda}_j)}+\phi)-f_\sigma(\bar{u}_{(\boldsymbol{x},\boldsymbol{L},\boldsymbol{a}_j,\boldsymbol{\lambda}_j)})-f^{\prime}_\sigma(\bar{u}_{(\boldsymbol{x},\boldsymbol{L},\boldsymbol{a}_j,\boldsymbol{\lambda}_j)}) \phi|.
		\end{equation}
		
		First, by definition, one has
		\begin{equation*}
			\|\mathscr{B}_{\sigma}(\boldsymbol{a}_j,\boldsymbol{\lambda}_j)(\phi)\|_{\mathcal{C}_{*,\tau}(\mathbb{R}^n\setminus\Sigma)} \lesssim \|\mathscr{N}_{\sigma}(\bar{u}_{(\boldsymbol{x},\boldsymbol{L},\boldsymbol{a}_j,\boldsymbol{\lambda}_j)})\|_{\mathcal{C}_{**,\tau}(\mathbb{R}^n\setminus\Sigma)}+\|\mathscr{R}_{\sigma}(\boldsymbol{a}_j,\boldsymbol{\lambda}_j)(\phi)\|_{\mathcal{C}_{**,\tau}(\mathbb{R}^n\setminus\Sigma)}.
		\end{equation*}
		Second, fixing a large $C\gg0$, we define the set
		\begin{equation*}
			\mathcal{B}_C=\left\{\phi \in \mathcal{C}_{*,\tau}(\mathbb{R}^n\setminus\Sigma): \begin{aligned}
				\|\phi\|_{\mathcal{C}_{*,\tau}(\mathbb{R}^n\setminus\Sigma)} \lesssim e^{-\gamma_\sigma L(1+\xi)} \ &{\rm and} \ \int_{\mathbb{R}^n} \phi f_\sigma^\prime(U_{(x_i,L^i_j,\lambda_j^i,a_j^i)})Z_{j, \ell}^i(\boldsymbol{a}_j,\boldsymbol{\lambda}_j) \ud x=0 
				\\
				{\rm for \ all } \ &(i,j,\ell)\in\mathcal{I}_\infty \end{aligned}\right\}. 
		\end{equation*}
		We observe that the first term of the right-hand side of the equation above is estimated in Lemma~\ref{lm:quantitativeestimates}.    
		Hence, we are left to provide similar estimates for the remaining term.
		
		\noindent{\bf Claim 1:} The following estimate holds
		\begin{equation*}
			\|\mathscr{R}_{\sigma}(\boldsymbol{a}_j,\boldsymbol{\lambda}_j)(\phi)\|_{\mathcal{C}_{**,\tau}(\mathbb{R}^n\setminus\Sigma)} \lesssim e^{-\frac{(n-6\sigma){\gamma_\sigma L}}{n-2\sigma}}\|\phi\|_{\mathcal{C}_{*,\tau}(\mathbb{R}^n\setminus\Sigma)}^2+\|\phi\|_{\mathcal{C}_{*,\tau}(\mathbb{R}^n\setminus\Sigma)}^{\frac{n+2\sigma}{n-2\sigma}} \leqslant \mathrm{o}(1)\|\phi\|_{\mathcal{C}_{*,\tau}(\mathbb{R}^n\setminus\Sigma)}.
		\end{equation*}
		
		\noindent Now, for any $\phi \in \mathcal{B}_C$ 
		We must estimate the $L^\infty$-norm of the error term in \eqref{errortermfixedpoint}.
		\noindent We start by estimating the term \eqref{errortermfixedpointintegrand}. Indeed, it is not hard to check 
		\begin{equation*}
			|\mathscr{Q}_{\sigma}(\boldsymbol{a}_j,\boldsymbol{\lambda}_j)(\phi)| \lesssim 
			\begin{cases}
				\bar{u}_{(\boldsymbol{x},\boldsymbol{L},\boldsymbol{a}_j,\boldsymbol{\lambda}_j)}^{\frac{n-6\sigma}{n-2\sigma}} \phi^2, & \text { if } \ |\bar{u}_{(\boldsymbol{x},\boldsymbol{L},\boldsymbol{a}_j,\boldsymbol{\lambda}_j)}| \geqslant\frac{1}{4} \phi, \\ 
				\phi^{\frac{n+2\sigma}{n-2\sigma}}, & \text { if } \ |\bar{u}_{(\boldsymbol{x},\boldsymbol{L},\boldsymbol{a}_j,\boldsymbol{\lambda}_j)}| \leqslant \frac{1}{4} \phi.
			\end{cases}
		\end{equation*}
		Again, the proof will be divided into two steps.
		First, we estimate the integrand in \eqref{errortermfixedpointintegrand}.
		
		\noindent{\bf Step 1:} The estimate below holds
		\begin{align*}
			& |\mathscr{Q}_{\sigma}(\boldsymbol{a}_j,\boldsymbol{\lambda}_j)(\phi)|\lesssim
			\begin{cases}
				\sum_{i=1}^N\left(\|\phi\|_{\mathcal{C}_{*,\tau}(\mathbb{R}^n\setminus\Sigma)}^2+\|\phi\|_{\mathcal{C}_{*,\tau}(\mathbb{R}^n\setminus\Sigma)}^{\frac{n-6\sigma}{n-2\sigma}}\right)|x-x_i|^{\zeta_1-2 \sigma}, \ &{\rm if} \ 0<\ud(x, \Sigma)<1,\\
				|x|^{-(n+2 \sigma)}\left(e^{\frac{(n-6\sigma){\gamma_\sigma L}}{n-2\sigma}}\|\phi\|_{\mathcal{C}_{*,\tau}(\mathbb{R}^n\setminus\Sigma)}^2+\|\phi\|_{\mathcal{C}_{*,\tau}(\mathbb{R}^n\setminus\Sigma)}^{\frac{n+2\sigma}{n-2\sigma}}\right), \ &{\rm if} \ \ud(x, \Sigma) \geqslant 1.
			\end{cases}
		\end{align*}
		
		\noindent As a matter of fact,  by our construction, it holds:
		\begin{itemize}
			\item[(i)] If $\operatorname{dist}(x, \Sigma)<1$, then 
			\begin{align*}
				&|\mathscr{Q}_{\sigma}(\boldsymbol{a}_j,\boldsymbol{\lambda}_j)(\phi)|\\ 
				&\lesssim \sum_{i=1}^N\left(\|\phi\|_{\mathcal{C}_{*,\tau}(\mathbb{R}^n\setminus\Sigma)}^2 \bar{u}_{(\boldsymbol{x},\boldsymbol{L},\boldsymbol{a}_j,\boldsymbol{\lambda}_j)}^{-\frac{n+2\sigma}{n-2\sigma}}|x-x_i|^{2 \zeta_1}+\|\phi\|_{\mathcal{C}_{*,\tau}(\mathbb{R}^n\setminus\Sigma)}^{\frac{n+2\sigma}{n-2\sigma}}|x-x_i|^{\frac{(n+2\sigma)\zeta_1}{n-2\sigma}}\right) \\
				& \lesssim \sum_{i=1}^N\left[\|\phi\|_{\mathcal{C}_{*,\tau}(\mathbb{R}^n\setminus\Sigma)}^2|x-x_i|^{\zeta_1-2 \gamma}|x-x_i|^{\zeta_1+\gamma_\sigma}+\|\phi\|_{\mathcal{C}_{*,\tau}(\mathbb{R}^n\setminus\Sigma)}^{\frac{n+2\sigma}{n-2\sigma}}|x-x_i|^{\zeta_1-2 \sigma}|x-x_i|^{\frac{(n+2\sigma)\zeta_1}{n-2\sigma}-\zeta_1+2 \sigma}\right] \\
				& \lesssim \sum_{i=1}^N\left(\|\phi\|_{\mathcal{C}_{*,\tau}(\mathbb{R}^n\setminus\Sigma)}^2+\|\phi\|_{\mathcal{C}_{*,\tau}(\mathbb{R}^n\setminus\Sigma)}^{\frac{n+2\sigma}{n-2\sigma}}\right)|x-x_i|^{\zeta_1-2 \sigma}.
			\end{align*}
			\item[(ii)] If $\operatorname{dist}(x, \Sigma) \geqslant 1$, then
			\begin{align*}
				|\mathscr{Q}_{\sigma}(\boldsymbol{a}_j,\boldsymbol{\lambda}_j)(\phi)| & \lesssim \|\phi\|_{\mathcal{C}_{*,\tau}(\mathbb{R}^n\setminus\Sigma)}^2 \bar{u}_{(\boldsymbol{x},\boldsymbol{L},\boldsymbol{a}_j,\boldsymbol{\lambda}_j)}^{\frac{n-6\sigma}{n-2\sigma}}|x|^{-4\gamma_\sigma}+\|\phi\|_{\mathcal{C}_{*,\tau}(\mathbb{R}^n\setminus\Sigma)}^{\frac{n+2\sigma}{n-2\sigma}}|x|^{-\frac{n+2\sigma}{n-2\sigma}(n-2 \sigma)}\\
				&\lesssim |x|^{-(n+2 \sigma)}\left(e^{\frac{(n-6\sigma){\gamma_\sigma L}}{n-2\sigma}}\|\phi\|_{\mathcal{C}_{*,\tau}(\mathbb{R}^n\setminus\Sigma)}^2+\|\phi\|_{\mathcal{C}_{*,\tau}(\mathbb{R}^n\setminus\Sigma)}^{\frac{n+2\sigma}{n-2\sigma}}\right).
			\end{align*}
		\end{itemize} 
		By combining the above two estimates, the proof of the first step is concluded.
		
		\noindent Second, we can use the estimate above the handle the term \eqref{errortermfixedpoint}.

		\noindent{\bf Step 2:} The estimate below holds
		\begin{align*}
			& |\mathscr{R}_{\sigma}(\boldsymbol{a}_j,\boldsymbol{\lambda}_j)(\phi)|\lesssim
			\begin{cases}
				\sum_{i=1}^N\left(\|\phi\|_{\mathcal{C}_{*,\tau}(\mathbb{R}^n\setminus\Sigma)}^2+\|\phi\|_{\mathcal{C}_{*,\tau}(\mathbb{R}^n\setminus\Sigma)}^{\frac{n+2\sigma}{n-2\sigma}}\right)|x-x_i|^{\zeta_1-\tau}, \ &{\rm if} \ 0<\ud(x, \Sigma)<1,\\
				|x|^{-(n-2 \sigma)}\left(e^{\frac{(n-6\sigma){\gamma_\sigma L}}{n-2\sigma}}\|\phi\|_{\mathcal{C}_{*,\tau}(\mathbb{R}^n\setminus\Sigma)}^2+\|\phi\|_{\mathcal{C}_{*,\tau}(\mathbb{R}^n\setminus\Sigma)}^{\frac{n+2\sigma}{n-2\sigma}}\right), \ &{\rm if} \ \ud(x, \Sigma) \geqslant 1.
			\end{cases}
		\end{align*}
		
		\noindent Indeed we need to plug Step 1 into \eqref{errortermfixedpoint} and proceed as in Claim 2 of Lemma \ref{lm:quantitativeestimates}.
		
		\color{black}
		
		\noindent The proof follows by recalling the definition of weighted norms in \eqref{weightednorm1} and \eqref{weightednorm2}. 
		
		\noindent{\bf Claim 2:} The map $\mathscr{B}_{\sigma}(\boldsymbol{a}_j,\boldsymbol{\lambda}_j):\mathcal{B}_C\rightarrow \mathcal{B}_C$ is a contraction.
		
		\noindent Now we consider two functions $\phi_1, \phi_2 \in \mathcal{B}_C$.
		From the estimates in Step 1 combined with the ones in Lemma~\ref{lm:quantitativeestimates}, it is easy to see for $L\gg1$ large, one has 
		\begin{equation*}
			\left\|\mathscr{B}_{\sigma}(\boldsymbol{a}_j,\boldsymbol{\lambda}_j)(\phi_1)-\mathscr{B}_{\sigma}(\boldsymbol{a}_j,\boldsymbol{\lambda}_j)(\phi_2)\right\|_{\mathcal{C}_{**,\tau}(\mathbb{R}^n\setminus\Sigma)} \leqslant \mathrm{o}(1)\|\phi_1-\phi_2\|_{\mathcal{C}_{*,\tau}(\mathbb{R}^n\setminus\Sigma)}.
		\end{equation*}
		Therefore, one can be combined, and the proof of the claim is concluded.
		
		
		Based on the last step, we can use the standard Banach contracting argument to obtain the desired fixed point; this completes the proof of the proposition.
	\end{proof}
	
	\section{Estimates for the projections on the approximate null space}\label{sec:estimatesparameters}
	In this section, we provide some estimates             
	related to the coefficient functions seen as functions on the perturbation parameters, namely
	\begin{equation}\label{coefficientexplictly}
		\{c_{j, \ell}^i\}_{(i,j,\ell)\in\mathcal{I}_{\infty}}\subset \mathcal{C}^\infty(\ell^\infty_\tau(\mathbb R^{(n+1)N})),  
	\end{equation}
	which were
	obtained in Section~\ref{sec:approximatesolutions}, where we recall $\mathcal{I}_{\infty}:=\{1\dots, N\}\times \mathbb N\times \{0,\dots, n\}$ is the total index set.    
	More precisely, we notice that from Proposition~\ref{prop:lyapunovschmidtreduction}, whenever $(\boldsymbol{a}_j,\boldsymbol{\lambda}_j)\in{\rm Adm}_{\sigma}(\Sigma)$ is a set of admissible parameters in Definition~\ref{def:balancedparameters}, one can find a solution $\bar{u}_{(\boldsymbol{x},\boldsymbol{L},\boldsymbol{a}_j,\boldsymbol{\lambda}_j)}\in{\rm Apx}_{\sigma}(\Sigma)$ (or simply $\bar{u}_{(\boldsymbol{x},\boldsymbol{L})}\in{\rm Apx}_{\sigma}(\Sigma)$) to perturbed equation \eqref{ourequationprojected} as  in Definition~\ref{def:approximatesolution}.      
	Here we recall 
	\begin{equation*}
		(\boldsymbol{x},\boldsymbol{L})\in{\rm Comp}_\sigma(\Sigma)\mapsto (\boldsymbol{q},\boldsymbol{a}_0,\boldsymbol{R})\in{\rm Bal}_\sigma(\Sigma)\mapsto (\boldsymbol{a}_j,\boldsymbol{\lambda}_j)\in{\rm Adm}_\sigma(\Sigma)\mapsto \bar{u}_{(\boldsymbol{x},\boldsymbol{L})}\in{\rm Apx}_\sigma(\Sigma);
	\end{equation*}
	or, equivalently, 
	\begin{equation*}
		\bar{u}_{(\boldsymbol{x},\boldsymbol{L})}=(\Upsilon_{\rm sol}\circ \Upsilon_{\rm per}\circ \Upsilon_{\rm conf})(\boldsymbol{x},\boldsymbol{L}).
	\end{equation*}
	is the explicit construction of approximate solutions.      
	Thus, applying the Lyapunov--Schmidt reduction, one can see that finding solutions to our original problem \eqref{ourintegralequationdual} is equivalent to solving the following infinite-dimensional system 
	\begin{equation}\label{todasystem}\tag{$\mathcal{S}_{2\sigma,\Sigma}$}
		\beta_{j, \ell}^i(\boldsymbol{a}_j,\boldsymbol{\lambda}_j)=0 \quad {\rm for} \quad (i,j,\ell)\in \mathcal{I}_{\infty}.
	\end{equation}
	Here the projection functions $\{\beta_{j, \ell}^i\}_{(i,j,\ell)\in\mathcal{I}_{\infty}}\subset \mathcal{C}^\infty(\ell^\infty_\tau(\mathbb R^{(n+1)N}))$ are given by
	\begin{equation}\label{projectedfunctinsbalanced}
		\beta_{j, \ell}^i(\boldsymbol{a}_j,\boldsymbol{\lambda}_j)=\int_{\mathbb{R}^n}\mathcal{N}_\sigma(\boldsymbol{x},\boldsymbol{L},\boldsymbol{a}_j,\boldsymbol{\lambda}_j) \overline{Z}_{j, \ell}^i(\boldsymbol{a}_j,\boldsymbol{\lambda}_j) \ud x \quad {\rm for} \quad (i,j,\ell)\in \mathcal{I}_{\infty},
	\end{equation}
	where we recall that 
	\begin{equation*}
		\mathcal{N}_\sigma(\boldsymbol{x},\boldsymbol{L},\boldsymbol{a}_j,\boldsymbol{\lambda}_j)=\bar{u}_{(\boldsymbol{x},\boldsymbol{L},\boldsymbol{a}_j,\boldsymbol{\lambda}_j)}-(-\Delta)^{-\sigma}(f_\sigma\circ \bar{u}_{(\boldsymbol{x},\boldsymbol{L},\boldsymbol{a}_j,\boldsymbol{\lambda}_j)})
	\end{equation*}
	and the family of cokernels $\{\overline{Z}_{j, \ell}^i(\boldsymbol{a}_j,\boldsymbol{\lambda}_j)\}_{(i,j,\ell)\in\mathcal{I}_{\infty}}\subset \mathcal{C}^{0}(\mathbb R^n\setminus\Sigma)$ given by Definition~\ref{def:kernels}.
	The idea is to find a set of configuration parameters such that its associated sequence of perturbation parameters satisfies Syst. \eqref{todasystem}.
	Then, the balancing conditions \eqref{balancing1} and \eqref{balancing2} will allow us to perturb this special configuration to find a true solution to our problem.
	In some sense, this is a discrete version of the perturbation technique we applied to approximate Delaunay solutions by half-bubble tower solutions. 
	
	\subsection{Projection on the normalized approximate kernels}
	Initially, we prove the decay of the functions defined in \eqref{projectedfunctinsbalanced}.
	For this, we shall consider two cases.
	Namely,   when the perturbation sequence of parameters is trivial, that is, $(\boldsymbol{a}_j,\boldsymbol{\lambda}_j)=(\boldsymbol{0},\boldsymbol{1})$.
	Notice that indeed $(\boldsymbol{0},\boldsymbol{1})\in{\rm Adm}_{\sigma}(\Sigma)$ is an admissible perturbation sequence.
	
	With this definition, we have the following estimate:
	
	\begin{lemma}\label{lm:estimatesbetaspecial}
		Let $\sigma \in(1,+\infty)$, $n>2\sigma$ and $N\geqslant 2$. 
		Assume that $(\boldsymbol{a}_j,\boldsymbol{\lambda}_j)=(\boldsymbol{0},\boldsymbol{1})\in{\rm Adm}_{\sigma}(\Sigma)$ is a set of trivial perturbations.  
		Then, there exists two constants $A_2>0, A_3<0$ independent of $L\gg1$ given by \eqref{A2} and \eqref{A3} such that the following estimates hold
		\begin{itemize}
			\item[{\rm (i)}] If $j=0$ and
			\begin{itemize}
				\item[{\rm (a)}] $\ell=0$, then one has
				\begin{align*}
					& \beta_{0, 0}^i(\boldsymbol{0},\boldsymbol{1})=-c_{n, \sigma}q_i\left[A_2 \sum_{i^{\prime} \neq i}|x_{i^{\prime}}-x_i|^{-(n-2 \sigma)}(R^i R^{i^{\prime}})^{\gamma_\sigma} q_{i^{\prime}}-q_i\right] e^{-\gamma_\sigma L}(1+\mathrm{o}(1))+\mathcal{O}(e^{-\gamma_\sigma L(1+\xi)});
				\end{align*}
				\item[{\rm (b)}] $\ell\in\{1,\dots,n\}$, then one has
				\begin{align*}
					{\beta}_{0, \ell}^i(\boldsymbol{0},\boldsymbol{1})=c_{n, \sigma} \lambda_0^i\left[A_3 \sum_{i^{\prime} \neq i} \frac{\left(x_{i^{\prime}}-x_i\right)_\ell}{|x_{i^{\prime}}-x_i|^{n-2 \sigma+2}}(R^i R^{i^{\prime}})^{\gamma_\sigma} q_{i^{\prime}} q_i e^{-\gamma_\sigma L}+\mathcal{O}(e^{-\gamma_\sigma L(1+\xi)})\right].
				\end{align*}
			\end{itemize}
			\item[{\rm (ii)}] If $j \geqslant 1$ and
			\begin{itemize}
				\item[{\rm (a)}] $\ell=0$, then one has
				\begin{align*}
					{\beta}_{j, 0}^i(\boldsymbol{0},\boldsymbol{1})=\mathcal{O}\left(e^{-\gamma_\sigma(1+\xi)} e^{-\nu t_j^i}\right);
				\end{align*}
				\item[{\rm (b)}] $\ell\in\{1,\dots,n\}$, then one has
				\begin{align*}
					{\beta}_{j, \ell}^i(\boldsymbol{0},\boldsymbol{1})=\mathcal{O}\left(e^{-\gamma_\sigma L(1+\xi)} e^{-(1+\nu) t_j^i}\right),
				\end{align*}
			\end{itemize}
			where $\nu=\min \left\{\zeta_1+\gamma_{\sigma}, \frac{\gamma_{\sigma}}{2}\right\}$ independent of $L\gg1$ large and $\xi>0$.
		\end{itemize}
	\end{lemma}
	
	\begin{proof}
		Recall the definition of the cokernel $\overline{Z}_{j, \ell}^i(\boldsymbol{0},\boldsymbol{1})$, we have
		\begin{align*}
			\beta^{i}_{j,\ell}(\boldsymbol{0},\boldsymbol{1})=\int_{\mathbb{R}^n}[(-\Delta)^{\sigma}\bar{u}_{(\boldsymbol{x}, \boldsymbol{L}, \boldsymbol{0},\boldsymbol{1})}- (f_\sigma\circ \bar{u}_{(\boldsymbol{x}, \boldsymbol{L}, \boldsymbol{0},\boldsymbol{1})})] Z_{j, \ell}^i(\boldsymbol{0},\boldsymbol{1})\ud x.
		\end{align*}
		Then The proof is the same as in \cite[Lemma~4.1]{MR4104278}, and we omit the details.
	\end{proof}
	
	It is not hard to check that only the perturbations of $(\boldsymbol{a}_j,\boldsymbol{\lambda}_j)$ will affect the numbers $\beta_{j, \ell}^i(\boldsymbol{a}_j,\boldsymbol{\lambda}_j)$, that is, we can get the same estimates for $\beta_{j, \ell}^{i^*}(\boldsymbol{0},\boldsymbol{1})$ for an admissible sequence of parameters. 
	Consequently, one can see that for any fixed $i_*\in\{1,\dots,N\}$, the corresponding $x_{i_*}\in \Sigma$ and $L_{i_*}\in\mathbb R_+$ are also fixed. 
	Hence, if we consider the approximate solution defined as 
	\begin{equation}\label{approximatesolutionfixed}
		\bar{u}^*_{(\boldsymbol{x}, \boldsymbol{L}, \boldsymbol{0},\boldsymbol{1})}:=\bar{u}_{(x_{i_*},L_{i_*},\boldsymbol{0},\boldsymbol{1})},
	\end{equation}
	then the same estimates for ${\beta}^{i_*}_{j, \ell}(\boldsymbol{0},\boldsymbol{1})$ in the above lemma are still in force.  
	
	Next, we estimate the coefficients in \eqref{coefficientexplictly} for a general admissible perturbation sequence. 
	So fixing $i_*\in\{1,\dots,n\}$, we would like to study the estimates for the variations of $\beta_{j, \ell}^{i_*}(\boldsymbol{a}_j,\boldsymbol{\lambda}_j)$. 
	Before, let us introduce some terminology.
	For any fixed $i_*\in\{1,\dots,N\}$ and $j_*\in\mathbb N$, we let $\mathrm{e}_{j_*}^{i_*}\in \mathbb{R}^n$ and let $r_{j_*}^{i_*}\in \mathbb R$ be such that 
	\begin{equation*}
		|\mathrm{e}_{j_*}^{i_*}| \lesssim (\lambda_{j_*}^{i_*})^2 \quad \text { and } \quad |r_{j_*}^{i_*}| \lesssim e^{-\tau t_{j_*}^{i_*}}.
	\end{equation*}
	In this fashion, we define the variation of the perturbation sequence as 
	\begin{equation*}
		(\boldsymbol{a}_j(t),\boldsymbol{\lambda}_j(t))=
		\begin{cases}
			(\boldsymbol{0},\boldsymbol{1}), & \ {\rm if} \ j\neq j_*\\
			t \mathrm{e}_{j_*}^{i_*}+ R^{i_*}(1+t r_{j_*}^{i_*}) & \ {\rm if} \ j= j_*.
		\end{cases}
	\end{equation*}
	Finally, we set 
	\begin{equation}
		\bar{u}^*_{(\boldsymbol{a}_j(t),\boldsymbol{\lambda}_j(t))}(x)=  \sum_{i=1}^N\left(\widehat{U}^{+,*}_{(\boldsymbol{a}_j(t),\boldsymbol{\lambda}_j(t))}+\chi_i(x-t\mathrm{e}_{j_*}^{i_*})\phi^*(\boldsymbol{a}_j(t),\boldsymbol{\lambda}_j(t))\right)(x),
	\end{equation}
	where
	\begin{equation*}
		\widehat{U}^{+,*}_{(\boldsymbol{a}_j(t),\boldsymbol{\lambda}_j(t))}(x)=\sum_{j\in\mathbb N}\widehat{U}^{*}_{(x_{i_*},L_{i_*},{a}^{i_*}_j(t),{\lambda}^{i^*}_j(t))}(x)
	\end{equation*}
	with
	\begin{equation*}
		\widehat{U}^{*}_{(x_{i_*},L_{i_*},{a}^{i_*}_j(t),{\lambda}^{i^*}_j(t))}(x)=U_{R^{i_*}(1+t r_{i_*}^{j_*})}(x-t e_{i_*}^{j_*}),
	\end{equation*}
	where $0<\tau\ll 1$ is sufficiently small.
	
	with this definition in hands, we have the following estimates
	\begin{lemma}\label{lm:estimatestechnical}
		Let $\sigma \in(1,+\infty)$, $n>2\sigma$ and $N\geqslant 2$.
		Assume that $(\boldsymbol{a}_j,\boldsymbol{\lambda}_j)\in{\rm Adm}_{\sigma}(\Sigma)$ is an admissible configuration as in Definition~\ref{def:balancedparameters} and $\Psi:\mathbb R\rightarrow\mathbb R$ be defined by \eqref{auxiliaryfunctioninteraction}.  
		Then, there exists constants $A_1<0, A_2>0, A_3<0$ independent of $L\gg1$ given by \eqref{A1}, \eqref{A2}, and \eqref{A3} such that the following estimates hold
		\begin{itemize}
			\item[{\rm (i)}] If $i=i_*$, $j_* \neq j$, and
			\begin{itemize}
				\item[{\rm (a)}] $\ell=0$, then one has
				\begin{align*}
					&\left.{\partial_t}\right|_{t=0}\int_{\mathbb{R}^n} \mathscr{N}_{\sigma}(\bar{u}^*_{(\boldsymbol{a}_j(t),\boldsymbol{\lambda}_j(t))}) Z_{j, \ell}^i(\boldsymbol{a}_j,\boldsymbol{\lambda}_j) \ud x\\
					&=-c_{n, \sigma} {\partial_t}\left[A_2 \sum_{i^{\prime} \neq i}|x_{i^{\prime}}-x_i|^{-(n-2 \sigma)}(\lambda_0^i \lambda_0^{i^{\prime}})^{\gamma_\sigma}+\Psi(|\ln {\lambda_0^i}/{\lambda_1^i}|) \frac{\ln{\lambda_1^i}/{\lambda_0^i}}{|\ln {\lambda_1^i}/{\lambda_0^i}|}\right]+\mathcal{O}(e^{-\gamma_\sigma L(1+\xi)});
				\end{align*}
				\item[{\rm (b)}] $\ell\in\{1,\dots,n\}$, then one has
				\begin{align*}
					&\left.{\partial_t}\right|_{t=0}\int_{\mathbb{R}^n} \mathscr{N}_{\sigma}(\bar{u}^*_{(\boldsymbol{a}_j(t),\boldsymbol{\lambda}_j(t))}) Z_{j, \ell}^i(\boldsymbol{a}_j,\boldsymbol{\lambda}_j) \ud x\\
					&=c_{n,\sigma} \lambda_{0}^{i}{\partial_t} \left[A_3 \sum_{i^{\prime} \neq i} \frac{x_{i^{\prime}}-x_i}{|x_{i^{\prime}}-x_i|^{n-2 \sigma+2}}(\lambda_0^i \lambda_0^{i^{\prime}})^{\gamma_\sigma}+A_1\frac{\min\{{\lambda_0^i}/{\lambda_1^i}, {\lambda_1^i}/{\lambda_0^i}\}^{\gamma_\sigma}}{|\max\{\lambda_{j^{\prime}}^i, \lambda_{j_*}^i\}|^{2}} {t \mathrm{e}_\ell}\right]+\mathcal{O}(\lambda_0^i e^{-\gamma_\sigma L(1+\xi)});
				\end{align*}
			\end{itemize}
			\item[{\rm (ii)}] If $i=i_*$, $j_*=j \geqslant 1$
			\begin{itemize}
				\item[{\rm (a)}] $\ell=0$, then one has
				\begin{align*}
					\left.{\partial_t}\right|_{t=0}\int_{\mathbb{R}^n} \mathscr{N}_{\sigma}(\bar{u}^*_{(\boldsymbol{a}_j(t),\boldsymbol{\lambda}_j(t))}) Z_{j, \ell}^i(\boldsymbol{a}_j,\boldsymbol{\lambda}_j) \ud x=-c_{n, \sigma} {\partial_t}\left[\Psi(|\ln {\lambda_0^i}/{\lambda_1^i}|) \frac{\ln{\lambda_1^i}/{\lambda_0^i}}{|\ln {\lambda_1^i}/{\lambda_0^i}|}\right]+\mathcal{O}(e^{-\gamma_\sigma L(1+\xi)});
				\end{align*}
				\item[{\rm (b)}] $\ell\in\{1,\dots,n\}$, then one has
				\begin{align*}
					\left.{\partial_t}\right|_{t=0}\int_{\mathbb{R}^n} \mathscr{N}_{\sigma}(\bar{u}^*_{(\boldsymbol{a}_j(t),\boldsymbol{\lambda}_j(t))}) Z_{j, \ell}^i(\boldsymbol{a}_j,\boldsymbol{\lambda}_j) \ud x=c_{n,\sigma} \lambda_{j}^{i}{\partial_t} \left[A_1\frac{\min\{{\lambda_0^i}/{\lambda_1^i}, {\lambda_1^i}/{\lambda_0^i}\}^{\gamma_\sigma}}{|\max\{\lambda_{j^{\prime}}^i, \lambda_{j_*}^i\}|^{2}} {t \mathrm{e}_\ell}\right]+\mathcal{O}(\lambda_j^i e^{-\gamma_\sigma L(1+\xi)}e^{-\tau t_{j_*}});
				\end{align*}
			\end{itemize}
			for some $\nu>0$ independent of $0<\tau\ll1$ small, $L\gg1$ large, and $\xi>0$.
		\end{itemize}
	\end{lemma}
	
	\begin{proof}
		The proof is the same as in \cite[Lemma~4.3]{MR4104278}; thus, we omit the details.
	\end{proof}
	
	Next, we study the case of a general sequence of perturbation. 
	In this proof, it will be fundamental to use the fact that our sequence of parameters is admissible.
	
	
	\begin{lemma}\label{lm:estimatesbetageneral}
		Let $\sigma \in(1,+\infty)$, $n>2\sigma$ and $N\geqslant 2$ and $(\boldsymbol{a}_j,\boldsymbol{\lambda}_j)\in{\rm Adm}_{\sigma}(\Sigma)$ be an admissible configuration as in Definition~\ref{def:balancedparameters}.  
		There exists constants $A_1,A_2>0, A_3<0$ independent of $L\gg1$ given by \eqref{A1}, \eqref{A2} and \eqref{A3} such that the following estimates hold
		\begin{itemize}
			\item[{\rm (i)}] If $j=0$ and
			\begin{itemize}
				\item[{\rm (a)}] $\ell=0$, then one has
				\begin{align*}
					\beta_{0,0}^i(\boldsymbol{a}_j,\boldsymbol{\lambda}_j)&=-c_{n, \sigma} q_i\left[A_2 \sum_{i^{\prime} \neq i}|x_{i^{\prime}}-x_i|^{-(n-2 \sigma)}(R_0^i R_0^{i^{\prime}})^{\gamma_\sigma} q_{i^{\prime}}-\left(\frac{R_1^i}{R_0^i}\right)^{\gamma_\sigma} q_i\right] e^{-\gamma_\sigma L}(1+\mathrm{o}(1))\\
					&+\mathcal{O}(e^{-\gamma_\sigma L(1+\xi)}).
				\end{align*}
				\item[{\rm (b)}] $\ell\in\{1,\dots,n\}$, then one has
				\begin{align*}
					\beta_{0, \ell}^i(\boldsymbol{a}_j,\boldsymbol{\lambda}_j)&=c_{n, \sigma} \lambda_0^i\left[A_3 \sum_{i^{\prime} \neq i} \frac{(x_{i^{\prime}}-x_i)_\ell}{|x_{i^{\prime}}-x_i|^{n-2 \sigma+2}}(R_0^i R_0^{i^{\prime}})^{\gamma_\sigma} q_{i^{\prime}}+A_0\left(\frac{R_1^i}{R_0^i}\right)^{\gamma_\sigma} \frac{a_0^i-a_1^i}{\left(\lambda_0^i\right)^2} q_i\right] q_i e^{-\gamma_\sigma L}\\
					&+\mathcal{O}(\lambda_0^i e^{-\gamma_\sigma L(1+\xi)})  \ {\rm for} \ \ell\in\{1,\dots,n\}.
				\end{align*}
			\end{itemize}
			\item[{\rm (ii)}] If $j \geqslant 1$ and
			\begin{itemize}
				\item[{\rm (a)}] $\ell=0$, then one has
				\begin{align*}
					{\beta}_{j, 0}^i(\boldsymbol{a}_j,\boldsymbol{\lambda}_j)=\mathcal{O}\left(e^{-\gamma_\sigma L(1+\xi)} e^{-\nu t_j^i}+e^{-\gamma_\sigma L(1+\xi)} e^{-\tau t_{j-1}^i}\right). 
				\end{align*}
				\item[{\rm (b)}] $\ell\in\{1,\dots,n\}$, then one has
				\begin{align*}
					{\beta}_{j, \ell}^i(\boldsymbol{a}_j,\boldsymbol{\lambda}_j)=\mathcal{O}\left(\lambda_j^i e^{-\gamma_\sigma L(1+\xi)} e^{-\nu t_j^i}+\lambda_j^i e^{-\gamma_\sigma L} e^{-\tau t_{j-1}^i}\right) \ {\rm for} \ \ell\in\{1,\dots,n\},
				\end{align*}
			\end{itemize}
			where $\nu=\min \left\{\zeta_1+\gamma_{\sigma}, \frac{\gamma_{\sigma}}{2}\right\}$ independent of $L\gg1$ large and $\xi>0$.
		\end{itemize}
	\end{lemma}
	
	\begin{proof}
		For the same reason as in Lemma \ref{lm:estimatesbetaspecial}, The proof is the same as in \cite[Lemma~4.4]{MR4104278}, and we omit the details.
	\end{proof}
	
	\subsection{Derivative of the projection on the normalized approximate kernels}
	Here we estimate the variations of the projection functions in \eqref{projectedfunctinsbalanced} with respect to the perturbation parameters.       
	As before, for any fixed $i_*\in\{1,\dots,N\}$ one has $x_{i_*}\in \Sigma$ and $L_{i_*}\in\mathbb R_+$, we denote by $u^*_{(\boldsymbol{x}, \boldsymbol{L}, \boldsymbol{0},\boldsymbol{1})}\in \mathcal{C}^{2\sigma+\alpha}(\mathbb R^n\setminus\Sigma)$ an approximate solution to \eqref{ourintegralequation}. 
	Using Proposition~\ref{prop:lyapunovschmidtreduction}, we know that by performing the Lyapunov--Schmidt reduction method, there exists an error function 
	\begin{equation}\label{errorfunctionfixed}
		\phi^*_{(\boldsymbol{a}_j,\boldsymbol{\lambda}_j)}:=\phi_{(x_{i_*},L_{i_*}\boldsymbol{a}_j,\boldsymbol{\lambda}_j)}\in \mathcal{C}_*(\mathbb R^n\setminus\Sigma).
	\end{equation}
	In this direction, it also makes sense to define 
	\begin{equation}\label{nonlinearoperatorfixed}
		\mathscr{N}^*_{\sigma}(\boldsymbol{a}_j,\boldsymbol{\lambda}_j)(\phi):=\mathscr{N}_{\sigma}({u}^*_{(\boldsymbol{a}_j,\boldsymbol{\lambda}_j)}+\phi)=({u}^*_{(\boldsymbol{a}_j,\boldsymbol{\lambda}_j)}+\phi)-(-\Delta)^{-\sigma}[f_\sigma\circ ({u}^*_{(\boldsymbol{a}_j,\boldsymbol{\lambda}_j)}+\phi)].
	\end{equation}
	Furthermore, let us introduce the linearized operator applied to this approximate solution $\mathscr{L}^*_{\sigma}(\boldsymbol{a}_j,\boldsymbol{\lambda}_j):\mathcal{C}^{\alpha}(\mathbb R^n\setminus\Sigma)\rightarrow\mathcal{C}^{2\sigma+\alpha}(\mathbb R^n\setminus\Sigma)$ given by 
	\begin{equation*}
		\mathscr{L}^*_{\sigma}(\boldsymbol{a}_j,\boldsymbol{\lambda}_j)(\phi)=\phi-(-\Delta)^{-\sigma}(f^{\prime}_\sigma\circ \bar{u}^*_{(\boldsymbol{a}_j,\boldsymbol{\lambda}_j)})\phi.
	\end{equation*}
	From now on, it will be convenient to denote the new coordinate system as 
	\begin{equation}\label{newcoordsystemfixed}
		\xi_{j, 0}^i=r_j^i \quad {\rm and} \quad {\rm for} \quad \xi_{j, \ell}^i=a_{j, \ell}^i \quad {\rm for} \quad \ell\in\{1, \dots, n\}.
	\end{equation}
	We study the variations with respect to \eqref{newcoordsystemfixed}.
	
	We now need to study the derivative of $c^{i_*}_{j, \ell}(\boldsymbol{a}_j,\boldsymbol{\lambda}_j)$ with respect to the variations of the parameters in \eqref{newcoordsystemfixed}. 
	Initially, we consider the most straightforward model case when there is only one point singularity at $\Sigma=\{0\}$ and $\bar{u}_{(\boldsymbol{x},\boldsymbol{L})}=\bar{u}_{(0,L_j)}\in \mathcal{C}^{2\sigma}(\mathbb R^n\setminus\{0\})$ is the associated approximate solution, {\it i.e.}, the  Delaunay solution from \eqref{exactdelaunay}. 
	We recall that this solution satisfies \eqref{ourequationprojected} with $\phi_{(0,L_j)}\equiv 0$ and vanishing right-hand side. 
	We define
	\begin{equation}\label{specialprojectioncoefficients}
		\beta^{i_*}_{j, \ell}(\boldsymbol{a}_j,\boldsymbol{\lambda}_j):=\int_{\mathbb{R}^n}\mathscr{N}^*_{\sigma}(\boldsymbol{a}_j,\boldsymbol{\lambda}_j) \overline{Z}^*_{j, \ell}(\boldsymbol{a}_j,\boldsymbol{\lambda}_j) \ud x .
	\end{equation}
	In this setting, one can still perform the reduction in Proposition~\ref{prop:lyapunovschmidtreduction} to find a perturbed solution in the form $u=\bar{u}_{(0,L_j)}+{\phi}$ of the following equation
	\begin{align}
		\label{ourequationprojectedfixed}
		\left\{\begin{array}{l}
			\mathcal{N}^*_\sigma(\boldsymbol{a}_j,\boldsymbol{\lambda}_j)(\phi)= \displaystyle\sum_{j\in\mathbb N}  \displaystyle\sum_{\ell=0}^n  c_{j, \ell}^{i_*}(\boldsymbol{a}_j,\boldsymbol{\lambda}_j)\overline{Z}_{j, \ell}^{i_*}(\boldsymbol{a}_j,\boldsymbol{\lambda}_j) \quad {\rm in} \quad \mathbb{R}^n\setminus\Sigma, \\
			\displaystyle\int_{\mathbb{R}^n} \phi\overline{Z}_{j, \ell}^{i_*}(\boldsymbol{a}_j,\boldsymbol{\lambda}_j) \ud x=0 \quad {\rm for} \quad (\ell,j)\in \{0,\dots,n\}\times\mathbb N.
		\end{array}\right.                   
	\end{align}
	In conclusion, for any fixed $i_*\in\{1,\dots,N\}$. 
	Let us denote by $\bar{u}^*_{(\boldsymbol{a}_j,\boldsymbol{\lambda}_j)}, {\phi}^*_{(\boldsymbol{a}_j,\boldsymbol{\lambda}_j)}$ the pair satisfying the infinite-dimensional reduced equation \eqref{ourequationprojectedfixed}.
	Notice that the reason to start with the trivial configuration $u_{(0,L_j)}\in \mathcal{C}^{2\sigma}(\mathbb R^n\setminus\{0\})$ is that we will have the identification $\partial_{\xi_{j, \ell}^i}\beta_{j, \ell}^i=\lim_{j\rightarrow+\infty}\partial_{\xi_{j, \ell}}\beta^{i_*}_{j, \ell}$, where we set $\xi_{j, \ell}:=\xi^{i_*}_{j, \ell}$. 
	
	Let us begin with the lemma below
	\begin{lemma}
		Let $\sigma \in(1,+\infty)$, $n>2\sigma$ and $N\geqslant 2$.
		Assume that $(\boldsymbol{a}_j,\boldsymbol{\lambda}_j)\in{\rm Adm}_{\sigma}(\Sigma)$ is an admissible configuration as in Definition~\ref{def:balancedparameters} with $\bar{u}_{(\boldsymbol{x},\boldsymbol{L},\boldsymbol{a}_j,\boldsymbol{\lambda}_j)}\in{\rm Apx}_{\sigma}(\Sigma)$ their associated approximate solution as in Definition~\ref{def:approximatesolution}.  
		Then, for $L\gg1$ is sufficiently large, one has
		\begin{equation*}
			|\partial_{\xi_{j, \ell}^i}{\phi}_{(\boldsymbol{x},\boldsymbol{L},\boldsymbol{a}_j,\boldsymbol{\lambda}_j)}(x)| \lesssim 
			\begin{cases} e^{-\gamma_\sigma L(1+\xi)}|x-x_i|^{-\gamma_\sigma} e^{-\nu|t^i-t_j^i|}, & {\rm if} \ \ell=0, \\ {(\lambda_j^i)^{-1}} e^{-\gamma_\sigma L(1+\xi)}|x-x_i|^{-\gamma_\sigma} e^{-\sigma|t^i-t_j^i|}, & {\rm if} \ \ell\in\{1\dots,n\},\end{cases}
			\quad {\rm in} \quad B_1(x_i),
		\end{equation*}
		where $\nu=\min \left\{\zeta_1+\gamma_{\sigma}, \frac{\gamma_{\sigma}}{2}\right\}$ independent of $L\gg1$ large and $\xi>0$.
	\end{lemma}
	
	\begin{proof}
		The proof is the same as in \cite[Lemma~5.1]{MR4104278}; thus, we omit the details.
	\end{proof}
	
	We remark that an estimate similar to this will hold for the pair $\bar{u}^*_{(\boldsymbol{a}_j,\boldsymbol{\lambda}_j)}, {\phi}^*_{(\boldsymbol{a}_j,\boldsymbol{\lambda}_j)}$. 
	The last lemma shows the suitable weighted H\"older spaces for this setting. 
	\begin{definition}
		Let $\sigma \in(1,+\infty)$, $n>2\sigma$ and $N\geqslant 2$. 
		For any $\alpha\in(0,1)$, let us introduce two new weighted norms
		\begin{align*}
			\|\phi\|_{\mathcal{C}_{*,\nu}(\mathbb{R}^n\setminus\Sigma)}&=\||x-x_i|^{\gamma_\sigma} e^{\nu|t^i-t_j^i|} \phi\|_{\mathcal{C}^{2 \sigma+\alpha}(B_1(x_i))}+\sum_{i^{\prime} \neq i}\||x-x_{i^{\prime}}|^{\gamma_\sigma} \phi\|_{\mathcal{C}^{2 \sigma+\alpha}(B_1(x_{i^\prime}))}\\
			&+\||x|^{n-2 \sigma} \phi\|_{\mathcal{C}^{2 \sigma+\alpha}(\mathbb{R}^n \setminus \cup_{i^{\prime}} B_1(x_{i^{\prime}}))}
		\end{align*}
		and 
		\begin{align*}
			\|\phi\|_{\mathcal{C}_{**,\nu}(\mathbb{R}^n\setminus\Sigma)}&=\||x-x_i|^{\gamma_\sigma^\prime} e^{\nu|t^i-t_j^i|} \phi\|_{\mathcal{C}^{2 \sigma+\alpha}(B_1(x_i))}+\sum_{i^{\prime} \neq i}\||x-x_{i^{\prime}}|^{\gamma_\sigma^\prime} \phi\|_{\mathcal{C}^{2 \sigma+\alpha}(B_1(x_{i^\prime}))}\\
			&+\||x|^{n+2 \sigma} \phi\|_{\mathcal{C}^{2 \sigma+\alpha}(\mathbb{R}^n \setminus \cup_{i^{\prime}} B_1(x_{i^{\prime}}))}
		\end{align*}
		where we recall that $t^i=-\ln|x-x_i|$ and $0<\nu\ll1$ is a small positive constant to be determined later.
		We also denote by $\mathcal{C}_{*,\nu}(\mathbb{R}^n\setminus\Sigma)$ and $\mathcal{C}_{**,\nu}(\mathbb{R}^n\setminus\Sigma)$ the corresponding weighted Hölder spaces. 
	\end{definition}
	
	In the light of Lemma~\ref{lm:estimatestechnical}, one can prove the estimate below
	\begin{lemma}\label{lm:derivateprojeestimates1}
		Let $\sigma \in(1,+\infty)$, $n>2\sigma$ and $N\geqslant 2$. 
		Assume that $(\boldsymbol{a}_j,\boldsymbol{\lambda}_j)\in{\rm Adm}_{\sigma}(\Sigma)$ is an admissible configuration as in Definition~\ref{def:balancedparameters} and $\Psi:\mathbb R\rightarrow\mathbb R$ be defined as \eqref{auxiliaryfunctioninteraction}.  
		Then, there exists constants $A_1<0, A_2>0, A_3<0$ independent of $L\gg1$ given by \eqref{A1}, \eqref{A2}, and \eqref{A3} and $\xi>0$ such that the following estimates hold:
		\begin{itemize}
			\item[{\rm (a)}] If $\ell=0$, then one has
			\begin{align*}
				\left.\partial_{r_j}\beta_{j, 0}(\boldsymbol{a}_j,\boldsymbol{\lambda}_j)\right|_{(\boldsymbol{a}_j,\boldsymbol{\lambda}_j)=(\boldsymbol{0},\boldsymbol{1})} & =-2 c_{n, \sigma} \Psi^{\prime}(L)+\mathcal{O}(e^{-\gamma_\sigma L(1+\xi)}),\\
				\left.\partial_{r_j}\beta_{j-1, 0}(\boldsymbol{a}_j,\boldsymbol{\lambda}_j)\right|_{(\boldsymbol{a}_j,\boldsymbol{\lambda}_j)=(\boldsymbol{0},\boldsymbol{1})} & =c_{n, \sigma} \Psi^{\prime}(L)+\mathcal{O}(e^{-\gamma_\sigma L(1+\xi)}), \\
				\left.\partial_{r_j}\beta_{j+1, 0}(\boldsymbol{a}_j,\boldsymbol{\lambda}_j)\right|_{(\boldsymbol{a}_j,\boldsymbol{\lambda}_j)=(\boldsymbol{0},\boldsymbol{1})}  & =c_{n, \sigma} \Psi^{\prime}(L)+\mathcal{O}(e^{-\gamma_\sigma L(1+\xi)}),\\
				\left.\partial_{r_j}\beta_{j_*, 0}(\boldsymbol{a}_j,\boldsymbol{\lambda}_j)\right|_{(\boldsymbol{a}_j,\boldsymbol{\lambda}_j)=(\boldsymbol{0},\boldsymbol{1})} &=\mathcal{O}(e^{-\gamma_\sigma(1+\xi)} e^{-\sigma|t_{j_*}-t_j|}) \quad  {\rm for}\quad |j_{*}-j| \geqslant 2.
			\end{align*}
			\item[{\rm (b)}] If $\ell\in\{1,\dots,n\}$, then one has
			\begin{align*}
				\left.\partial_{a_{j, \ell}}\beta^{i_*}_{j, \ell}(\boldsymbol{a}_j,\boldsymbol{\lambda}_j)\right|_{(\boldsymbol{a}_j,\boldsymbol{\lambda}_j)=(\boldsymbol{0},\boldsymbol{1})}=c_{n, \sigma} \lambda_j \sum_{j^{\prime} \neq j} \frac{\min \{{\lambda_{j^{\prime}}}/{\lambda_j},{\lambda_j}/{\lambda_{j^{\prime}}}\}^{\gamma_\sigma}}{\max \{\lambda_{j^{\prime}}^2, \lambda_j^2\}}+\mathcal{O}(e^{-\gamma_\sigma L(1+\xi)}) \quad {\rm if} \quad j \neq j_*
			\end{align*}
			and 
			\begin{align*}
				\left.\partial_{a_{j, \ell}}\beta^{i_*}_{j, \ell}(\boldsymbol{a}_j,\boldsymbol{\lambda}_j)\right|_{(\boldsymbol{a}_j,\boldsymbol{\lambda}_j)=(\boldsymbol{0},\boldsymbol{1})}=c_{n, \sigma} \lambda_j \frac{\min \{{\lambda_{j_*}}/{\lambda_j},{\lambda_j}/{\lambda_{j_*}}\}^{\gamma_\sigma}}{\max \{\lambda_{j_*}^2, \lambda_j^2\}}+\mathcal{O}(e^{-\gamma_\sigma L(1+\xi)}).
			\end{align*}
			In addition, it follows
			\begin{equation*}
				\left.\partial_{\xi_{j, \ell^{\prime}}}\beta^{i_*}_{j_*, \ell}(\boldsymbol{a}_j,\boldsymbol{\lambda}_j)\right|_{(\boldsymbol{a}_j,\boldsymbol{\lambda}_j)=(\boldsymbol{0},\boldsymbol{1})}=0 \quad {\rm if } \quad \ell \neq \ell^{\prime}.
			\end{equation*}
		\end{itemize}
	\end{lemma}
	
	\begin{proof}
		The proof is the same as in \cite[Lemma~5.2]{MR4104278}; thus, we omit the details.
	\end{proof}
	
	The strategy to proof the desired estimates for the case of a generally admissible perturbation sequence $(\boldsymbol{a}_j,\boldsymbol{\lambda}_j)\in{\rm Adm}_{\sigma}(\Sigma)$ is first to study the trivial case $(\boldsymbol{a}_j,\boldsymbol{\lambda}_j)=(\boldsymbol{0},\boldsymbol{1})$ and then perform a by-now standard perturbation of parameters method.
	
	For simplicity, we only state the latter case:
	
	\begin{lemma}\label{lm:estimatesbetaderivativesspecial}
		Let $\sigma \in(1,+\infty)$, $n>2\sigma$ and $N\geqslant 2$.
		Assume that $(\boldsymbol{a}_j,\boldsymbol{\lambda}_j)\in{\rm Adm}_{\sigma}(\Sigma)$ is an admissible configuration as in Definition~\ref{def:balancedparameters} with $\bar{u}_{(\boldsymbol{x},\boldsymbol{L},\boldsymbol{a}_j,\boldsymbol{\lambda}_j)}\in{\rm Apx}_{\sigma}(\Sigma)$ their associated approximate solution as in Definition~\ref{def:approximatesolution}.  
		Then, for any $i_*\in\{1,\dots,N\}$ fixed and $j\in\mathbb N$, we have the following estimate
		\begin{equation*}
			\left\|\partial_{\xi_{j, \ell}^i}\left(\phi_{(\boldsymbol{x},\boldsymbol{L},\boldsymbol{a}_j,\boldsymbol{\lambda}_j)}-\phi^*_{(\boldsymbol{a}_j,\boldsymbol{\lambda}_j)}\right)\right\|_{\mathcal{C}_{*,\nu}(\mathbb{R}^n\setminus\Sigma)} \lesssim 
			\begin{cases}
				e^{-\gamma_\sigma L(1+\xi)} e^{-\nu t_j^i} \quad {\rm for } \quad \ell=0\\ 
				{(\lambda_j^i)^{-1}} e^{-\gamma_\sigma L(1+\xi)} e^{-\nu t_j^i} \quad  {\rm for } \quad \ell\in\{1,\dots,n\}.
			\end{cases}
		\end{equation*}
		In particular, it follows
		\begin{equation*}
			\left|\partial_{\xi_{j, \ell}^i}\left(\beta_{j^{\prime}, \ell^{\prime}}^i(\boldsymbol{a}_j,\boldsymbol{\lambda}_j)-\beta_{j^{\prime}, \ell^{\prime}}(\boldsymbol{a}_j,\boldsymbol{\lambda}_j)\right)\right| \lesssim 
			\begin{cases}
				e^{-\gamma_\sigma L(1+\xi)} e^{-\nu t_j^i} e^{-\nu|t_j^i-t_{j^{\prime}}^i|} \quad {\rm for } \quad \ell=0\\
				{(\lambda_j^i)^{-1}} e^{-\gamma_\sigma(1+\xi)} e^{-\nu t_j^i} e^{-\nu|t_j^i-t_{j^{\prime}}^i|}\quad  {\rm for } \quad \ell\in\{1,\dots,n\},
			\end{cases}
		\end{equation*}
		where $0<\tau\ll\nu$ small enough with $\nu=\min \left\{\zeta_1+\gamma_{\sigma}, \frac{\gamma_{\sigma}}{2}\right\}$ independent of $L\gg1$ large and $\xi>0$.
	\end{lemma}
	
	\begin{proof}
		The proof is the same as in \cite[Lemmas~5.5 and 5.6]{MR4104278}; thus, we omit the details.
	\end{proof}
	
	
	
	\color{black}
	
	\section{Gluing technique}\label{sec:gluingtechnique}
	In this section, we prove our main results. We keep the notation and assumptions in the previous sections. 
	The proof here is similar in spirit to the one in \cite[Theorem~1]{MR4104278}.
	Nevertheless, we include it here for the sake of completeness.
	
	\subsection{Infinite-dimensional Toda-system}
	We apply a fixed-point strategy in a weighted space of sequences.
	Before we start, we define some notation.
	\begin{definition}
		For any $\tau>0$, let us introduce the following weighted norm 
		\begin{equation*}
			|(\boldsymbol{a}_j,\boldsymbol{\lambda}_j)|_{\infty,\tau}=\sup_{j\in\mathbb N} e^{(2j+1) \tau}|(\boldsymbol{a}_j,\boldsymbol{\lambda}_j)|_{\infty}.
		\end{equation*}
		We also consider the associated Banach space given by
		\begin{equation*}
			\ell_\tau^\infty(\mathbb R^{(n+1)N})=\left\{(\boldsymbol{a}_j,\boldsymbol{\lambda}_j)\in  \ell^\infty(\mathbb R^{(n+1)N}) : |(\boldsymbol{a}_j,\boldsymbol{\lambda}_j)|_{\infty,\tau}<+\infty\right\}.
		\end{equation*}
		For any $(\boldsymbol{\bar{a}}_j,\boldsymbol{\bar{r}}_j)\in  \ell_\tau^\infty(\mathbb R^{(n+2)N})$, we define the interaction operator 
		\begin{equation*}
			\mathscr{T}_{(\boldsymbol{\bar{a}}_j,\boldsymbol{\bar{r}}_j)}: \ell_\tau^\infty(\mathbb R^{(n+1)N})\rightarrow  \ell_\tau^\infty(\mathbb R^{(n+1)N})
		\end{equation*}
		given by $\mathscr{T}_{(\boldsymbol{\bar{a}}_j,\boldsymbol{\bar{r}}_j)}=(\mathscr{T}_{(\boldsymbol{\bar{a}}_j)},\mathscr{T}_{(\boldsymbol{\bar{r}}_j)})$.
		Here 
		\begin{equation*}
			\mathscr{T}_{(\boldsymbol{\bar{a}}_j)}(\boldsymbol{\tilde a}_j)=\mathscr{T}_{(\boldsymbol{\bar{a}}_j)} \boldsymbol{\tilde a}_j^{\rm t}\quad {\rm and} \quad \mathscr{T}_{(\boldsymbol{\bar{r}}_j)}(\boldsymbol{\tilde r}_j)=\mathscr{T}_{(\boldsymbol{\bar{r}})_j} \boldsymbol{\tilde r}_j^{\rm t},
		\end{equation*}
		where $\mathscr{T}_{(\boldsymbol{\bar{a}}_j)}=(\mathscr{T}_{(\boldsymbol{\bar{a}}_j)}^1,\dots, \mathscr{T}_{(\boldsymbol{\bar{a}}_j)}^N)$ and $\mathscr{T}_{(\boldsymbol{\bar{r}})_j}=(\mathscr{T}_{(\boldsymbol{\bar{r}})_j}^1,\dots, \mathscr{T}_{(\boldsymbol{\bar{r}})_j}^N)$ with 
		\begin{equation}\label{iteractionoperator1}
			\mathscr{T}_{(\boldsymbol{\bar{a}}_j)}^i=\left(\begin{array}{ccccccc}
				-1 & 1+e^{-2 L_i} & -e^{-2 L_i} & 0 & \cdots & \cdots & 0 \\
				0 & -1 & 1+e^{-2 L_i} & -e^{-2 L_i} & 0 & \ddots & 0 \\
				0 & 0 & -1 & 1+e^{-2 L_i} & -e^{-2 L_i} & 0 & \vdots \\
				\cdots & \cdots & \cdots & \cdots & \cdots & \cdots & \cdots
			\end{array}\right)
		\end{equation}
		and 
		\begin{equation}\label{iteractionoperator2}
			\mathscr{T}_{(\boldsymbol{\bar{r}}_j)}^i=\left(\begin{array}{ccccccc}
				-1 & 2 & -1 & 0 & \cdots & \cdots & 0 \\
				0 & -1 & 2 & -1 & 0 & \ddots & 0 \\
				0 & 0 & -1 & 2 & -1 & 0 & \vdots \\
				\cdots & \cdots & \cdots & \cdots & \cdots & \cdots & \ldots
			\end{array}\right)
		\end{equation}
		being infinite-dimensional matrix for all $i\in\{1,\dots,N\}$.
	\end{definition}
	
	It is straightforward to see that these infinite-dimensional matrices are not invertible since they have a trivial kernel.
	However, they are indeed invertible in some suitably weighted norms defined above.
	In this direction, we have the following surjectiveness result for the interaction operator.
	\begin{lemma}
		Let $\sigma \in(1,+\infty)$, $n>2\sigma$, and $N\geqslant 2$.
		For any $\tau>0$, the interaction operator $\mathscr{T}_{(\boldsymbol{\bar{a}}_j,\boldsymbol{\bar{r}}_j)}: \ell_\tau^\infty(\mathbb R^{(n+1)N})\rightarrow  \ell_\tau^\infty(\mathbb R^{(n+1)N})$ has an inverse, denoted by $\mathscr{T}^{-1}_{(\boldsymbol{\bar{a}}_j,\boldsymbol{\bar{r}}_j)}: \ell_\tau^\infty(\mathbb R^{(n+1)N})\rightarrow  \ell_\tau^\infty(\mathbb R^{(n+1)N})$. 
		Moreover, one has 
		\begin{equation}\label{estimateinverse}
			\sup_{|(\boldsymbol{{a}}_j,\boldsymbol{{r}}_j)|_{\infty,\tau}=1}\|\mathscr{T}^{-1}_{(\boldsymbol{\bar{a}}_j,\boldsymbol{\bar{r}}_j)}(\boldsymbol{{a}}_j,\boldsymbol{{r}}_j)\|\lesssim e^{-2\tau}.
		\end{equation}
	\end{lemma}
	
	\begin{proof}
		The proof is given by directly constructing the inverse operator.
		First, we observe that $\mathscr{T}^{-1}_{(\boldsymbol{\bar{r}})_j}: \ell_\tau^\infty(\mathbb R^N)\rightarrow  \ell_\tau^\infty(\mathbb R^N)$ can be found in \cite[Lemma 7.3]{MR2522830}.
		
		We are left to provide the inverse for $\mathscr{T}_{(\boldsymbol{\bar{a}}_j)}: \ell_\tau^\infty(\mathbb R^{nN})\rightarrow  \ell_\tau^\infty(\mathbb R^{nN})$.
		Indeed, for any $\boldsymbol{b}\in \ell_\tau^\infty(\mathbb R^{nN})$, we have to solve $\mathscr{T}_{(\boldsymbol{\bar{a}}_j)}(\boldsymbol{\tilde{a}}_j)=\boldsymbol{b}_j$. 
		This is accomplished by defining $\mathscr{T}^{-1}_{(\boldsymbol{\bar{a}})_j}: \ell_\tau^\infty(\mathbb R^{nN})\rightarrow  \ell_\tau^\infty(\mathbb R^{nN})$ as
		\begin{equation*}
			\boldsymbol{\tilde{a}}_j=\sum_{k=j+1}^{\infty}\left(\sum_{s=0}^{k-j-1} e^{-2 L_i s}\right)\boldsymbol{b}_j:=\mathscr{T}^{-1}_{(\boldsymbol{\bar{a}})_j}.
		\end{equation*}
		Whence, by performing the same routine computations, one can quickly check that $\boldsymbol{\tilde{a}}_j\in \ell_\tau^\infty(\mathbb R^{nN})$ satisfies the required conditions and that the operator $\mathscr{T}^{-1}_{(\boldsymbol{\bar{a}})_j}$ is a complete inverse of $\mathscr{T}_{(\boldsymbol{\bar{a}})_j}$.
		
		In addition, one has
		\begin{equation*}
			|\boldsymbol{\tilde{a}}_j|_{\infty} \lesssim  |\boldsymbol{\tilde{b}}_j|_{\infty,\tau} \sum_{k=j+1}^{\infty}\left(\sum_{s=0}^{k-j-1} e^{-2 L_i s}\right) e^{-(2 k+1) \tau} \lesssim e^{-(2 j+3) \tau}|\boldsymbol{\tilde{b}}_j|_{\infty,\tau},
		\end{equation*}
		which implies the estimate \eqref{estimateinverse}. 
		The lemma is proved.
	\end{proof}
	
	\begin{lemma}\label{lm:invertibility}
		Let $\sigma \in(1,+\infty)$, $n>2\sigma$, and $N\geqslant 2$.
		Assume that $(\boldsymbol{R},\boldsymbol{\hat{a}}_0,\boldsymbol{q})\in {\rm Bal}_{\sigma}(\Sigma)$ is a balanced configuration.
		Then, for $L\gg1$ sufficiently large, there exists $0<\tau<\min \{\xi, \nu\}$ and an admissible perturbation sequence $(\boldsymbol{a}_j,\boldsymbol{\lambda}_j)\in{\rm Adm}_{\sigma}(\Sigma)\subset \ell^\infty_{\tau}(\mathbb R^{(n+1)N})$ such that $\beta_{j, \ell}^i(\boldsymbol{a}_j,\boldsymbol{\lambda}_j)=0$ for $(i,j,\ell)\in\mathcal{I}_\infty$, that is, $(\boldsymbol{a}_j,\boldsymbol{\lambda}_j)\in{\rm Adm}_{\sigma}(\Sigma)$ solves the infinite-dimensional system \eqref{todasystem}.    
	\end{lemma}
	
	\begin{proof}
		Indeed, for any $i\in \{1,\dots, N\}$, let us define the operator $\mathscr{G}^i:\ell^\infty_{\tau}(\mathbb R^{(n+1)N})\rightarrow \ell^\infty_{\tau}(\mathbb R^{(n+1)N})$ given by $\mathscr{G}^i=(\mathscr{G}^i_0,\dots,\mathscr{G}^i_n)$, where $\mathscr{G}^i_\ell:\ell^\infty_{\tau}(\mathbb R^{N})\rightarrow \ell^\infty_{\tau}(\mathbb R^{N})$ for each $\ell\in\{0,\dots,n\}$.
		More precisely, we have
		\begin{equation*}
			\mathscr{G}_0^0(\boldsymbol{a}_j,\boldsymbol{\lambda}_j)=\frac{1}{F(L_i)}[\beta_{j, 0}^i(\boldsymbol{a}_j,\boldsymbol{\lambda}_j)-{\beta}_{j, 0}^i(\boldsymbol{0},\boldsymbol{1})]\mathrm{e}_i^{\rm t}-\mathscr{T}_{(\boldsymbol{\bar{r}}_j)}(\boldsymbol{\tilde r}_j)
		\end{equation*}
		and 
		\begin{equation*}
			\mathscr{G}_0^\ell(\boldsymbol{a}_j,\boldsymbol{\lambda}_j)=\frac{e^{\gamma_\sigma L_i}}{\lambda_j^i}[\beta_{j, 0}^i(\boldsymbol{a}_j,\boldsymbol{\lambda}_j)-{\beta}_{j, 0}^i(\boldsymbol{0},\boldsymbol{1})]\mathrm{e}_i^{\rm t}-\mathscr{T}_{(\boldsymbol{\bar{a}}_j)}(\boldsymbol{\tilde a}_j) \quad {\rm for} \quad \ell\in\{1,\dots,n\},
		\end{equation*}
		where $\mathrm{e}_i\in \ell^\infty(\mathbb R^{(n+1)N})$ is the $i$-th vector in its standard Schauder basis, which we denote by $\{\mathrm{e}_i\}_{i\in\mathbb N}\subset \ell^\infty(\mathbb R^{(n+1)N})$.
		
		One can easily see that $\beta_{i,j}^\ell(\boldsymbol{a}_j,\boldsymbol{\lambda}_j)=0$ for $j \geqslant 1$ if
		\begin{align}\label{contractiongluing1}
			(\tilde{a}^i_j)^{\rm t}=-(\mathscr{T}^i_{(\boldsymbol{\bar{a}}_j)})^{-1}\left(-\frac{{e^{\gamma_\sigma L_i}}}{\lambda^i_j}{\beta}_{j, 0}^i(\boldsymbol{0},\boldsymbol{1})\mathrm{e}_i^{\rm t}+\mathscr{G}^i_\ell(\boldsymbol{a}_j,\boldsymbol{\lambda}_j)\right)
		\end{align}
		and 
		\begin{align}\label{contractiongluing2}
			(r^i_j)^{\rm t}=-(\mathscr{T}^i_{(\boldsymbol{r}_j)})^{-1}\left(-\frac{1}{F(L_i)}{\beta}_{j, 0}^i(\boldsymbol{0},\boldsymbol{1})\mathrm{e}_i^{\rm t}+\mathscr{G}^i_0(\boldsymbol{a}_j,\boldsymbol{\lambda}_j)\right).
		\end{align}            
		Next, we show that the terms on the right-hand sides of \eqref{contractiongluing1} and  \eqref{contractiongluing2} are contractions in an appropriate sense. 
		First, by Lemma~\ref{lm:estimatesbetaspecial}, one has
		\begin{align*}
			\left|{\beta}_{j, \ell}^i(\boldsymbol{0},\boldsymbol{1})\right| \lesssim\left\{\begin{array}{ll}
				e^{-\gamma_\sigma L_1(1+\xi)} e^{-\nu t_j^i}, & \text { if } \ell=0, \\
				\lambda_j^i e^{-\gamma_\sigma L_1(1+\xi)} e^{-\nu t_j^i}, & \text { if } \ell \geqslant 1,
			\end{array} \quad \text { for } j \geqslant 1.\right.
		\end{align*}
		Also, let us denote by the projection on the  $j$-th component it follows 
		\begin{equation}\label{contractiondiscrte1}
			(\widehat{\mathscr{G}}_{\ell,j}+\widetilde{\mathscr{G}}_{\ell,j})(\boldsymbol{a}_j,\boldsymbol{\lambda}_j):=\Pi_j\left(\frac{e^{\gamma_\sigma L_i}}{\lambda_j^i}[\beta_{j, 0}^i(\boldsymbol{a}_j,\boldsymbol{\lambda}_j)-{\beta}_{j, 0}^i(\boldsymbol{0},\boldsymbol{1})]\mathrm{e}_i^{\rm t}-\mathscr{T}_{(\boldsymbol{\bar{a}}_j)}(\boldsymbol{\tilde a}_j)\right),
		\end{equation}
		where
		\begin{equation*}
			\widehat{\mathscr{G}}_{\ell,j}(\boldsymbol{a}_j,\boldsymbol{\lambda}_j)=\int_0^1\left[\frac{e^{\gamma_\sigma} L_i}{\lambda_j^i} \partial_t \beta_{j, \ell}^i(t(\tilde{a}^i_j,r^i_j)^{\rm t})-\overline{\mathscr{A}}^i\right]\left((\bar{a}^i_j,r^i_j)^{\rm t}\right) \ud t
		\end{equation*}
		and
		\begin{equation*}
			\widetilde{\mathscr{G}}_{\ell,j}(\boldsymbol{a}_j,\boldsymbol{\lambda}_j)=\overline{\mathscr{A}}^i-\mathscr{T}^i_{(\boldsymbol{\bar{a}}_j)}({\tilde a}^i_j)
		\end{equation*}
		with
		\begin{equation*}
			\overline{\mathscr{A}}^i_j\left((\tilde{a}^i_j,r^i_j)^{\rm t}\right)=\sum_{j^{\prime}\in\mathbb N} \frac{e^{\gamma_\sigma L}}{\lambda_j^i} \partial_{\bar{a}_{j^{\prime}}^i} \beta_{j, \ell}^{i_*}(\boldsymbol{a}_j,\boldsymbol{\lambda}_j)\cdot\left[\bar{a}_{j^{\prime}}^i\right],
		\end{equation*}
		where $\beta_{j, \ell}^{i_*}(\boldsymbol{a}_j,\boldsymbol{\lambda}_j)\in\mathbb R$ is defined as \eqref{specialprojectioncoefficients} and $\bar{a}_j^i\in \mathbb R^{nN}$ corresponds to the translation perturbation of the $j$-th bubble in the Delaunay solution from Lemma~\ref{lm:derivateprojeestimates1}. 
		Furthermore, by definition, one has
		\begin{equation*}
			\mathscr{T}_{(\boldsymbol{\bar{a}}_j)}( (\bar{a}^i_j)^{\rm t})=\mathscr{T}_{(\boldsymbol{\bar{a}}_j)}((\tilde{a}^i_j)^{\rm t}).
		\end{equation*}
		Now we have to estimate the terms on the left-hand side of \eqref{contractiondiscrte1}.
		Indeed, we begin by estimating the first term.
		As a consequence of Lemma~\ref{lm:derivateprojeestimates1} for $\ell\in\{1, \ldots, n\}$, one finds
		\begin{align*}
			\left|\widehat{\mathscr{G}}_{\ell,j}(\boldsymbol{a}_j,\boldsymbol{\lambda}_j)\right| & \lesssim \sum_{j^{\prime}\in \mathbb N} \frac{e^{\gamma_\sigma L}}{\lambda_j^i}|\partial_{\bar{a}_{j^{\prime}}^i}(\beta_{j, \ell}^i-\beta_{j, \ell}^{i_*})(\boldsymbol{a}_j,\boldsymbol{\lambda}_j)|\left|\bar{a}_{j^{\prime}}^i\right|+\mathcal{O}\left(e^{-\gamma_\sigma L_i} \xi e^{-\min \{\nu, \tau\} t_j^i}\right) \\
			& \lesssim e^{-\gamma_\sigma \xi} \sum_{j^{\prime}\in\mathbb N} e^{-\nu t_{j^{\prime}}^{i^{\prime}}} e^{-\nu|t_j^i-t_{j^{\prime}}^{i^{\prime}}|}|\bar{a}_{j^{\prime}}^i|+\mathcal{O}\left(e^{-\frac{(n-2 \gamma) L_i}{2} \xi} e^{-\min \{\nu, \tau\} t_j^i}\right) \\
			& \lesssim \left(e^{-\gamma_\sigma L \xi}e^{-\min \{\nu, \tau\} t_j^i}\right).
		\end{align*}
		In addition, we apply Lemma~\ref{lm:derivateprojeestimates1} to estimate the second term; this gives us
		\begin{equation}
			\left|\widetilde{\mathscr{G}}_{\ell,j}(\boldsymbol{a}_j,\boldsymbol{\lambda}_j)\right| \lesssim e^{-\gamma_\sigma L_i \xi} \left[e^{-\nu L_i}\left(|\tilde{a}_{j-1}^i|+|\tilde{a}_{j+1}^i|\right)+\sum_{j^{\prime} \neq j \pm 1} e^{-\nu|t_{j^{\prime}}^i-t_j^i|}|\tilde{a}_{j^{\prime}}^i|\right].
		\end{equation}	
		Therefore, by combining these two estimates, it follows that for $0<\tau<\nu\ll1$, one has
		\begin{equation*}
			\left\|\mathscr{G}_\ell^i(\boldsymbol{a}_j,\boldsymbol{\lambda}_j)\right\|_{\ell^\infty_{\tau_i}(\mathbb R^{(n+1)N})} \lesssim e^{\tau L}e^{-\gamma_\sigma L_i \xi} \|(\tilde{a}^i_j)^{\rm t}\|_{\ell^\infty_{\tau_i}(\mathbb R^{nN})}+\mathcal{O}(e^{-\gamma_\sigma L_i \xi}) \quad {\rm for } \quad \ell\in\{1, \ldots, n \}
		\end{equation*}
		and
		\begin{equation*}
			\left\|\mathscr{G}_0^i(\boldsymbol{a}_j,\boldsymbol{\lambda}_j)\right\|_{\ell_{\tau_i}(\mathbb R^{(n+1)N})} \lesssim e^{\tau L}e^{-\gamma_\sigma L_i \xi}\| (r^i_j)^{\rm t}\|_{\ell_{\tau_i}(\mathbb R^{N})}+\mathcal{O}(e^{-\gamma_\sigma L_i \xi}),
		\end{equation*}
		where $\tau_i=\frac{\tau L_i}{2}$.	
		Next, up to some error, Eq. \eqref{contractiongluing1} and \eqref{contractiongluing2} can be reformulated as 
		\begin{equation*}
			(\tilde{a}^i_j)^{\rm t}=(\mathscr{T}^i_{(\boldsymbol{\bar{a}})_j})^{-1}\left[e^{\tau L} e^{-\gamma_\sigma L_i\xi} \|(\tilde{a}^i_j)^{\rm t}\|_{\ell^\infty_{\tau_i}(\mathbb R^{nN})}+\mathcal{O}(e^{-\gamma_\sigma L_i\xi})\right]=: \boldsymbol{\mathscr{G}}_{(\boldsymbol{a}_j)}((\tilde{a}^i_j)^{\rm t})
		\end{equation*}
		and
		\begin{equation*}
			(r^i_j)^{\rm t}=(\mathscr{T}^i_{(\boldsymbol{r})_j})^{-1}\left[e^{\tau L} e^{-\gamma_\sigma L_i\xi}\|(r^i_j)^{\rm t}\|_{\ell^\infty_{\tau_i}(\mathbb R^{N})}+\mathcal{O}(e^{-\gamma_\sigma L_i\xi} )\right]=: \boldsymbol{\mathscr{G}}_{(\boldsymbol{r}_j)}((r^i_j)^{\rm t}),
		\end{equation*}
		where the right-hand sides of the above equations are estimated in $\ell^\infty_{\tau_i}(\mathbb R^{(n+1)N})$ norm.
		
		At last, for $0<\tau<\xi\ll1$, let us consider the set
		\begin{equation*}
			\mathscr{B}^*_{L}:=\left\{(\tilde{a}^i_j,r^i_j)^{\rm t}\in \ell^\infty_{\tau_i}(\mathbb R^{(n+1)N}) : \|(\tilde{a}^i_j,r^i_j)\|_{\infty,\frac{\tau L_i}{2}} \lesssim e^{-\tau L}\right\}.
		\end{equation*}
		Notice that $\mathscr{G}_{(\boldsymbol{\bar{a}}_j,\boldsymbol{r}_j)}:\mathscr{B}^*_{L}\rightarrow \ell^\infty_{\tau_i}(\mathbb R^{(n+1)N})$ given by $\mathscr{G}_{(\boldsymbol{\bar{a}}_j,\boldsymbol{r}_j)}=(\mathscr{G}_{(\boldsymbol{\bar{a}}_j)}, \mathscr{G}_{(\boldsymbol{r})_j})$ maps $\mathscr{B}^*_{L}$ into itself, and it is a contraction. 
		Therefore, one can invoke Banach's contraction principle to find a fixed point in the set $\mathscr{B}^*_{L}$, which solves \eqref{todasystem}.
		The proof is then finished.
	\end{proof}
	
	Next, we have invertibility lemma based on the balancing conditions from Definition~\ref{def:balancedparameters}.
	
	\begin{lemma}\label{lm:algebraiclemma}
		Let $\sigma \in(1,+\infty)$, $n>2\sigma$, and $N\geqslant 2$.
		Assume that $(\boldsymbol{q}^b,\mathbf{a}_0^b,\boldsymbol{R}^b)\in{\rm Bal}_{\sigma}(\Sigma)$ is a balanced configuration.     
		Let us consider the operator $\mathcal{F}:\mathbb{R}^{2N}\rightarrow \mathbb{R}^{N}$ is given by 
		\begin{equation*}
			\mathcal{F}(\boldsymbol{q},\boldsymbol{R})=A_2 \sum_{i^{\prime} \neq i}|x_{i^{\prime}}-x_i|^{-(n-2 \sigma)}(R^i R^{i^{\prime}})^{\gamma_\sigma} q_{i^{\prime}}-q_i.
		\end{equation*}  
		Then, the linearized operator around $(\boldsymbol{q}^b,\boldsymbol{R}^b)$, denoted by $\ud\mathcal{F}_{(\boldsymbol{q}^b,\boldsymbol{R}^b)}:\mathbb{R}^{2N}\rightarrow \mathbb{R}^{N}$, is invertible.
	\end{lemma}

	\begin{proof}
		Notice that the linearized operator $\left.\ud{\mathcal{F}}({\boldsymbol{q},\boldsymbol{R}})\right|_{(\boldsymbol{q}^b,\boldsymbol{R}^b)}: \mathbb{R}^{2N} \rightarrow \mathbb{R}^N$ has the following expression 
		\begin{equation*}
			\ud{\mathcal{F}}_{({\boldsymbol{q},\boldsymbol{R}})}=(\boldsymbol{q}_{i^\prime},\boldsymbol{R}_{i^\prime})=:(\ud\hat{\mathcal{F}}_{\boldsymbol{q}},\ud\hat{\mathcal{F}}_{\boldsymbol{R}}),
		\end{equation*}
		where $\boldsymbol{q}_{i^\prime}\in\mathbb R_+^N$ and $\boldsymbol{R}_{i^\prime}\in\mathbb R_+^N$ are defined, respectively, as 
		\begin{equation*}
			\boldsymbol{q}_{i^\prime}=({q}_{ii^\prime}) \quad {\rm and} \quad \boldsymbol{R}_{i^\prime}=({R}_{ii^\prime})
		\end{equation*}
		with 
		\begin{equation*}
			{q}_{ii^\prime}=
			\begin{cases}
				-1, & {\rm if} \quad i=i^\prime\\
				A_2|x_i-x^{i^\prime}|^{-(n-2 \sigma)}(R^{i, b} R^{i^{\prime}, b})^{\gamma_\sigma}, & {\rm if} \quad i\neq i^\prime\\
			\end{cases}
		\end{equation*}
		and 
		\begin{equation*}
			{R}_{ii^\prime}=
			\begin{cases}
				\gamma_\sigma({R^{i, b}})^{-1} \sum_{i^{\prime} \neq i} A_2|x_{i^{\prime}}-x_i|^{-(n-2\sigma)}(R^{i, b} R^{i^{\prime}, b})^{\gamma_\sigma} q_i^b, & {\rm if} \quad i=i^\prime\\
				\gamma_\sigma({R^{i, b}})^{-1} A_2|x_{i^\prime}-x_i|^{-(n-2 \sigma)}(R^{i, b} R^{i^\prime, b})^{\gamma_\sigma} q_{i^\prime}^b & {\rm if} \quad i\neq i^\prime.       \end{cases}
		\end{equation*}
		Next, from the balancing condition \eqref{balancing1}, it follows $ \mathcal{F}(\boldsymbol{q}^b,\boldsymbol{R}^b)=0$.
		Also, one can see that $\ud\hat{\mathcal{F}}_{\boldsymbol{q}}$ is symmetric and has only a one-dimensional kernel.
		More precisely, we have $\operatorname{Ker}(\ud\hat{\mathcal{F}}_{\boldsymbol{q}})=\operatorname{span}\{\boldsymbol{q}^b\}$.
		
		Finally, the balancing condition \eqref{balancing1} also implies $ \ud\hat{\mathcal{F}}_{\boldsymbol{q}}(\boldsymbol{R})=\gamma_\sigma\boldsymbol{q}$.
		From this, it is easy to conclude that the operator $\ud\mathcal{F}_{(\boldsymbol{q}^b,\boldsymbol{R}^b)}$ is surjective.    
		
	\end{proof}

	\subsection{Proof of the main result}     
	Now we can provide proof for our main result in this paper.
	\begin{proof}[Proof of Theorem~\ref{maintheorem}]
		By Lemma~\ref{lm:invertibility}, we are reduced to find $(\boldsymbol{R},\hat{\boldsymbol{a}}_0,\boldsymbol{q})\in \mathbb R^{(n+2)N}$ such that 
		$\beta_{0, \ell}^i(\boldsymbol{a}_j,\boldsymbol{\lambda}_j)=0$ for all $j\in \mathbb N$, where $(\boldsymbol{a}_j,\boldsymbol{\lambda}_j)=\Upsilon_{\rm per}(\boldsymbol{R},\hat{\boldsymbol{a}}_0,\boldsymbol{q})$. 
		
		The rest of the proof will be divided into two main parts: the zero-mode and the linear-mode case.
		First, if $j=0$, using Lemma~\ref{lm:estimatesbetageneral} (i), one has that equation $\beta_{0,0}^i(\boldsymbol{a}_j,\boldsymbol{\lambda}_j)=0$ is reduced to
		\begin{align*}
			-c_{n, \sigma} q_i\left[A_2 \sum_{i^{\prime} \neq i}|x_{i^{\prime}}-x_i|^{-(n-2 \sigma)}(R_0^i R_0^{i^{\prime}})^{\gamma_\sigma} q_{i^{\prime}}-\left(\frac{R_1^i}{R_0^i}\right)^{\gamma_\sigma} q_i\right] e^{-\gamma_\sigma L}(1+\mathrm{o}(1))+\mathcal{O}(e^{-\gamma_\sigma L(1+\xi)})=0.
		\end{align*}
		Furthermore, recall that since $R_0^i=R^i(1+r_0^i)$, one can use that $r_0^i\in\mathbb R_+$ satisfies $|r_0^i|\lesssim e^{-2\tau}$, to reformulate the equation above as
		\begin{equation}\label{gluing1}
			\mathscr{F}_1(\boldsymbol{R},\hat{\boldsymbol{a}}_0,\boldsymbol{q})=\mathrm{o} (1).
		\end{equation}
		Here $\mathscr{F}_1:\mathbb R^{2N}\rightarrow \mathbb R^{(n+2)N}$ is given by 
		\begin{equation}\label{gluingoperator1}
			\mathscr{F}_1(\boldsymbol{R},\hat{\boldsymbol{a}}_0,\boldsymbol{q}):=A_2 \sum_{i^{\prime} \neq i}|x_{i^{\prime}}-x_i|^{-(n-2 \sigma)}(R^i R^{i^{\prime}})^{\gamma_\sigma} q_{i^{\prime}}-q_i.
		\end{equation}
		Second, if $\ell\in\{1, \dots, n\}$, using Lemma~\ref{lm:estimatesbetageneral} (ii), it is not hard to check that the system $\beta_{0,\ell}^i(\boldsymbol{a}_j,\boldsymbol{\lambda}_j)=0$ are reduced to
		\begin{align*}
			c_{n, \sigma} \lambda_0^i\left[A_3 \sum_{i^{\prime} \neq i} \frac{(x_{i^{\prime}}-x_i)_\ell}{|x_{i^{\prime}}-p_i|^{n-2 \sigma+2}}(R_0^i R_0^{i^{\prime}})^{\gamma_\sigma} q_{i^{\prime}}+A_0\left(\frac{R_1^i}{R_0^i}\right)^{\gamma_\sigma} \frac{a_0^i-a_1^i}{\left(\lambda_0^i\right)^2} q_i\right] q_i e^{-\gamma_\sigma L}+\mathcal{O}(\lambda_0^i e^{-\gamma_\sigma L(1+\xi)})=0.
		\end{align*}
		In addition, since 
		\begin{equation*}
			a_j^i=(\lambda_j^i)^2 \bar{a}_j^i \quad {\rm and} \quad \bar{a}_j^i=\hat{a}_0^i+\tilde{a}_j^i,
		\end{equation*}
		one can use that $\tilde{a}_j^i\in\mathbb R^{nN}$ also has the decay $|\tilde{a}_j^i|\lesssim e^{-2\tau}$, the above equation can be rewritten as
		\begin{equation}\label{gluing2}
			\mathscr{F}_2(\boldsymbol{R},\hat{\boldsymbol{a}}_0,\boldsymbol{q})=o(1),
		\end{equation}
		where $\mathscr{F}_2:\mathbb R^{nN}\rightarrow \mathbb R^{(n+2)N}$ is given by
		\begin{equation}\label{gluingoperator2}
			\mathscr{F}_2(\boldsymbol{R},\hat{\boldsymbol{a}}_0,\boldsymbol{q}):=A_3 \sum_{i^{\prime} \neq i} \frac{(x_{i^{\prime}}-x_i)_\ell}{|x_{i^{\prime}}-x_i|^{n-2 \sigma+2}}(R^i R^{i^{\prime}})^{\gamma_\sigma} q_{i^{\prime}}+A_0 \hat{a}_0^i q_i.
		\end{equation}
		
		To conclude we need to choose suitable $(\boldsymbol{R},\hat{\boldsymbol{a}}_0,\boldsymbol{q})\in \mathbb R^{(n+2)N}$ such that equations \eqref{gluing1} and \eqref{gluing2} are solvable. 
		Notice that the solvability of \eqref{gluing1} and \eqref{gluing2} depends on the following invertibility property of the linearized operator of
		$\mathscr{F}:\mathbb R^{(n+2)N}\rightarrow \mathbb R^{(n+2)N}$ given by $\mathscr{F}=(\mathscr{F}_1,\mathscr{F}_2)$ around $(\boldsymbol{R},\hat{\boldsymbol{a}}_0,\boldsymbol{q})$.
		Moreover, from Lemma~\ref{lm:algebraiclemma}, this accomplished since $(\boldsymbol{R},\hat{\boldsymbol{a}}_0,\boldsymbol{q})\in{\rm Bal}_{\sigma}(\Sigma)$, that is, it satisfies \eqref{balancing1} and \eqref{balancing2}.
		More precisely, the balancing condition \eqref{balancing1}, one can easily perturb $(\boldsymbol{R}^b,\boldsymbol{q}^b)$           
		to find $(\boldsymbol{R},\boldsymbol{q})$ solving \eqref{gluing1}. 
		Next, using the second balancing condition \eqref{gluing2}, one can find $\boldsymbol{\hat{a}}_0\in \mathbb R^{nN}$ around $\boldsymbol{\hat{a}}_0^{b}\in \mathbb R^{nN}$ which solves \eqref{gluing2}.
		
		In conclusion, we use the maximum principle in Lemma~\ref{lm:maximumprinciple} to show that $u>0$ to conclude the proof of the main theorem.   
	\end{proof}
	
	\appendix
	
	\section{Estimates on the bubble-towers interactions}\label{sec:bubbetowerinteractionestimates}
	
	In this appendix, we quote some important integrals in our proof. 
	The following expressions may be found in \cite[Appendix 7]{MR4104278} for $\sigma\in\mathbb R_+$.
	Let $\lambda_1,\lambda_2\lambda_3>0$ and $x_1,x_2\in\mathbb R^n$ with $x\neq0$, we define
	\begin{equation*}
		U_1:=U_{0,\lambda_1}, \quad U_2:=U_{0,\lambda_2}, \quad {\rm and} \quad U_3:=U_{x,\lambda_3},
	\end{equation*}
	where $w_{x_0,\lambda}$ is given by \eqref{bubbles}.
	We also recall 
	\begin{equation*}
		\gamma_{\sigma}:=\frac{n-2\sigma}{2} \quad {\rm and} \quad \gamma_{\sigma}^{\prime}:=\frac{n+2\sigma}{2}
	\end{equation*}
	to be the Fowler rescaling exponent and its Lebesgue conjugate, respectively.
	
	In what follows, we use the constants below
	\begin{equation}\label{A1}
		A_1=\frac{(n+2 \sigma)(n-2 \sigma)}{n} \int_{\mathbb{R}^n}\left({|x|^{2\gamma_{\sigma}}\left(1+|x|^2\right)^{\gamma_{\sigma}^{\prime}}}+1\right)^{-1} \ud x>0,
	\end{equation}
	\begin{equation}\label{A2}
		A_2=\frac{n+2 \sigma}{2} \int_{\mathbb{R}^n} {(|x|^2-1)}\left(1+|x|^2\right)^{-\gamma_{\sigma}-1}\ud x>0,
	\end{equation}
	and
	\begin{equation}\label{A3}
		A_3=-\frac{(n-2 \sigma)^2}{n} \int_{\mathbb{R}^n} {|x|^2}\left(1+|x|^2\right)^{-\gamma_{\sigma}-1} \ud x<0 .
	\end{equation}
	
	\begin{lemma}
		For any $\lambda_1,\lambda_2>0$.
		It holds
		\begin{equation*}
			\int_{\mathbb{R}^n} f_{\sigma}^{\prime}(U_1) U_2 {\partial_{\lambda_1} U_1} \ud x=\frac{1}{\lambda_1} \Psi\left(\left|\log \frac{\lambda_2}{\lambda_1}\right|\right) \frac{\log \frac{\lambda_2}{\lambda_1}}{\left|\log \frac{\lambda_2}{\lambda_1}\right|},
		\end{equation*}
		where
		\begin{equation*}
			\Psi(\ell)=e^{-\gamma_{\sigma} \ell}(1+\mathrm{o}(1)) \quad {\rm as} \quad \ell \rightarrow +\infty
		\end{equation*}
		with
		\begin{equation}\label{auxiliaryfunctioninteraction}
			\Psi(\ell):=\int_{\mathbb{R}} f_{\sigma}^{\prime}(v_{\rm sph}(t)) v_{\rm sph}(t+\ell) v^{\prime}(t) \ud t.
		\end{equation}
	\end{lemma}
	
	\begin{proof}
		See \cite[Lemma~7.1]{MR4104278}.
	\end{proof}
	
	\begin{lemma}
		If $\lambda_3=\mathcal{O}(\lambda_1)$, then the following estimates hold
		\begin{equation*}
			\int_{\mathbb{R}^n} f_{\sigma}^{\prime}(U_1) U_3 \partial_{\lambda_1} U_1 \ud x=A_2 |x_2|^{2 \sigma-n}\frac{\left(\lambda_1 \lambda_3\right)^{\gamma_{\sigma}}}{\lambda_1}\left(1+\mathcal{O}\left(\lambda_1\right)^2\right)
		\end{equation*}
		and
		\begin{equation*}
			\int_{\mathbb{R}^n} f_{\sigma}^{\prime}(U_1) U_3 \partial_{x_\ell} U_1 \ud x=A_3 {x_\ell}{|x|^{2 \sigma-n-2}}\left(\lambda_1 \lambda_3\right)^{\gamma_{\sigma}}\left(1+\mathcal{O}\left(\lambda_1^2\right)\right) \quad {\rm for} \quad \ell\in \{0,\dots,n\}.
		\end{equation*}
	\end{lemma}
	
	\begin{proof}
		See \cite[Lemma~7.2]{MR4104278}.
	\end{proof}
	
	\begin{lemma}
		Let $\lambda_1,\lambda_2>0$ and $a\in \mathbb R^n$.
		If $|a| \leqslant \max \left\{\lambda_1^2, \lambda_2^2\right\} \ll 1$ and $\min \left\{\frac{\lambda_1}{\lambda_2}, \frac{\lambda_2}{\lambda_1}\right\} \ll 1$, then the following estimates holds
		\begin{align}
			& \int_{\mathbb{R}^n} (\partial_a U_{a,\lambda_1}^{\gamma_{\sigma^{\prime}}}
			)U_{0,\lambda_2}^{\gamma_{\sigma}}
			\ud x=-A_0 c_{\lambda_1,\lambda_2}C_{\lambda_1,\lambda_2}+c_{\lambda_1,\lambda_2}\mathcal{O}\left(C_{\lambda_1,\lambda_2}^2+c_{\lambda_1,\lambda_2}^2C_{\lambda_1,\lambda_2}^2\right),
		\end{align}
		where
		\begin{equation*}
			c_{\lambda_1,\lambda_2}=\min \left\{\left(\frac{\lambda_1}{\lambda_2}\right)^{\gamma_{\sigma}},\left(\frac{\lambda_2}{\lambda_1}\right)^{\gamma_{\sigma}}\right\} \quad  {\rm and} \quad C_{\lambda_1,\lambda_2}=\frac{a}{\max \left\{\lambda_1^2, \lambda_2^2\right\}}.
		\end{equation*}
	\end{lemma}
	
	\begin{proof}
		See \cite[Lemma~7.3]{MR4104278}.
	\end{proof}
	
	\section{Nondegeneracy of the bubble solution}\label{sec:complentaryproofs}
	In this section, we add the proof of the nondegeneracy of the spherical solution.
	
	\begin{proof}[Proof of Lemma~\ref{lm:nondegeneracy}]
		Let us start with $\phi\in H^{\sigma}(\mathbb R^n)$.
		Using the statement in \cite[Lemma 5.1]{MR3899029}, it suffices to know that $\phi \in L^{\infty}\left(\mathbb{R}^n\right)$. 
		We will divide the proof of this fact into three cases, which we describe as follows 
		
		\noindent{\bf Case 1:} $n>6\sigma$.
		
		Indeed, notice that, from \eqref{linearizedequation} since $f^{\prime}_\sigma(u_{\rm sph})\in L^{\infty}(\mathbb R^n)$, one can find a large constant $C\gg1$ depending only on $n, \sigma$ such that
		\begin{equation}\label{nondegeneracy1}
			|\phi(x)| \leqslant \int_{\mathbb{R}^n} {C}{|x-y|^{2\sigma-n}}\left(\frac{|\phi(y)|}{1+|y|^{4 \sigma}}+\frac{1}{1+|y|^{n-2\sigma}}\right) \mathrm{d}y \quad \text { for } \quad x \in \mathbb{R}^n.
		\end{equation}
		Also, by partitioning the Euclidean space as $\mathbb R^n=B_d(0)\cup B_d(x)\cup (B_d(0)\cup B_d(x))^c$ with $d:={|x|}/{2} \geqslant 1$, and integrating on each subpart, we obtain
		\begin{align}\label{nondegeneracy3}
			\int_{\mathbb{R}^n} \frac{|x-y|^{2\sigma-n}}{1+|y|^{n-2\sigma}}{\mathrm{d} y} & \lesssim \frac{1}{|x|^{n-4\sigma}} \quad {\rm for \ all} \quad x \in \mathbb{R}^n.
		\end{align}
		Furthermore, by substituting the last inequality into \eqref{nondegeneracy1}, one has
		\begin{align}\label{nondegeneracy4}
			|\phi(x)| \leqslant C\left[\int_{\mathbb{R}^n} \frac{|x-y|^{2\sigma-n}}{1+|y|^{4\sigma}} {|\phi(x)|}\ud y+\frac{1}{1+|x|^{n-4\sigma}}\right] \quad \text { for } \quad x \in \mathbb{R}^n.
		\end{align}	
		Next, since $n>6\sigma$, one has that $[p_0,p_*)\neq\varnothing$, where $p_0=\frac{2 n}{n-2 \sigma}$ and $p_*=\frac{n}{2 \sigma}$, which allows us to use the Hardy--Littlewood--Sobolev inequality to get
		\begin{align}\label{nondegeneracy2}
			\nonumber
			\|\phi\|_{L^{p_1}(\mathbb{R}^n)} & \lesssim\left\|\frac{|\phi(x)|}{1+|x|^{4 \sigma}} \ast {|x|^{2 \sigma-n}}\right\|_{L^{p_1}(\mathbb{R}^n)} \\\nonumber
			& \lesssim\left\|\frac{|\phi(x)|}{1+|x|^{4 \sigma}}\right\|_{L^{q_1}(\mathbb{R}^n)}\\
			&\lesssim\|\phi\|_{L^{p_0}(\mathbb{R}^n)}\left\|\frac{1}{1+|x|^{4\sigma}}\right\|_{L^{q_0}(\mathbb{R}^n)},
		\end{align}            
		for any $p \in[p_0, p_*)$ and $p_2=\frac{n p_0}{n-2\sigma p_0}$. 
		
		In what follows, we are based on the estimate \eqref{nondegeneracy2} to run the bootstrap argument below and obtain the desired $L^\infty$-estimate.
		First, notice that from \eqref{nondegeneracy2}, we have $\phi \in L^{p_1}(\mathbb R^n)$, and so $\phi\in L^{p_1}(\mathbb R^n)$ for all $p\in [p_0,p_1]$.
		Second, we check whether $p_1\geqslant p_*$ or not. 
		In the affirmative case, we apply \eqref{nondegeneracy2} with $p=p_*-\varepsilon$ for $0<\varepsilon\ll1$ small enough to obtain that $\phi\in L^{p_1}(\mathbb R^n)$ for all $p\in [p_0,+\infty)$.
		In the negative case, we use \eqref{nondegeneracy2} with $p=p_1$, which gives us that $\phi\in L^{p_2}(\mathbb R^n)$ for all $p\in [p_0,p_2]$, where 
		$p_2=\frac{np_1}{n-2\sigma p_1}$.
		Third, we repeat the same process for this new exponent.
		
		More precisely, it is not hard to check that the bootstrap sequence $\{p_\ell\}_{\ell\in\mathbb N}\subset [p_0,+\infty)$ satisfies 
		\begin{equation*}
			p_{\ell+1}=\left(1+\frac{4\sigma}{n-6\sigma}\right)p_\ell \quad {\rm for \ all} \quad \ell\in\mathbb N.
		\end{equation*}
		Hence, $\lim_{\ell\rightarrow+\infty}p_{\ell}=+\infty$, which shows that the bootstrap technique terminates in a finite step.
		
		Now, let us fix some $p \gg 1$ large enough. using the same strategy as in \eqref{nondegeneracy3}, we find
		\begin{align}\label{nondegeneracy0}
			\int_{\mathbb{R}^n}{|x-y|^{2\sigma-n}} \frac{|\phi(y)|}{1+|y|^{4\sigma}} \ud y & \lesssim\left(\int_{\mathbb{R}^n} \frac{|x-y|^{(2\sigma-n) p^{\prime}}}{1+|y|^{4\sigma p_0^{\prime}}} \ud y\right)^{\frac{1}{p^{\prime}}}\|\phi\|_{L^p(\mathbb{R}^n)} \\\nonumber
			& \lesssim \frac{1}{1+|x|^{\frac{n\left(p^{\prime}-1\right)}{p^{\prime}}+2\sigma}}\lesssim 1 \quad {\rm for \ all} \quad x\in\mathbb R^n,
		\end{align}
		where $p^{\prime}=\frac{p-1}{p}$ is the conjugate Lebesgue exponent of $p$.
		Finally, from the last estimate combined with \eqref{nondegeneracy4}, we deduce that $\phi\in L^{\infty}(\mathbb{R}^n)$; this finishes the first case.
		
		\noindent{\bf Case 2:} $n=6 \sigma$.
		
		Here we observe that since for $n=6 \sigma$, it holds that 
		$p_0=p_*=3$, one has $[3,3)=\varnothing$; thus \eqref{nondegeneracy2} does not make sense for this case.
		However, we still have \eqref{nondegeneracy4}.
		In addition, since by Sobolev embedding, we know $\phi\in H^{\sigma}(\mathbb R^n)\hookrightarrow L^3(\mathbb R^n)$, which, as before, yields
		\begin{align*}
			\|\phi\|_{L^{p_1}(\mathbb{R}^n)} & \lesssim \left\|\frac{|\phi(x)|}{1+|x|^{4\sigma}} \ast \frac{1}{|x|^{4\sigma}}\right\|_{L^{p_1}(\mathbb{R}^n)}+\left\|\frac{1}{1+|x|^{2\sigma}}\right\|_{L^{p_1}(\mathbb{R}^n)} \\
			& \lesssim \left\|\frac{|\phi(x)|}{1+|x|^{4 \sigma}}\right\|_{L^{q_1}(\mathbb{R}^n)}+1\\
			&\lesssim \|\phi\|_{L^3(\mathbb{R}^n)}\left\|\frac{1}{1+|x|^{4\sigma}}\right\|_{L^{q_0}(\mathbb{R}^n)}+1,
		\end{align*}
		where $q_0 \in(3, +\infty),\zeta_1=\frac{3 q_0}{q_0+3} \in(\frac{3}{2}, 3)$, and $p_1=\frac{3 q_1}{3-q_1} \in(3, +\infty)$. 
		
		This means that $\phi \in L^{p}(\mathbb{R}^n)$ for all $p \geqslant 3$. 
		More precisely, by taking $q_0\gg1$, one can make $p\gg 1$ large enough.
		Finally, by the same argument in the last case, we have $\phi \in L^{\infty}(\mathbb{R}^n)$, which concludes the argument for the second case.

		\noindent{\bf Case 3:} $2\sigma<n<6\sigma$.
		
		In this case, using the Hardy--Littlewood--Sobolev inequality, it follows that 
		\begin{align*}
			\|\phi\|_{L^{p_1}(\mathbb{R}^n)} & \lesssim\left\|\frac{|\phi(x)|}{1+|x|^{4 \sigma}} \ast {|x|^{n-2 \sigma}}\right\|_{L^{p_1}(\mathbb{R}^n)} \\
			& \lesssim\left\|\frac{|\phi(x)|}{1+|x|^{4 \sigma}}\right\|_{L^{q_1}(\mathbb{R}^n)}\\
			&\lesssim\|\phi\|_{L^{p_0}(\mathbb{R}^n)}\left\|\frac{1}{1+|x|^{4\sigma}}\right\|_{L^{q_0}(\mathbb{R}^n)},
		\end{align*}
		where $p_0=\frac{2 n}{n-2 \sigma}=2_\sigma^*$, $q_0 \in(\frac{n}{2\sigma}, \frac{2 n}{6\sigma-n})$, $q_1=\frac{p_0 q_0}{q_0+p_0}$, and $p_1=\frac{n q_1}{n-2\sigma q_1} \in(p_0, +\infty)$. 
		This means that $\phi \in L^{p}(\mathbb{R}^n)$ for all $p \geqslant p_0$. From \eqref{nondegeneracy0} we conclude that $\phi \in L^{\infty}(\mathbb{R}^n)$, which finishes the proof of this case.
		
		The lemma is proved.
	\end{proof}

	\begin{acknowledgement}
		This paper was finished when the first-named author held a Post-doctoral position at the University of British Columbia, whose hospitality he would like to acknowledge.
	\end{acknowledgement}

\end{document}